\documentclass[11pt]{article}
\usepackage{amssymb}
\usepackage{amsmath}
\usepackage{amsthm}
\usepackage{graphicx}
\usepackage[margin=0.5in]{geometry}
\usepackage{bm}
\usepackage{amsthm}
\usepackage{amscd}
\usepackage{amsbsy}
\usepackage{amsmath}
\usepackage{amsfonts}
\usepackage{amssymb}
\usepackage{calligra}
\usepackage{xcolor}
\usepackage{float}
\usepackage[T1]{fontenc}
\usepackage{multirow}
\usepackage{mathrsfs}
\usepackage{mathtools}
\usepackage{graphicx}
\usepackage{epstopdf}
\usepackage{subfigure}
\usepackage{booktabs}
\usepackage{listings}
\usepackage{longtable}
\usepackage{latexsym}
\usepackage{arydshln}
\usepackage{enumerate}
\usepackage{epsfig}
\usepackage{cite}
\usepackage{verbatim}
\usepackage{overpic}
\usepackage{abstract}

\allowdisplaybreaks[4]

\date{}
\title{Planar  Semi-quasi Homogeneous  Polynomial differential systems with a given degree
\thanks{The second author is supported by the NSF of China (No. 11771101) and NSF of Guangdong Province (No.2015A030313669).}}
\author{Yuzhou Tian$^{1,2}$ and Haihua Liang$^{\ast,1}$\\
 1. School of Mathematics and Systems Science,\ Guangdong Polytechnic Normal\\ University,\ Guangzhou\ 510665,\ P.R.\ China\\
 2. School of Mathematics (Zhuhai),\ Sun Yat-sen University,\ Zhuhai\ 519082,\\ P.R.\ China\\
 e-mail: \ lianghhgdin@126.com}

\newtheorem {theorem*}{Theorem}
\newtheorem {theorem} {Theorem}

\newtheorem{algorithm}{Algorithm: Part}
\newtheorem{definition}{Definition}

\newtheorem{proposition}{Proposition}
\newtheorem{lemma}{Lemma}
\newtheorem{corollary}{Corollary}
\newtheorem{remark}{Remark}

\numberwithin{equation}{section}

\usepackage[colorlinks,
            CJKbookmarks=true,
            linkcolor=blue,
            anchorcolor=red,
            citecolor=blue
            ]{hyperref}
\begin{document}
\maketitle
\noindent {\bf Abstract} This paper  study the planar semi-quasi homogeneous
 polynomial differential systems (short for PSQHPDS), which can be regard as   a generalization of semi-homogeneous
 and of quasi-homogeneous systems.
By using the algebraic skills, several important  properties of  PSQHPDS
 are derived and are employed to establish an algorithm for
  obtaining all the explicit expressions of  PSQHPDS with a given degree.
 Afterward, we apply this algorithm to research the center problem of
  quadratic and cubic  PSQHPDS. It is proved that the quadratic one hasn't center,
  and,  that the  cubic one  has  center if and only if it can
  be written as   $\dot{x}=x^2-y^3, \dot{y}=x$ after a linear transformation of coordinate
   and a rescaling of time.
\smallskip

\noindent {\bf 2000 Math Subject Classification: } Primary 37C10.  Secondary 34C25.  Tertiary 68W99

\smallskip

\section{Introduction}

Since Hilbert D.  posed the 16th problem in 1900,  the qualitative   theory of planar differential systems have become   a  main  subject in the field of ordinary differential equation.  Many people have researched the qualitative behaviors of various planar polynomial  differential systems   including the homogeneous, semi-homogeneous  and quasi-homogeneous  polynomial differential systems,  and obtained a lot of valuable
 results.  In this paper,  we consider the  differential systems of the form
\begin{equation}\label{1} \dot{\mathop{x}}\,=P\left( x,y \right),\  \dot{\mathop{y}}\,=Q\left( x,y \right), \end{equation}
where $P$, $Q$ $\in $ $\mathbb{R}\left[ x,y \right]$, the ring of all polynomials in the variables $x$ and $y$ with coefficients in real number system $\mathbb{R}$.  We say that system \eqref{1} has degree $n$, if $n$ is the maximum of degree $P$ and $Q$.  Denoted by $X = \left(P, Q\right)$ the polynomial vector field associated to \eqref{1}.

The polynomial differential system \eqref{1} or the vector field $X$ is {\itshape quasi-homogeneous} if there exist ${{s}_{1}},{{s}_{2}},d\in {\mathbb{N}^{+}}$ such that for any arbitrary $  \alpha \in {\mathbb{{R}}^{+}}$,
\[P\left( {{\alpha }^{{{s}_{1}}}}x,{{\alpha }^{{{s}_{2}}}}y \right)={{\alpha }^{{{s}_{1}}-1+d}}P\left( x,y \right),\ Q\left( {{\alpha }^{{{s}_{1}}}}x,{{\alpha }^{{{s}_{2}}}}y \right)={{\alpha }^{{{s}_{2}}-1+d}}Q\left( x,y \right). \]
We call ${{s}_{1}}$ and ${{s}_{2}}$  the {\itshape weight exponents} of system \eqref{1} ,  and $d$ the {\itshape weight degree} with respect to the weight exponents ${{s}_{1}}$ and ${{s}_{2}}$.  In the particular case that
${{s}_{1}={{s}_{2}}}=1$,   system \eqref{1} is the classical {\itshape homogeneous } polynomial differential system of degree $d$.
As was done in \cite{1}, the  vector $w=\left( {{s}_{1}},{{s}_{2}},d \right)$ is called the weight vector of system \eqref{1}, and  a weight vector ${{w}_{m}}=\left( s_{1}^{*},s_{2}^{*},{{d}^{*}} \right)$ is called a {\itshape minimal weight vector} of   system  \eqref{1} if any other weight vector $w=\left( {{s}_{1}},{{s}_{2}},d \right)$ of system \eqref{1} satisfies ${{s}_{1}}\ge s_{1}^{*},{{s}_{2}}\ge s_{2}^{*},d\ge {{d}^{*}}$.

The polynomial differential system \eqref{1} or the vector field $X$ is {\itshape semi-quasi homogeneous} if there exist ${{s}_{1}},{{s}_{2}}, {{d}_{1}}, {{d}_{2}}\in {\mathbb{{N}}^{+}}$ and ${{d}_{1}}\ne {{d}_{2}}$ such that for any arbitrary  $  \alpha \in {\mathbb{R}^{+}}$,
\begin{equation}\label{2}P\left({{\alpha }^{{{s}_{1}}}}x, {{\alpha }^{{{s}_{2}}}}y \right)={{\alpha }^{{{s}_{1}}-1+{{d}_{1}}}}P\left( x, y \right), \ Q\left( {{\alpha }^{{{s}_{1}}}}x, {{\alpha }^{{{s}_{2}}}}y \right)={{\alpha }^{{{s}_{2}}-1+{{d}_{2}}}}Q\left( x, y \right).  \end{equation}
We endow ${{s}_{1}}$ and ${{s}_{2}}$  the {\itshape weight exponents} of system \eqref{1}.  We also call ${d_{1}}$ and ${d_{2}}$ the {\itshape weight degree} with respect to weight exponents ${s_{1}}$ and ${s_{2}}$, respectively.  As far as we know, this  concept of planar semi-quasi homogeneous polynomial
differential systems (or {\itshape PSQHPDS} for short)
 was firstly proposed by \cite{27}. It is worth mentioning that, although   many articles also use the concept of semi-quasi homogeneous
differential systems (such as  \cite{49} and the first chapter of \cite{50}), the meaning is very different from  the one in \cite{27}.   Obviously,     {\itshape PSQHPDS} \eqref{1} is a {\itshape semihomogeneous} polynomial differential system  once ${{s}_{1}}={{s}_{2}}=1$.

 We call, analogously to \cite{1},    $w=\left( {{s}_{1}}, {{s}_{2}}, {d_{1}, {d_{2}}} \right)$ the weight vector of {\itshape PSQHPDS} \eqref{1}.
  A weight vector ${{w}_{m}}=\left( s_{1}^{*}, s_{2}^{*}, {{d_{1}}^{*}}, {{d_{2}}^{*}} \right)$ is called  a {\itshape minimal weight vector} of   {\itshape PSQHPDS}  \eqref{1} if any other weight vector  $w=\left( {{s}_{1}}, {{s}_{2}}, {d_{1}, {d_{2}}} \right)$ of   \eqref{1} satisfies ${{s}_{1}}\ge s_{1}^{*}, {{s}_{2}}\ge s_{2}^{*}, {d_{1}}\ge {{{d}_{1}}^{*}}, {{d}_{2}}\ge{{d}_{2}}^{*}$. As we will see in Section \ref{se-3},   the minimal weight vector of a given  is  {\itshape PSQHPDS}  unique once $s_1\ne s_2$.

As we have mentioned, the  homogeneous and semihomogeneous   systems are respectively the specific kind
 of the quasi-homogeneous and   semi-quasi homogeneous systems. The qualitative behaviours of  the homogeneous and semi-homogeneous polynomial differential systems has   attracted the interest of numerous
 people due to the   specific  forms of these systems. In fact, many authors have investigated  these systems
 from different aspects, such as  the structural stability \cite{28,44}, the phase portraits\cite{29,30,41,42,43}, the centers problem \cite{33,34,35,36,37},  the limit cycles \cite{45}, the first integrals\cite{10,32}, the  canonical forms\cite{40}, the algebraic classification\cite{38,39}, and so on.

The quasi-homogeneous polynomial differential systems have also gained wide attention since the beginning of the 21st century.
The result obtained   mainly   include the  integrability\cite{2, 3, 4, 1, 6, 7, 8},  polynomial and rational first integrability\cite{9, 10}, normal forms\cite{2},  centers\cite{14, 15, 16, 17, 18, 19}  and limit cycles\cite{20, 21, 22, 23}. Recently,  the authors of  \cite{1}  establish an algorithm for obtaining all the quasi-homogeneous but non-homogeneous polynomial   systems with a given degree.  This algorithm is very  helpful  to study many  problems of the quasi-homogeneous but non-homogeneous polynomial   systems. Indeed, several  work has been done by means of this algorithm,  see\cite{1,24,17,25,26}. In these papers,  after obtaining all the expressions of  quasi-homogeneous system of degree $2-6$,   the authors   investigate the dynamic behaviour of them in an    explicit way.

The next natural step is to study the  {\itshape PSQHPDS}, this is because not only it is     a generalization of semi-homogeneous and quasi-homogeneous system, but also it has a lot of applications  in the others fields. For example, the  {\itshape PSQHPDS} $\dot{\mathop{x}}\,=y,\  \dot{\mathop{y}}\,=x^3+y^2$, with minimal weight vector  $w_m=(2,3,2,4)$, is a Li\'{e}nard equation of the second type, which arises   from Newtonian considerations when the mass of a particle is
position dependent \cite{53,54}.  In 2013, Zhao \cite{27} study the periodic orbits  of     {\itshape PSQHPDS}. He   gave sufficient conditions for the nonexistence and existence of periodic orbits, and  a lower bound for the maximum number of limit cycles of such system.  However, to the best of our knowledge, there are not other  articles  consider the dynamic problem  of {\itshape PSQHPDS} except \cite{27}.

Just as shown by \cite{1,24,17,25,26},   a commendable way to investigate the polynomial differential systems is to obtain the explicit expressions of these systems.
Therefore, the problem of how to obtain  easily  the explicit  expressions of {\itshape PSQHPDS} deserve investigation.
In this paper, inspired by \cite{1}, we will establish an algebraic  algorithm to obtain all the expression of  {\itshape PSQHPDS} with a given degree.
 By contrast with  the quasi-homogeneous systems,   the weight vector of  {\itshape PSQHPDS} has four components $s_1,s_2,d_1$ and $d_2$, which are usually  unknown, the complexity
arise distinctly.  To overcome this complexity and to avoid the onerous discussion, we will introduce a quantity $\lambda=(d_1-d_2)/(s_1-s_2)$ (when $s_1\neq s_2$) and call it the index of  {\itshape PSQHPDS} (see Definition \ref{de-1} and Remark \ref{re-12} for the uniqueness of  $\lambda$). Obviously,
  the case that $\lambda=0$ corresponds to the    quasi-homogeneous polynomial systems.

Noting that the expression of semihomogeneous system with a given degree  can be
obtained very directly and very easily. Thus we will assume in this paper that \eqref{1} is not a semihomogeneous polynomial system. Moreover,
if the weight exponents of {\itshape PSQHPDS} \eqref{1} fulfill  $s_1=s_2$, then  system  \eqref{1} become a  semihomogeneous system
(see Proposition \ref{co-8} of Section \ref{se-2}).  Thus we only need to consider the situation that     $s_1\neq s_2$.   We  can further
assume that $s_1>s_2$, because   in the other case we could interchange the
variables $x$ and $y$. Therefore,  we would like to make  the following convention    without loss of generality. \\
{\bf Convention}. We always assume in this paper, unless otherwise specified, that

 (i) $(P,Q)$ is not a semihomogeneous vector field;

(ii) $s_1>s_2$; and

(iii) $P(x,y)\not\equiv 0$, $Q(x,y)\not\equiv 0$.

The last  part of this convention is natural because system \eqref{1} will become a scalar differential equations provided  $P(x,y)$ or $Q(x,y)$
vanish identically.

The rest of this paper is organized as follows.
Section \ref{se-2} is devoted to  studying  the algebraic properties of {\itshape PSQHPDS} \eqref{1}. The concept of index of  \eqref{1} is   introduced in the end of this section.   In  Section \ref{se-3}  we divide {\itshape PSQHPDS} into two classes  to determine  the forms  of their  weight vectors   respectively. In Section \ref{se-4},    we construct an algebraic algorithms  which can be used to obtain all the {\itshape PSQHPDS} as long as the  degree of system is given. This algorithm   contains three parts   which depend on the number of monomials appearing in $P$ and $Q$, where $(P,Q)$ is the vector field associated to {\itshape PSQHPDS} \eqref{1}.        Section \ref{se-5} is devoted to illustrating the application of our algorithm. To this end,  we  calculate  all the {\itshape PSQHPDS} of degree $2$ and $3$ step by step. Finally,  in Section \ref{se-6}, by studying  the canonical forms for the {\itshape PSQHPDS} obtained in Section \ref{se-5},  we  give the criteria for the existence  of centers for {\itshape PSQHPDS}  of degree $2$ and $3$.
\section{The properties of   PSQHPDS }\label{se-2}

From   relation   \eqref{2}, we can obtain several important properties which play a significant role to get the algorithm
of  {\itshape PSQHPDS}. Some of these properties are the extensions of the ones given in \cite{1} and the others are new.
These properties allows us to  discover the  structure of {\itshape PSQHPDS}. By using the structure, we can
express {\itshape PSQHPDS}  in a simple way.

\begin{lemma}\label{le-1}
If ${{w}_{m}}=\left( s_{1}^{*}, s_{2}^{*}, d_{1}^{*}, d_{2}^{*} \right)$ is a minimal weight vector of system \eqref{1}, then all the
vectors of the form $\left( rs_{1}^{*}, rs_{2}^{*}, r( d_{1}^{*}-1 )+1, r( d_{2}^{*}-1)+1 \right)$ , where $r$ is a positive integer, are weight vectors of \eqref{1}.
\end{lemma}
\begin{proof}
 From the definition of the weight vector,  for any arbitrary $\beta \in {\mathbb{R}^{+}}$, we have
 \[P\left( {{\beta }^{s_{1}^{*}}}x, {{\beta }^{s_{2}^{*}}}y \right)={{\beta }^{s_{1}^{*}-1+d_{1}^{*}}}P\left( x, y \right), \]
 and
 \[Q\left( {{\beta }^{s_{1}^{*}}}x, {{\beta }^{s_{2}^{*}}}y \right)={{\beta }^{s_{2}^{*}-1+d_{2}^{*}}}Q\left( x, y \right). \]
 \par Let ${{t}_{1}}=rs_{1}^{*}$ and ${{t}_{2}}=rs_{2}^{*}$. We deduce that
 \[P\left( {{\beta }^{\frac{{{t}_{1}}}{r}}}x, {{\beta }^{\frac{{{t}_{2}}}{r}}}y \right)={{\beta }^{\frac{1}{r}\left[ {{t}_{1}}\text{+}r\left( d_{1}^{*}-1 \right)+1-1 \right]}}P\left( x, y \right). \]
 Setting $\alpha ={{\beta }^{\frac{1}{r}}}$, we obtain that
 \begin{equation}\label{3}
 P\left( {{\alpha }^{rs_{1}^{*}}}x, {{\alpha }^{rs_{2}^{*}}}y \right)={{\alpha }^{rs_{1}^{*}-1+r( d_{1}^{*}-1 )+1}}P\left( x, y \right).
 \end{equation}

 In a similar way, we can get that
 \begin{equation}\label{4}
 Q\left( {{\alpha }^{rs_{1}^{*}}}x, {{\alpha }^{rs_{2}^{*}}}y \right)={{\alpha }^{rs_{2}^{*}-1+r( d_{2}^{*}-1)+1}}Q\left( x, y \right).
 \end{equation}

  By \eqref{3}, \eqref{4} and the definition of   weight vector, the conclusion holds.
\end{proof}
\begin{remark}\label{re-1}
For the general semi-quasi homogeneous system, the  converse of Lemma \ref{le-1} is not true. As a counterexample, one can consider system $\dot{\mathop{x}}\, ={{a}_{1, 1}}xy$,  $\dot{\mathop{y}}\, ={{b}_{0, 1}}y$.
It is easy to check that ${{w}_{m}}=\left( 1, 1, 2, 1 \right)$ is the minimal weight vector of this system. However, there exist a weight vector   $w=\left( 3, 2, 3, 1 \right)$    of this system which   can't be written in  the
form $(r,r,r+1,1)$.
\end{remark}

 In order to observe the  properties of the coefficients of {\itshape PSQHPDS},
we write the polynomials $P$ and $Q$ of system \eqref{1} as the sum of their  homogeneous parts.
\begin{equation}\label{5}
P\left( x, y \right)=\sum\limits_{j=0}^{l}{{{P}_{j}}\left( x, y \right)},  \quad \text{where} \quad{{P}_{j}}\left( x, y \right)=\sum\limits_{i=0}^{j}{{{a}_{i, j-i}}}{{x}^{i}}{{y}^{j-i}},
\end{equation}
and
\begin{equation}\label{6}
Q\left( x, y \right)=\sum\limits_{j=0}^{m}{{{Q}_{j}}\left( x, y \right)},  \quad \text{where} \quad{{Q}_{j}}\left( x, y \right)=\sum\limits_{i=0}^{j}{{{b}_{i, j-i}}}{{x}^{i}}{{y}^{j-i}},
\end{equation}
where $l=\mbox{deg}(P)$, $m=\mbox{deg}(Q)$. Substituting \eqref{5} and \eqref{6} into \eqref{2}, we obtain that
\begin{equation}\label{7}
{{a}_{i, j-i}}{{\alpha }^{\left( i-1 \right){{s}_{1}}+\left( j-i \right)s{}_{2}-\left( {{d}_{1}}-1 \right)}}={{a}_{i, j-i}},
\end{equation}
and
\begin{equation}\label{8}
{{b}_{i, j-i}}{{\alpha }^{i{{s}_{1}}+\left( j-i-1 \right)s{}_{2}-\left( {{d}_{2}}-1 \right)}}={{b}_{i, j-i}}.
\end{equation}

\begin{proposition}\label{pr-1}
If   system \eqref{1} is a  {\itshape PSQHPDS} with weight vector $w=\left( s_{1}^{{}}, s_{2}^{{}}, d_{1}^{{}}, d_{2}^{{}} \right)$, then the following statements hold.
\begin{itemize}
\item[a)]${{P}_{0}}={{Q}_{0}}=0$.
\item[b)]If $d_{1}^{{}}>1$, $d_{2}^{{}}>1$, then ${{a}_{1, 0}}={{b}_{0, 1}}=0$.
\item[c)]If ${{s}_{0}}$ is the greatest common divisor of $s_{1}$ and $s_{2}$,  then ${{s}_{0}} \mid\left( {{d}_{i}}-1 \right)$ for $i=1, 2$.
 \end{itemize}
\end{proposition}
\begin{proof}
By equation \eqref{7} and \eqref{8}, statement $a)$ and $b)$ can be easily proved.

Let's prove  $c)$. If ${{s}_{0}} \mid\left( {{d}_{1}}-1 \right)$ is not fulfilled,  then equation $\left( i-1 \right){{s}_{1}}+\left( j-i \right){{s}_{2}}={{d}_{1}}-1$
has not solution of non-negative integer, where $i$ and $j$ are unknowns. By   \eqref{7}, we have ${{a}_{i, j-i}}=0$.
This implies that $P\left( x, y \right)\equiv 0$,  a contradiction. Hence ${{s}_{0}} \mid\left( {{d}_{1}}-1 \right)$.
Similarly, we can prove that ${{s}_{0}} \mid\left( {{d}_{2}}-1 \right)$.
\end{proof}
\begin{corollary}\label{co-5}
If ${{w}_{m}}=\left( s_{1}^{*}, s_{2}^{*}, d_{1}^{*}, d_{2}^{*} \right)$ is a minimal weight vector of \eqref{1}, then $s_{1}^{*}$, $s_{2}^{*}$ are coprime.
\end{corollary}
\begin{proof}
Let $s_{0}^{*}$ be a greatest common divisor of $s_{1}^{*}$ and $s_{2}^{*}$. Then there exist two positive integers
$\tilde{s}_{1}$ and $\tilde{s}_{2}$ such that $s_{1}^{*}=s_{0}^{*}\tilde{s}_{1}$ and $s_{2}^{*}=s_{0}^{*}\tilde{s}_{2}$.
From Proposition \ref{pr-1}, $s_{0}^{*}$ divides both $d_{1}^{*}-1$ and $d_{2}^{*}-1$, that is,
there exist two non-negative integers $\tilde{d}_{1}$ and $\tilde{d}_{2}$ such that $d_{1}^{*}-1=s_{0}^{*}\tilde{d}_{1}$ and $d_{2}^{*}-1=s_{0}^{*}\tilde{d}_{2}$.

From the definition of the weight vector,   for any arbitrary $\beta \in {\mathbb{R}^{+}}$, we have
 \[P\left( {{\beta }^{s_{1}^{*}}}x, {{\beta }^{s_{2}^{*}}}y \right)={{\beta }^{s_{1}^{*}-1+d_{1}^{*}}}P\left( x, y \right)={\beta}^{s_{0}^{*}(\tilde{s}_{1}-1+(\tilde{d}_{1}+1))}P(x, y), \]
 and
 \[Q\left( {{\beta }^{s_{1}^{*}}}x, {{\beta }^{s_{2}^{*}}}y \right)={{\beta }^{s_{2}^{*}-1+d_{2}^{*}}}Q\left( x, y \right)={\beta}^{s_{0}^{*}(\tilde{s}_{2}-1+(\tilde{d}_{2}+1))}Q(x, y). \]
Taking $\alpha={\beta}^{s_{0}^{*}}$, it turns out  from the above equalities that $w=(\tilde{s}_{1}, \tilde{s}_{2}, \tilde{d}_{1}+1, \tilde{d}_{2}+1)$ is a weight vector of \eqref{1}. If $s_{0}^{*}>1$, then $\tilde{s}_{1}<s_{1}^{*}$ and $\tilde{s}_{2}<s_{2}^{*}$. This   contradicts  with our hypothesis that $w_{m}$ is a minimal weight vector of \eqref{1}. The proof is finished.
\end{proof}

\begin{remark}\label{re-4}
For the general semi-quasi homogeneous system, the converse of Corollary \ref{co-5} is not true. For instance, the system of Remark \ref{re-1} has a weight vector $w=(3,2,3,1)$, where
$3$ and $2$ are coprime number, but $w$  is not a minimal weight vector.
\end{remark}

From Proposition \ref{pr-1}, we have the following   apparent conclusion.
\begin{corollary}\label{co-1}
If $m=l=1$, then    {\itshape PSQHPDS} \eqref{1} is a homogeneous linear differential system.
\end{corollary}

 Based on Corollary \ref{co-1}, in what follows, we can suppose that the degree of the system is $n\geq2$.

\begin{corollary}\label{co-8}
  If   system \eqref{1} is {\itshape PSQHPDS} with weight vector $w=\left( s_{1}^{{}}, s_{2}^{{}}, d_{1}^{{}}, d_{2}^{{}} \right)$, then it is a semihomogeneous differential system as long as $s_1=s_2$.
\end{corollary}

\begin{proof}
   Let $s_1=s_2=s_0$. Then, from the Proposition \ref{pr-1}, one can obtain that there exist two non-negative integers $\tilde{d}_{1}$ and $\tilde{d}_{2}$ such that $d_{1}-1=s_{0}\tilde{d}_{1}$ and $d_{2}-1=s_{0}\tilde{d}_{2}$. By definition of the weight vector, for any arbitrary $\beta\in\mathbb{R}^+$, we get
   $$P(\beta^{s_{0}}x,\beta^{s_{0}}y)=\beta^{s_0-1+d_1}P(x,y)=\beta^{s_0(1+\tilde{d}_1)}P(x,y),$$
   and
   $$Q(\beta^{s_0}x,\beta^{s_0}y)=\beta^{s_0-1+d_2}Q(x,y)=\beta^{s_0(1+\tilde{d}_2)}Q(x,y).$$
Letting $\alpha=\beta^{s_0}$, the above equations become
  $$P(\alpha x,\alpha y)=\alpha^{1+\tilde{d}_1}P(x,y),\ \mbox{and}\ \
  Q(\alpha x, \alpha y)=\alpha^{1+\tilde{d}_2}Q(x,y).$$

  Thus the result is proved.
\end{proof}
\begin{proposition}\label{pr-2}
Suppose that  ${{w}_{{}}}=\left( s_{1}^{{}}, s_{2}^{{}}, d_{1}^{{}}, d_{2}^{{}} \right)$ is a weight vector of
{\itshape PSQHPDS} \eqref{1}, then
\begin{itemize}
\item[a)] each homogeneous part of $P$ and $Q$ has at most one nonzero monomial;
\item[b)] there is a unique $q\in \left\{ 0, 1, \cdots , m \right\}$ such that ${{b}_{q, m-q}}\ne 0$ and
a unique $p\in \left\{ 0, 1, \cdots , l \right\}$ such that ${{a}_{p, l-p}}\ne 0$.\end{itemize}
\end{proposition}
\begin{proof} Recall that $Q_j$ is the homogeneous part of $Q$ of degree $j$. Suppose that   there exist
 $q_{1}, q_{2}\in  \{ 1, \cdots , j  \}$ ($j\geq 1$) such that ${b}_{q_1, j-q_1}{b}_{q_2, j-q_2}\ne 0$. From
   \eqref{8}  we find
\begin{equation*}
{{q}_{1}}{{s}_{1}}+\left( j-{{q}_{1}}-1 \right){{s}_{2}}={{d}_{2}}-1,
\end{equation*}
and
\begin{equation*}
{{q}_{2}}{{s}_{1}}+\left( j-{{q}_{2}}-1 \right){{s}_{2}}={{d}_{2}}-1.
\end{equation*}
Thus
\[( q_1-q_2)( {{s}_{1}}-{{s}_{2}})=0. \]
This implies that   $q_1=q_2$. Therefore,  $Q_j$ can have at most one nonzero monomial.
  Similarly, each  homogeneous part of $P$ can have at most one nonzero monomial. Thus statement $a)$ is true.

  Conclusion $b)$ follows from     conclusion  $a)$ directly because $l$ and $m$ are respectively the degree of
polynomial $P$ and $Q$.
\end{proof}

The next proposition will show the relationship between the nonzero monomials of ${{P}_{j}}$ and ${{Q}_{j}}$.
\begin{proposition}\label{pr-3}Suppose that  ${{w}_{{}}}=\left( s_{1}^{{}}, s_{2}^{{}}, d_{1}^{{}}, d_{2}^{{}} \right)$ is a weight vector of
{\itshape PSQHPDS} \eqref{1},  then for arbitrary $j\in \{ 1, \cdots , \min\{ l, m \} \}$
and $p, q\in \{ 1, \cdots, j \}$, we have the following conclusion.
\begin{itemize}
\item[a)]If ${{b}_{q, j-q}}\ne 0$, then ${{a}_{i, j-i}}=0$ provided  $i\ne q+1+\lambda $ and $i\le j$.
\item[b)]If ${{a}_{p, j-p}}\ne 0$, then ${{b}_{i, j-i}}=0$ provided $i\ne p-1-\lambda $ and $i\le j$.
\end{itemize}
Here $\lambda =(d_1-d_2)/(s_1-s_2)$.
\end{proposition}
\begin{proof}
Suppose that ${{b}_{q, j-q}}\ne 0$. If ${{a}_{i, j-i}}\neq0$,   then by applying \eqref{7} and \eqref{8},
we get that $\left( q-i+1 \right)\left( {{s}_{1}}-{{s}_{2}} \right)={{d}_{2}}-{{d}_{1}}$.
Since ${{s}_{1}}>{{s}_{2}}$,   $i=q+1+(d_1-d_2)/(s_1-s_2)=q+1+\lambda $, a contradiction.
  Therefore,  statement $a)$ holds.  By a similar way,  one can prove that statement $b)$ also holds.
\end{proof}

From Proposition \ref{pr-2} and Proposition \ref{pr-3}, we can grasp the rough  structure of  {\itshape PSQHPDS} \eqref{1}.
  For convenience,  denote the homogeneous parts of $P$ by $P_{n-t}$,  $t=0, 1, . . . , n$. According to Proposition \ref{pr-2},  if $P_{n-t}\not\equiv 0$, then there exist a unique nonnegative integer $i^{*}$ and $i^{*}\leq n-t$ such that the coefficient
  ${{a}_{{{i}^{*}}, n-t-{{i}^{*}}}}\ne 0$. Similarly,  if $Q_{n-t}\not\equiv0$, then  there exist a unique nonnegative integer $j^{*}$ and $j^{*}\leq n-t$
  such that the coefficient ${{b}_{{{j}^{*}}, n-t-{{j}^{*}}}}\ne 0$.
Let $X_{n-t}$ denote the homogeneous part of $X$ of degree $n-t$, where $X=(P, Q)$ is the vector field associated to  {\itshape PSQHPDS} \eqref{1}.
 In other words, $X_{n-t}=(P_{n-t}, Q_{n-t})$.  On account of the above discussion,  we obtain that $X_{n-t}$ is one of the following expressions:
\begin{itemize}
\item[$a)$]${{X}_{n-t}}\equiv 0$.
\item[$b)$]If $X_{n-t}\not\equiv0$, then $X_{n-t}$ is one of the following
\begin{enumerate}[$b. 1)$]
\item${{X}_{n-t}}=\left( {{a}_{{{i}^{*}}, n-t-{{i}^{*}}}}{{x}^{{{i}^{*}}}}{{y}^{n-t-{{i}^{*}}}}, 0 \right)$,
\item${{X}_{n-t}}=\left( 0, {{b}_{{{j}^{*}}, n-t-{{j}^{*}}}}{{x}^{{{j}^{*}}}}{{y}^{n-t-{{j}^{*}}}} \right)$,
\item${{X}_{n-t}}=\left({{a}_{{{i}^{*}}, n-t-{{i}^{*}}}}{{x}^{{{i}^{*}}}}{{y}^{n-t-{{i}^{*}}}}, {{b}_{{{j}^{*}}, n-t-{{j}^{*}}}}{{x}^{{{j}^{*}}}}{{y}^{n-t-{{j}^{*}}}} \right)$.
\end{enumerate}
\end{itemize}
\par Now define the following vector field:
\begin{align*}
  &A_{n-t}=\begin{cases}
&( 0, 0 )\quad \text{if the first component of} \;{{X}_{n-t}}\; \text{is zero},\\
&\left( {{a}_{{{i}^{*}}, n-t-{{i}^{*}}}}{{x}^{{{i}^{*}}}}{{y}^{n-t-{{i}^{*}}}}, 0 \right) \quad \text{if the first component of} \;{{X}_{n-t}}\; \text{is nonzero};
  \end{cases}\\
  &B_{n-t}=\begin{cases}
    &\left( 0, 0 \right)\quad \text{if the second component of} \;{{X}_{n-t}}\; \text{is zero},   \\
   &\left( 0, {{b}_{{{j}^{*}}, n-t-{{j}^{*}}}}{{x}^{{{j}^{*}}}}{{y}^{n-t-{{j}^{*}}}} \right)\quad \text{if the second component of} \;{{X}_{n-t}}\; \text{is nonzero}.
  \end{cases}
\end{align*}
From the definition of $A_{n-t}$ and $B_{n-t}$, $X_{n-t}$ can be represented by ${{X}_{n-t}}\text{=}{{A}_{n-t}}\text{+}{{B}_{n-t}}$.  By Proposition \ref{pr-3} we know that if ${{A}_{n-t}}\not\equiv0$ , ${{B}_{n-t}}\not\equiv0$, then $\lambda \in \mathbb{Z}$ and  ${{j}^{*}}={{i}^{*}}-1-\lambda $.
Thus, we can get the following conclusion, as a corollary of Proposition \ref{pr-3}, easily.
\begin{corollary}\label{co-2}
Suppose that  system \eqref{1} is a  {\itshape PSQHPDS} with weight vector $w=\left( s_{1}^{{}}, s_{2}, d_{1}, d_{2} \right)$ and that
  $\lambda=(d_1-d_2)/(s_1-s_2) \in \mathbb{Q}\backslash \mathbb{Z}$, then we have
\begin{itemize}
\item[a)] if ${{X}_{n-t}}={{A}_{n-t}}+{{B}_{n-t}}\not\equiv0$,
then one of ${{A}_{n-t}}$ and $B_{n-t}$ is zero;
\item[b)]  $l\ne m$ and  the degree of system \eqref{1} is $n\ge 2$.
\end{itemize}
\end{corollary}
\begin{proof}
The conclusion $a)$ is obvious. Let's prove
$b)$.   Suppose that $l=m$, i.e.,  $n=l=m$. Then there exist two values $p$ and $q$ such that ${{a}_{p, n-p}}\ne 0$ and $b_{q, n-q}\ne 0$.  This implies
that ${A}_{l}={{A}_{n}}\not\equiv0$ and ${B}_{m}={{B}_{n}}\not\equiv0$, in contraction with conclusion $a)$.
The proof is finished.
\end{proof}

  Next,  let's   assume that both of ${{a}_{i, j-i}}$ and ${{b}_{i, j-i}}$ are nonzero. From \eqref{7} and \eqref{8}, we have
\begin{equation}\label{14}
  \begin{cases}
    (i-1)s_{1}+(j-i)s_{2}=d_{1}-1, \\
is_{1}+(j-i-1)s_{2}=d_{2}-1. \\
  \end{cases}
\end{equation}
Then ${{d}_{1}}-\lambda {{s}_{1}}-1={{d}_{2}}-\lambda {{s}_{2}}-1$,
 where $\lambda =(d_1-d_2)/(s_1-s_2)$. If we write $T={{d}_{1}}-\lambda {{s}_{1}}-1={{d}_{2}}-\lambda {{s}_{2}}-1$,
  then \eqref{14} is changed  into
\begin{equation}\label{15}
\begin{cases}
(i-1-\lambda)s_{1}+(j-i)s_{2}=T, \\
is_{1}+(j-i-1-\lambda)s_{2}=T. \\
\end{cases}
\end{equation}

\begin{corollary}\label{co-4} Suppose that  system \eqref{1} is a  {\itshape PSQHPDS} with weight vector $w=\left( s_{1}^{{}}, s_{2}, d_{1}, d_{2} \right)$ and that
  $\lambda=(d_1-d_2)/(s_1-s_2) \in   \mathbb{Z}$.
Then $X_{n-t}\not\equiv0$ $(t<n)$ if and only if there is a unique $i\in\mathbb{Z}$ and $min\{0, 1+\lambda\}\leq i\leq max\{n-t, n-t+1+\lambda\}$ such that
\begin{equation*}
a_{i, n-t-i}^{2}+b_{i-1-\lambda , n-t-i+1+\lambda }^{2}\ne 0\quad and \quad\left( i-1-\lambda  \right){{s}_{1}}+\left( n-t-i \right){{s}_{2}}=T.
\end{equation*}
Here, $T=d_1-\lambda s_1-1$, and we define
$${{a}_{i, n-t-i}}={{b}_{i-1-\lambda , n-t-i+1+\lambda }}=0$$
whenever $i\left( n-t-i \right)<0$ or $\left( i-1-\lambda  \right)\left( n-t-i+1+\lambda  \right)<0$.
\end{corollary}
\begin{proof}
The sufficiency is clear.  Thus we only need to  prove the necessity.

Since ${{X}_{n-t}}\not\equiv0$, we have  ${{A}_{n-t}}\not\equiv0$ or ${{B}_{n-t}}\not\equiv0$.  We will divide the proof into two cases.

$Case \;1. $ If ${{A}_{n-t}}\not\equiv0$, then there is a unique $i\in \left\{ 0, \cdots , n-t \right\}$ such that ${{a}_{i, n-t-i}}\ne 0$.  By the first equation of \eqref{15}, we have
\begin{equation}\label{16} \left( i-1-\lambda  \right){{s}_{1}}+\left( n-t-i \right){{s}_{2}}=T.
\end{equation}

 $Case \;2. $
 If ${{B}_{n-t}}\not\equiv0$, then there is a unique ${{i}_{1}}\in \left\{ 0, \cdots , n-t \right\}$ such that ${{b}_{{{i}_{1}}, n-t-{{i}_{1}}}}\ne 0$.  By the second equation of
 \eqref{15}, we have
 \begin{equation}\label{17}
 {{i}_{1}}{{s}_{1}}+\left( n-t-{{i}_{1}}-1-\lambda  \right){{s}_{2}}=T.
 \end{equation}
 Let ${{i}_{1}}=i-1-\lambda $,  then $i\in \left\{ 1+\lambda , \cdots , n-t+1+\lambda  \right\}$ and
  $$b_{i_1,n-t-i_1}={{b}_{i-1-\lambda, n-t-i+1+\lambda }}\ne 0,$$
  and \eqref{17} is changed into
  \[\left( i-1-\lambda  \right){{s}_{1}}+\left( n-t-i \right){{s}_{2}}=T. \]

  Therefore, the conclusion is true in both cases.
\end{proof}

Form Corollary \ref{co-2} and Corollary \ref{co-4}, we can see that  {\itshape PSQHPDS}  has evident different properties
  between $\lambda \in \mathbb{Q}\backslash \mathbb{Z}$ and $\lambda \in \mathbb{Z}$.  Thus we would like to introduce   the notion of
  index of the  {\itshape PSQHPDS}.

\begin{definition}\label{de-1}\textbf {\emph{(Index)}}
  \emph{If system \eqref{1} is a {\itshape PSQHPDS} with weight vector $w=( s_{1}^{{}}, s_{2}, d_{1}, d_{2} )$,
  then we will say that $\lambda =(d_1-d_2)/(s_1-s_2)$ is an} index \emph{of system \eqref{1} or the vector field $X$ associated to  \eqref{1}}
  (see Remark \ref{re-12} for the uniqueness of the index of any given {\itshape PSQHPDS}).
\end{definition}
Note that when index $\lambda=0$,   system \eqref{1} becomes a quasi-homogeneous polynomial differential system.

\section{Weight vector of PSQHPDS}\label{se-3}
The main purpose of this section is to obtain the weight vector and the minimal weight vector of {\itshape PSQHPDS}
(remember that we have excluded the   semihomogeneous polynomial system)   of degree $n\geq 2$.

Firstly,  from  Subsection 2, we know that a {\itshape PSQHPDS} \eqref{1} of degree $n\geq2$ can be represented by
$$X=\sum\limits_{t=0}^{n-1}{{{X}_{n-t}}}=\sum\limits_{t=0}^{n-1}{{{A}_{n-t}}}\text{+}\sum\limits_{t=0}^{n-1}{{{B}_{n-t}}}.$$
Moreover, there exists at least an integer $t$ $(0\leq t\leq n-1)$ such that ${{A}_{n-t}}\not\equiv0$ and at least an integer $\tilde{t}$ $(0\leq \tilde{t}\leq n-1)$
such that ${{B}_{n-\tilde{t}}}\not\equiv0$,  and $t\cdot \tilde{t}=0$. We define the following sets
\[A=\{A_{n-t}|A_{n-t}\not\equiv0, t\in\mathbb{N}, t< n\}\quad \text{and} \quad B=\{B_{n-t}|B_{n-t}\not\equiv0, t\in\mathbb{N}, t< n\}. \]
Denote the cardinal of the set $\mathcal{X}$ by $|\mathcal{X}|$.  Obviously,   $|A|\geq1$ and $|B|\geq1$.
Noting that {\itshape PSQHPDS}  \eqref{1}  is a  semihomogeneous system  provided $|A|=1$ and $|B|=1$.  Thus in what follows
it suffices to consider the cases that $|A|\geq 2$ or $|B|\geq 2$. We will divide the   discussion   into two cases:
\begin{itemize}
\item[$Case\;1. $]$|A|\geq2$ and $|B|\geq1$, that is, $X$ has at least two monomials in the first component;
\item[$Case\;2. $]$|A|=1$ and $|B|\geq2$, that is, $X$ contains only one monomial in the
first component and has at least two monomials in the second component.
\end{itemize}
\subsection{Weight vector for the case $|A|\geq 2$ and $|B|\geq1$}
\begin{proposition}\label{pr-4}
 Suppose that   system \eqref{1} is a  {\itshape PSQHPDS}   of degree $n\ge 2$ with weight vector
 $w=\left( s_{1}, s_{2}, d_{1}, d_{2} \right)$, and  that $|A|\geq2$, $|B|\geq1$. Let $t,\tilde{t}\in\{0, . . . , n-1\}$  and
 $\triangle t\in\{1, . . . , n-t-1\}$, such that $A_{n-t}A_{n-t-\triangle t}\neq0$ and $B_{n-\tilde{t}}\not\equiv0$.
 The coefficients of the  non-zero component of vector field $A_{n-t}$, $A_{n-t-\triangle t}$ and $B_{n-\tilde{t}}$ are respectively $a_{p, n-t-p}$, $a_{p_{1}, n-t-\triangle t-p_{1}}$
 and $b_{q, n-\tilde{t}-q}$, where $p\in\{0, . . . , n-t\}$, $p_{1}\in\{0, . . , n-t-\triangle t\}$ and $q\in\{0, . . . , n-\tilde{t}\}$.
 Then the following statements hold.
 \begin{itemize}
 \item[a)]$k:=p_{1}-p\geq1$ and $k\leq n-t-\triangle t-p$;
 \item[b)]If $d_{1}=1$, then
 \begin{itemize}
   \item[b.1)]The only coefficient  of the  polynomial $P$ that can be different from zero are $a_{0, n-t}$ and $a_{1, 0}$.
   \item[b.2)]$t\in\{0, 1, . . . , n-2\}$, $s_{1}=(n-t)s_{2}$ and $d_{2}=\big((n-t-1)q+n-\tilde{t}-1\big)s_{2}+1$.
   \item[b.3)]the minimal weight vector of system \eqref{1} is
   $$w_{m}=(n-t, 1, 1, (n-t-1)q+n-\tilde{t}\;).$$
 \end{itemize}
 \item[c)]If $d_{1}>1$, then
 \begin{itemize}
   \item[c.1)]$s_{1}=\frac{(\triangle t+k)(d_{1}-1)}{D}$, $s_{2}=\frac{k(d_{1}-1)}{D}$ and $d_{2}=\frac{\big((n-\tilde{t}-1)k+q\triangle t\big)(d_{1}-1)}{D}+1$,
    where $D=(n-t-1)k+(p-1)\triangle t>0$;
   \item[c.2)]the minimal weight vector of system \eqref{1} is
   $$w_{m}=\bigg(\frac{k+\triangle t}{s},\frac{k}{s},\frac{D}{s}+1,\frac{(n-\tilde{t}-1)+q\triangle t}{s}+1\bigg),$$
   where $s$ is the greatest common divisor of $\triangle t$ and $k$.
 \end{itemize}
 \end{itemize}
\end{proposition}
\begin{proof}
 $ a)$ If $A_{n-t}A_{n-t-\triangle t}\neq0$  ($\triangle t\in\{1, . . . , n-t-1\}$), then $A_{n-t}\not\equiv0$ and $A_{n-t-\triangle t}\not\equiv0$. Applying Proposition \ref{pr-2}, we know that there
 exist a unique $p\in \left\{ 0, \cdots , n-t \right\}$ such that ${{a}_{p, n-t-p}}\ne 0$ and a unique $p_{1}\in \left\{ 0, \cdots , n-t-\triangle t \right\}$
 such that ${{a}_{p_{1}, n-t-\triangle t-p_{1}}}\ne 0$. Similarly, if $B_{n-\tilde{t}}\not\equiv0$,  then there exists
 a unique $q\in \left\{ 0, \cdots , n-\tilde{t} \right\}$ such that ${{b}_{q, n-\tilde{t}-q}}\ne 0$. From \eqref{7} and \eqref{8}, we have
 \begin{equation}\label{23}
 \begin{cases}
 (p-1)s_{1}+(n-t-p)s_{2}=d_{1}-1, \\
 (p_{1}-1)s_{1}+(n-t-\triangle t-p_{1})s_{2}=d_{1}-1,
 \end{cases}
 \end{equation}
 and
 \begin{equation}\label{24}
 qs_{1}+(n-\tilde{t}-q-1)s_{2}=d_2-1.
 \end{equation}
Using \eqref{23} it is be clear that $(p_{1}-p)(s_{1}-s_{2})=s_{2}\triangle t$. Since $s_{1}-s_{2}$, $s_{2}$ and $\triangle t$ are positive numbers,
we get that  $p_{1}>p$.  Further, in view of $p_1\leq n-t-\triangle t$, it follows that $k=p_1-p\leq n-t-\triangle t-p$, and thus statement $a)$ holds.

Substituting ${{p}_{1}}=p+k$ into \eqref{23} gives:
\begin{equation}\label{25}
\begin{cases}
(p-1)s_{1}+(n-t-p)s_{2}=d_{1}-1, \\
(p+k-1)s_{1}+(n-t-\triangle t-p-k)s_{2}=d_{1}-1.
\end{cases}
\end{equation}

$b)$ If $d_{1}=1$, then
we can regard \eqref{25} as  homogeneous linear equations   with unknowns $s_{1}$ and $s_{2}$. In order to obtain the positive solution
it is necessary  that
\begin{equation}\label{27}
\begin{vmatrix}
p-1&n-t-p\\
p+k-1&n-t-\triangle t-p-k
\end{vmatrix}=0,
\end{equation}
that is,
\begin{equation}\label{28}
(n-t-1)k+(p-1)\triangle t=0.
\end{equation}
Since $n-t-1\geq0$, $k\geq1$ and $\Delta t\geq1$, \eqref{28} means that $p-1\leq0$, i.e.,
 $p=0$ or $p=1$. If $p=1$, then $(n-t-1)k=0$, that is, $t=n-1$. This in contraction with $\triangle t\in\{1, . . . , n-t-1\}$.
 Thus $p$ must be zero, and hence $(n-t-1)k=\triangle t$.  Since $\triangle t\in\{1, . . . , n-t-1\}$,
 we obtain that $k=1$ and $n-t-\triangle t=1$. This implies that the only coefficient
  of    polynomial $P$ that can be different from zero are $a_{0, n-t}$ and $a_{1, 0}$. Hence conclusion  $b. 1)$ is confirmed.

Substituting $p=0$  into the first equation of \eqref{25} (with $d_1=1$) gives:
\begin{equation}\label{29}
 (n-t)s_{2}=s_{1}.
\end{equation}
By $s_{1}>s_{2}$ and \eqref{29}, we get $t\in\{0, 1, . . . , n-2\}$. Substituting \eqref{29} into   equation \eqref{24}, it gives:
\begin{equation}\label{30}
\big((n-t-1)q+n-\tilde{t}-1\big)s_{2}=d_{2}-1.
\end{equation}
Equation \eqref{30} implies that $d_{2}=\big((n-t-1)q+n-\tilde{t}-1\big)s_{2}+1$. Hence, the weight vector is
$$w=\big((n-t)s_{2}, s_{2}, 1, \big((n-t-1)q+n-\tilde{t}-1\big)s_{2}+1\big).$$
We get conclusion  $b. 2)$.

From the definition of minimal weight vector, taking $s_{2}=1$,  we obtain that
$${{w}_{m}}=\big( n-t, 1, 1, (n-t-1)q+n-\tilde{t}\;\big).$$
Therefore statement $b. 3)$ is proved.

$c)$ If $d_{1}>1$, we can also regard \eqref{25} as linear equations with unknowns $s_{1}$ and $s_{2}$.
 This linear equations has a unique solution if and only if its coefficient matrix is invertible,
  that is, $D=(n-t-1)k+(p-1)\triangle t\neq0$.  We claim that $D>0$.

Indeed, since $\triangle t\in\{1, . . . , n-t-1\}$, we get $t\in\{0, 1, . . . , n-2\}$.
If $p>1$, then by $n-1\geq t$, it follows that $D>0$. If $p=1$, then by $t\leq n-2$ we obtain $D=(n-t-1)k\geq k>0$. If $p=0$, we assume to the contrary that $D\leq0$.
Since $\triangle t\in\{1, . . . , n-t-1\}$ and $t\in\{0, 1, . . . , n-2\}$, it follows
 that $(n-t-1)k\leq\triangle t\leq n-t-1$. This means that $k=1$
and $\triangle t=n-t-1$. Then,  the second equation of \eqref{25}  becomes $0=d_{1}-1$ in contraction with $d_{1}>1$. So the claim is confirmed.

Thus,  the solution of   system \eqref{25} with unknowns $s_{1}$ and $s_{2}$ is
\begin{equation}\label{31}
  \begin{cases}
    s_{1}=\frac{(\triangle t+k)(d_{1}-1)}{D}, \\
    s_{2}=\frac{k(d_{1}-1)}{D}.
  \end{cases}
\end{equation}
Substituting \eqref{31} into \eqref{24} gives:
\begin{equation}\label{32}
  d_{2}=\frac{\big((n-\tilde{t}-1)k+q\triangle t\big)(d_{1}-1)}{D}+1.
\end{equation}
This completes the proof of statement $c.1)$.

$c.2)$ Suppose that $s$ is the greatest common divisor of $\triangle t$ and $k$. Then,  $D$ can be represented by $D=su$, where $u$ is a non-negative integer.
From the number theory, we get that there exist two integers $v$ and $h$ such that $v\triangle t/s+kh/s=1$. Hence
$$\frac{\triangle t\left( {{d}_{1}}-1 \right)}{su}v+\frac{k\left( {{d}_{1}}-1 \right)}{su}h=\frac{{{d}_{1}}-1}{u},$$
that is, $$\left( {{s}_{1}}-{{s}_{2}} \right)v+{{s}_{2}}h=\frac{{{d}_{1}}-1}{u}.$$
 Moreover, we can write $${{d}_{1}}-1=ru,\ {{s}_{1}}=\frac{r\left( \triangle t+k \right)}{s},\ {{s}_{2}}=\frac{rk}{s}$$ and
 $$ {{d}_{2}}=\frac{\big(( n-\tilde{t}-1)k+q\triangle t \big)r}{s}+1,$$
  where $r$ is non-negative integer. Clearly, all of  $s_{1}^{{}}, s_{2}^{{}} ,  d_{1}^{{}}$ and $d_{2}^{{}}$ are increasing with $r$. From the definition of
  minimal weight vector, taking $r=1$, we complete the proof of statement $c. 2)$. 	
\end{proof}
\begin{remark}\label{re-3}
From  the definition of index, we obtain that
\begin{enumerate}[(i)]
  \item if $d_{1}=1$, then $\lambda=-q-\frac{n-\tilde{t}-1}{n-t-1}$;
  \item if $d_{1}>1$, then $\lambda=p-q-1-\frac{(t-\tilde{t})k}{\triangle t}$.
\end{enumerate}
\end{remark}
\subsection{Weight vector for the case $|A|=1$ and $|B|\geq 2$}
\begin{proposition}\label{pr-8}
Suppose that    system \eqref{1} is a  {\itshape PSQHPDS}  of degree $n\ge 2$ with
weight vector $w=\left( s_{1}, s_{2}, d_{1}, d_{2} \right)$, and that   $|A|=1$, $|B|\geq2$.
Let $\tilde{t},t\in\{0, . . . , n-1\}$ and $\triangle \tilde{t}\in\{1, . . . , n-\tilde{t}-1\}$, such that $B_{n-\tilde{t}}B_{n-\tilde{t}-\triangle \tilde{t}}\neq0$
and $A_{n-t}\not\equiv0$. The coefficients of the non-zero component of vector field $B_{n-\tilde{t}}$, $B_{n-\tilde{t}-\triangle \tilde{t}}$ and $A_{n-t}$
are respectively $b_{q, n-\tilde{t}-q}$, $b_{q_{1}, n-\tilde{t}-\triangle \tilde{t}-q_{1}}$ and $a_{p, n-t-p}$, where $q\in\{0, . . . , n-\tilde{t}\}$, $q_{1}\in\{0, . . , n-\tilde{t}-\triangle \tilde{t}\}$
and $p\in\{0, . . . , n-t\}$. Then the following statements hold.
\begin{itemize}
  \item[a)]$k:={{q}_{1}}-q\ge 1$ and $k\le n-\tilde{t}-\triangle\tilde{t}-q$;
  \item[b)]$d_{2}>1$;
  \item[c)]$s_{1}=\frac{(k+\triangle\tilde{t})(d_{2}-1)}{D}$, $s_{2}=\frac{k(d_{2}-1)}{D}$ and $d_{1}=\frac{\big((n-t-1)k+(p-1)\triangle\tilde{t}\;\big)(d_{2}-1)}{D}+1$, where $D=(n-\tilde{t}-1)k+\triangle\tilde{t}q>0$.
      Moreover, if $p=0$, then $(n-t-1)k\geq\triangle\tilde{t}$.
  \item[d)]the minimal weight vector of system \eqref{1} is
  $$w_{m}=\bigg(\frac{k+\triangle\tilde{t}}{s}, \frac{k}{s}, \frac{(n-t-1)k+(p-1)\triangle\tilde{t}}{s}+1, \frac{D}{s}+1\bigg),$$
   where $s$ is the greatest common divisor of $\triangle\tilde{t}$ and $k$.
\end{itemize}
\end{proposition}
\begin{proof}
$a)$ If $B_{n-\tilde{t}}B_{n-\tilde{t}-\triangle\tilde{t}}\neq0$, where $\triangle \tilde{t}\in\{1, . . . , n-\tilde{t}-1\}$,  then $B_{n-\tilde{t}}\not\equiv0$ and $B_{n-\tilde{t}-\triangle\tilde{t}}\not\equiv0$.
Using  Proposition \ref{pr-2}, we have that there exists a unique $q\in \left\{ 0, \cdots , n-\tilde{t} \right\}$ such that ${{b}_{q, n-\tilde{t}-q}}\ne 0$
and a unique $q_{1}\in \left\{ 0, \cdots , n-\tilde{t}-\triangle \tilde{t} \right\}$
such that ${{b}_{q_{1}, n-\tilde{t}-\triangle \tilde{t}-q_{1}}}\ne 0$.
 Similarly, if $A_{n-t}\not\equiv0$,  then there exists a unique $p\in \left\{ 0, \cdots , n-t \right\}$
such that ${{a}_{p, n-t-p}}\ne 0$. From \eqref{7} and \eqref{8}, we have
\begin{equation}\label{39}
\begin{cases}
qs_{1}+(n-\tilde{t}-q-1)s_{2}=d_{2}-1, \\
q_{1}s_{1}+(n-\tilde{t}-\triangle\tilde{t}-q_{1}-1)s_{2}=d_{2}-1,
\end{cases}
\end{equation}
and
\begin{equation}\label{40}
(p-1)s_{1}+(n-t-p)s_{2}=d_{1}-1.
\end{equation}
Using \eqref{39} it is   clear that $(q_{1}-q)(s_{1}-s_{2})=s_{2}\triangle \tilde{t}$. Since $s_{1}-s_{2}$, $s_{2}$ and $\triangle \tilde{t}$ are positive numbers, we get that  $q_{1}>q$, and statement $a)$ holds.

 Substituting ${{q}_{1}}=q+k$ into \eqref{39} gives:
\begin{equation}\label{41}
\begin{cases}
qs_{1}+(n-\tilde{t}-q-1)s_{2}=d_{2}-1, \\
(q+k)s_{1}+(n-\tilde{t}-\triangle\tilde{t}-q-k-1)s_{2}=d_{2}-1.
\end{cases}
\end{equation}

Assume that $d_{2}=1$. Then
the necessary condition for \eqref{41} to have positive solution is
\begin{equation}\label{43}
\begin{vmatrix}
  q&n-\tilde{t}-q-1\\
  q+k&n-\tilde{t}-\triangle\tilde{t}-q-k-1
\end{vmatrix}=0,
\end{equation}
that is,
\begin{equation}\label{44}
(n-\tilde{t}-1)k+\triangle\tilde{t}q=0.
\end{equation}
Since $n-\tilde{t}-1\geq0$, $k\geq1$ and $\Delta \tilde{t}\geq1$, \eqref{44} implies that $q=0$ and $n-\tilde{t}-1=0$, which   contradicts
 with $1\leq\triangle\tilde{t}\leq n-\tilde{t}-1$. Thus we have  confirmed   statement $b)$.

$c)$ Since $d_{2}\neq1$, we can also view \eqref{41} as linear equations with unknowns $s_{1}$ and $s_{2}$. It
 has a unique solution if and only if $D=(n-\tilde{t}-1)k+\triangle\tilde{t}q\neq0$.  Analogously to the proof
of Proposition \ref{pr-4} we can prove that  $D>0$.

Thus,  the solution of the linear system \eqref{41} with unknowns $s_{1}$ and $s_{2}$ is
\begin{equation}\label{45}
\begin{cases}
  s_{1}=\frac{(k+\triangle\tilde{t}\;)(d_{2}-1)}{D}, \\
  s_{2}=\frac{k(d_{2}-1)}{D}.
\end{cases}
\end{equation}
Substituting \eqref{45} into \eqref{40} gives:
\begin{equation}\label{46}
d_{1}=\frac{\big((n-t-1)k+(p-1)\triangle\tilde{t}\;\big)(d_{2}-1)}{D}+1.
\end{equation}
If $p=0$, then $(n-t-1)k+(p-1)\triangle\tilde{t}\geq 0$ due to the fact that $d_{2}>1, d_{1}\geq 1$  and $D>0$.
This prove the statement  $c)$.

$d)$ The proof  can be done  in a way completely analogous to the   proof of statement $d)$ of Proposition \ref{pr-4}, and hence is omitted.

The proof of Proposition \ref{pr-8} is finished.
\end{proof}
\begin{remark}\label{re-5}
 From the definition of index, we obtain $\lambda=p-q-1-(t-\tilde{t}\;)k/(\triangle\tilde{t})$.
\end{remark}

\section{Algorithm of PSQHPDS}\label{se-4}
Our main aim of this section is to obtain the algorithm of the semi-quasi homogeneous but non-semihomogeneous polynomial vector field of degree $n\geq2$.
This algorithm  consists of three parts.  Everyone can applying this algorithm
step by step to obtain directly all the semi-quasi homogeneous planar polynomial vector fields with a given degree.
We will illustrate the application of this algorithm in the next section.
\subsection{The algorithm of $|A|\geq2$ and $|B|\geq1$ with $d_{1}=1$}\label{sub-4.1}

If $B_{n-\tilde{t}}\not\equiv0$ and $B_{n-\tilde{t}_{1}}\not\equiv0$ (the definition of $B_{n-t}$ can be  seen in section \ref{se-2}),
  with $\tilde{t}\neq\tilde{t}_{1}$, then by  applying Proposition \ref{pr-2},
  we know that  there exist a unique $q\in\{0, . . . , n-\tilde{t}\}$ such that $b_{q, n-\tilde{t}-q}\neq0$
  and a unique $q_{1}\in\{0, . . . , n-\tilde{t}_{1}\}$ such that $b_{q_{1}, n-\tilde{t}_{1}-q_{1}}\neq0$.
  From   equation \eqref{8}, one can get that $\tilde{t}$,
   $\tilde{t}_{1}\in\{0, . . . , n-1\}$, $q\in\{0, . . . , n-\tilde{t}\}$ and $q_{1}\in\{0, . . . , n-\tilde{t}_{1}\}$
   satisfy the following equations
\begin{equation}\label{54}
\begin{split}
\pi_{q}^{\tilde{t}, 0}[0]&:\ qs_{1}+(n-\tilde{t}-q-1)s_{2}+1-d_{2}=0, \\
\pi_{q_{1}}^{\tilde{t}_{1}, 0}[0]&:\ q_{1}s_{1}+(n-\tilde{t}_{1}-q_{1}-1)s_{2}+1-d_{2}=0.
\end{split}
\end{equation}

In the following,  we consider the relation between the equations $\pi_{q}^{\tilde{t}, 0}[0]$ and $\pi_{q_{1}}^{\tilde{t}_{1}, 0}[0]$.
\begin{proposition}\label{pr-11}
Suppose that  system \eqref{1} is a  {\itshape PSQHPDS}   of degree $n\ge 2$
 with weight vector $w=\left( s_{1}, s_{2}, d_{1}, d_{2} \right)$, and that $d_{1}=1$.
  Let $t\in\{0, . . . , n-2\}$ and $\tilde{t}$, $\tilde{t}_{1}\in\{0, . . . , n-1\}$  such that $s_{1}=(n-t)s_{2}$.
  If the two equations $\pi_{q}^{\tilde{t}, 0}[0]$ and $\pi_{q_{1}}^{\tilde{t}_{1}, 0}[0]$ is compatible , then
\begin{equation}\label{62}
 q_{1}-q=\frac{\tilde{t}_{1}-\tilde{t}}{n-t-1}
\end{equation}
\end{proposition}
\begin{proof}
Since $s_{1}=(n-t)s_{2}$ and the two equations $\pi_{q}^{\tilde{t}, 0}[0]$ and $\pi_{q_{1}}^{\tilde{t}_{1}, 0}[0]$ is compatible,  then the following
linear equations
\begin{equation}\label{63}
  \begin{cases}
    s_{1}-(n-t)s_{2}=0, \\
    qs_{1}+(n-\tilde{t}-q-1)s_{2}+1-d_{2}=0, \\
    q_{1}s_{1}+(n-\tilde{t}_{1}-q_{1}-1)s_{2}+1-d_{2}=0,
  \end{cases}
\end{equation}
with unknowns $s_{1}$, $s_{2}$ and $d_{2}-1$ exist nontrivial solution only if its coefficient matrix is non-invertible,
that is, $(n-t-1)(q_{1}-q)-(\tilde{t}_{1}-\tilde{t})=0$.
This completes the proof of the Proposition.
\end{proof}
\begin{remark}\label{re-10}
  By Remark \ref{re-3},  one can easily find that
  \begin{equation}\label{liangadd1}\lambda=-q-(n-\tilde{t}-1)/(n-t-1)=-q_{1}-(n-\tilde{t}_{1}-1)/(n-t-1)\end{equation}
   with $d_{1}=1$. If $t=0$, then
we can change \eqref{62} into
\begin{equation}\label{79}
  q_{1}-q=\frac{\tilde{t}_{1}-\tilde{t}}{n-1}.
\end{equation}
Since $q_{1}-q\in\mathbb{Z}$, $q\leq n-\tilde{t}$, $q_{1}\leq n-\tilde{t}_{1}$, $-(n-1)\leq\tilde{t}_{1}-\tilde{t}\leq n-1$
and $\tilde{t}_{1}\neq\tilde{t}$, in order to satisfy \eqref{79} it is necessary that
$\tilde{t}_{1}=n-1$, $q_{1}=1$, $\tilde{t}=0$, $q=0$ or $\tilde{t}_{1}=0$, $q_{1}=0$, $\tilde{t}=n-1$, $q=1$.

What is more, from \eqref{liangadd1}, we can see that $\lambda$ is neither dependent on $q, \tilde{t}$, nor on $s_1,s_2,d_2$.
\end{remark}

Using the above results, we will establish the first part of the
algebraic algorithm for {\itshape PSQHPDS}.
If $t=0$, then by   statement $b)$ of Proposition \ref{pr-4} and Remark \ref{re-10},
the  corresponding {\itshape PSQHPDS} is
\begin{equation}\label{80}
\begin{split}
  &\dot{x}=a_{0, n}y^{n}+a_{1, 0}x, \\
  &\dot{y}=b_{0, n}y^{n}+b_{1, 0}x, \\
  \end{split}
\end{equation}
with $w=(ns_{2}, s_{2}, 1, (n-1)s_{2}+1)$ and $w_{m}=(n, 1, 1, n)$, and we denote the corresponding vector field  by $X_{0,0,0}$.

Next we turn to the case that $t>0$. Since  $t\in\{1, . . . , n-2\}$, it
 is nature to assume that $n\geq3$.  Noting that  $\deg(X)=n$, which implies $t\cdot\tilde{t}=0$, it follows that $\tilde{t}=0$.

 Now we fix $n\geq3$, $t\in\{1, . . . , n-2\}$, $q\in\{0, 1,. . . , n\}$, and give the algorithm to obtain the
 corresponding  {\itshape PSQHPDS} as follows.
\begin{algorithm}\label{al-4}
  \mbox{}\par
  \begin{enumerate}[(step 1). ]
  \item\emph{We calculate the index $\lambda$ by} (i) \emph{of Remark \ref{re-3}. }
   \item\emph{We calculate the following set}
   \begin{equation}\label{64}
     I_{t, 0, q}:=\{(\tilde{t}^{*}, q^{*}):\tilde{t}^{*}=(q^{*}-q)(n-t-1), \tilde{t}^{*}\in\{1, . . . , n-1\}\;and\;q^{*}\in\{0, . . . , n-\tilde{t}^{*}\}\}.
   \end{equation}
   \emph{ By   equation \eqref{24}, we know that each ordered pair $(\tilde{t}^{*}, q^{*})\in I_{t, 0, q}$
    determines a vector field $B_{n-\tilde{t}^{*}}$. For convenience,
    using $B_{n-\tilde{t}^{*}}^{\tilde{t}^{*}, q^{*}}$ to represent $B_{n-\tilde{t}^{*}}$, we define $B_{n-\tilde{t}^{*}}^{\tilde{t}^{*}, q^{*}}$ as}
 \[B_{n-\tilde{t}^{*}}^{\tilde{t}^{*}, q^{*}}=(0, b_{q^{*}, n-\tilde{t}^{*}-q^{*}}x^{q^{*}}y^{n-\tilde{t}^{*}-q^{*}}). \]
 \emph{The ordered pair $(0, q)$  determines a vector field $B_{n}^{0, q}$, i.e.,  }
 \[B_{n}^{0, q}=(0, b_{q, n-q}x^{q}y^{n-q}). \]
 \emph{Denoted by $Q_{t, 0, q}$ the semi-quasi homogeneous polynomial determined by the  set $I_{t, 0, q}$.  Moreover, we get that}
    \begin{equation}\label{65}
     (0, Q_{t, 0, q})=B_{n}^{0, q}+\sum\limits_{(\tilde{t}^{*}, q^{*})\in I_{t, 0, q}}B_{n-\tilde{t}^{*}}^{\tilde{t}^{*}, q^{*}}.
    \end{equation}
    \emph{Together with statement} b.1) \emph{of Proposition \ref{pr-4}, we obtain the
     semi-quasi homogeneous but non-semihomogeneous vector field $X_{t, 0, q}$ as follows}
    \begin{equation}\label{67}
      X_{t, 0, q}=(a_{0, n-t}y^{n-t}+a_{1, 0}x, Q_{t, 0, q}).
    \end{equation}
    \emph {For semi-quasi homogenous vector field $X_{t, 0, q}$, the weight vector $w$ and the minimal weight vector $w_{m}$ are given in statement} b) \emph{of Proposition \ref{pr-4}. }\\
    \emph{We note that if index $\lambda=0$, then the system is quasi-homogenous. }
   \item\emph{This algorithm is ended. }
  \end{enumerate}
\end{algorithm}

\begin{remark} If we take the value of $t$ successively from $1$ to $n-2$, and then choose in turn  $q=0, . . . , n$ for each $t$
(keep $\lambda\ne 0$),
then  together with the \eqref{80} and the Algorithm Part \ref{al-4}, we can get all the
   {\itshape PSQHPDS} of degree $n$ and $|A|\geq2$, $|B|\geq1$ with weight degree $d_{1}=1$.
 \end{remark}

\subsection{The algorithm of $|A|\geq2$ and $|B|\geq1$ with $d_{1}>1$}\label{sub-4.2}
 In view of Proposition \ref{pr-4} and Remark \ref{re-3}, one can get that $t\in\{0, . . . , n-2\}$, $\tilde{t}\in\{0, . . . , n-1\}$, $q\in\{0, . . . , n-\tilde{t}\;\}$, $p\in\{0, . . . , n-t-2\}$, $\triangle t\in\{1, . . . , n-t-p-1\}$ and $k\in\{1, . . . , n-t-\triangle t-p\}$ satisfy the following equations
 \begin{equation}\label{49}
 \begin{split}
 t\cdot\tilde{t}&=0, \\
    e_{p}^{t, 0}\left[ 0 \right]&:\left( p-1 \right){{s}_{1}}+\left( n-t-p \right){{s}_{2}}+1-{{d}_{1}}=0, \\
   e_{p}^{t, \triangle t}\left[ k \right]&:\left( k+p-1 \right){{s}_{1}}+\left( n-\triangle t-t-k-p \right){{s}_{2}}+1-{{d}_{1}}=0, \\
    \lambda=&p-q-1-\frac{(t-\tilde{t})k}{\triangle t}.
 \end{split}
 \end{equation}

 For every $t\in\{0, . . . , n-2\}$,  $p\in\{0, . . . , n-t-2\}$ and $\triangle t\in\{1, . . . , n-t-p-1\}$, we define the set of equations $A_{p}^{t}(\triangle t)$ as
\[{A^{t}_{p}}(\triangle t)=\{e_{p}^{t, \triangle t}[k]:k=1, . . . , n-t-\triangle t-p\}. \]
\par Fixed $t\in\{0, . . . , n-2\}$ and $p\in\{0, . . . , n-t-2\}$, we define the following set of equations
 \begin{equation}\label{50}
 E_{p}^{t}=\{e_{p}^{t, 0}[0]\}\bigcup A_{p}^{t}(1)\bigcup. . . \bigcup A_{p}^{t}(n-t-p-1).
 \end{equation}

In order to determine the {\itshape PSQHPDS},
we need to find a set of linear equations that contains the equation $e_{p}^{t, 0}$ and at most one equation of each set
of equations $A_{p}^{t}(\triangle t)$ and such that the set of all these equations define a compatible linear system
 being $s_{1}$ and $s_{2}$ the unknowns, and satisfying that if we add some other equation the new
  linear system be incompatible. As was done in \cite{1}, we say that such linear systems is the \emph{maximal linear systems} of $E_{p}^{t}$.  Each maximal linear systems will determine  semi-quasi homogeneous polynomial $P$.
\begin{remark}\label{re-6}
We note that $A_{0}^{t}(n-t-1)=\{e_{0}^{t, n-t-1}[1]\}$, that   is, $0=d_{1}-1$, in contraction with $d_{1}>1$. Thus, we remove from $E_{0}^{t}$ the set $\{e_{0}^{t, n-t-1}[1]\}$. We denote
\[\varepsilon_{p}^{t}=E_{p}^{t}\;if\;p>0, \;and\;\varepsilon_{0}^{t}=E_{0}^{t}\backslash \{e_{0}^{t, n-t-1}[1]\}. \]
\end{remark}

 If $A_{n-\triangle t_{1}-t}\not\equiv0$ and $A_{n-\triangle t_{2}-t}\not\equiv0$ (the definition of $A_{i}$ can be seen in
  Section \ref{se-2}), we fix $\triangle t_{1}$ and $\triangle t_{2}$ with $\triangle t_{1}\neq\triangle t_{2}$.
  In the next proposition, we shall consider the relation  between
   equations $e_{p}^{t, 0}[0]$, $e_{p}^{t, \triangle t_{1}}[k_{1}]$ and $e_{p}^{t, \triangle t_{2}}[k_{2}]$.
\begin{proposition}\label{pr-10}
Suppose that     system \eqref{1} is a  {\itshape PSQHPDS}
  of degree $n\ge 2$ with weight vector $w=\left( s_{1}, s_{2}, d_{1}, d_{2} \right)$,
 where $d_{1}>1$. Let $t\in\{0, . . . , n-2\}$, $p\in\{0, . . . , n-t\}$ and $ \triangle t_{1}, \triangle t_{2}\in\{1, . . . , n-t-1\}$.
  If the three equations $e_{p}^{t, 0}[0]$, $e_{p}^{t, \triangle t_{1}}[k_{1}]$ and $e_{p}^{t, \triangle t_{2}}[k_{2}]$ are compatible, then
\begin{equation}\label{51}
  k_{1}\triangle t_{2}=k_{2}\triangle t_{1}.
\end{equation}
\end{proposition}
\begin{proof}
Since the three equations $e_{p}^{t, 0}[0]$, $e_{p}^{t, \triangle t_{1}}[k_{1}]$ and $e_{p}^{t, \triangle t_{2}}[k_{2}]$ are compatible,
 then the following linear equations
\begin{equation}\label{52}
  \begin{cases}
    \left( p-1 \right){{s}_{1}}+\left( n-t-p \right){{s}_{2}}+1-{{d}_{1}}=0, \\
    \left( k_{1}+p-1 \right){{s}_{1}}+\left( n-\triangle t_{1}-t-k_{1}-p \right){{s}_{2}}+1-{{d}_{1}}=0, \\
    \left( k_{2}+p-1 \right){{s}_{1}}+\left( n-\triangle t_{2}-t-k_{2}-p \right){{s}_{2}}+1-{{d}_{1}}=0,
  \end{cases}
\end{equation}
with unknowns $s_{1}$, $s_{2}$ and $d_{1}$ have  nontrivial solution only if its
coefficient matrix is non-invertible. That is, $ k_{1}\triangle t_{2}-k_{2}\triangle t_{1}=0$. The proof is finished.
\end{proof}

 If $B_{n-\tilde{t}}\not\equiv0$ and $B_{n-\tilde{t}_{1}}\not\equiv0$, then after  fixing  $\tilde{t}$ and $\tilde{t}_{1}$ with $\tilde{t}\neq\tilde{t}_{1}$,
we know from Proposition \ref{pr-2}  that,  there exist a unique $q\in\{0, . . . , n-\tilde{t}\}$
such that $b_{q, n-\tilde{t}-q}\neq0$ and a unique $q_{1}\in\{0, . . . , n-\tilde{t}_{1}\}$
such that $b_{q_{1}, n-\tilde{t}_{1}-q_{1}}\neq0$. From   equation \eqref{8},
 one can get that $\tilde{t}$, $\tilde{t}_{1}\in\{0, . . . , n-1\}$, $q\in\{0, . . . , n-\tilde{t}\}$
  and $q_{1}\in\{0, . . . , n-\tilde{t}_{1}\}$ satisfy the equations $\pi_{q}^{\tilde{t}, 0}[0]$ and $\pi_{q_{1}}^{\tilde{t}_{1}, 0}[0]$.

Now we shall consider the relation  between the equations $e_{p}^{t, 0}[0]$, $e_{p}^{t, \triangle t_{1}}[k_{1}]$, $\pi_{q}^{\tilde{t}, 0}[0]$ and $\pi_{q_{1}}^{\tilde{t}_{1}, 0}[0]$.
\begin{proposition}\label{pr-12}
Suppose that   system \eqref{1} is a    {\itshape PSQHPDS} of degree $n\ge 2$
with weight vector $w=\left( s_{1}, s_{2}, d_{1}, d_{2} \right)$, where $d_{1}>1$.
Let $t\in\{0, . . . , n-2\}$, $p\in\{0, . . . , n-t\}$, $\triangle t_{1}\in\{1, . . . , n-t-1\}$,
 $k_{1}\in\{1, . . . , n-t-\triangle t_{1}-p\}$, $\tilde{t}\in\{0, . . . , n-1\}$ and $q\in\{0, . . . , n-\tilde{t}\;\}$.
  If the four equations $e_{p}^{t, 0}[0]$, $e_{p}^{t, \triangle t_{1}}[k_{1}]$, $\pi_{q}^{\tilde{t}, 0}[0]$ and
  $\pi_{q_{1}}^{\tilde{t}_{1}, 0}[0]$ are compatible, then
\begin{equation}\label{55}
  q-q_{1}=\frac{(\tilde{t}-\tilde{t}_{1})k_{1}}{\triangle t_{1}}
\end{equation}
\end{proposition}
\begin{proof}
Since  equations $e_{p}^{t, 0}[0]$, $e_{p}^{t, \triangle t_{1}}[k_{1}]$,
 $\pi_{q}^{\tilde{t}, 0}[0]$ and $\pi_{q_{1}}^{\tilde{t}_{1}, 0}[0]$ are compatible,  the following linear equations
\begin{equation}\label{53}
\begin{cases}
\left( p-1 \right){{s}_{1}}+\left( n-t-p \right){{s}_{2}}+1-{{d}_{1}}=0, \\
    \left( k_{1}+p-1 \right){{s}_{1}}+\left( n-\triangle t_{1}-t-k_{1}-p \right){{s}_{2}}+1-{{d}_{1}}=0, \\
    qs_{1}+(n-\tilde{t}-q-1)s_{2}+1-d_{2}=0, \\
    q_{1}s_{1}+(n-\tilde{t}_{1}-q_{1}-1)s_{2}+1-d_{2}=0,
\end{cases}
\end{equation}
with unknowns $s_{1}$, $s_{2}$, $d_{1}-1$ and $d_{2}$ have nontrivial solution only if its coefficient matrix is non-invertible,
that is, $(q-q_{1})\triangle t_{1}-(\tilde{t}-\tilde{t}_{1})k_{1}=0$. Hence Proposition \ref{pr-12} is proved.
\end{proof}
\begin{remark}\label{re-7}
Using Remak \ref{re-3} and Proposition \ref{pr-10}, one can easily get that $\lambda=p-q-1-\frac{(t-\tilde{t})k}{\triangle t}=p-q_{1}-1-\frac{(t-\tilde{t}_{1})k_{1}}{\triangle t_{1}}$ with $d_{1}>1$.
Therefore,  $\lambda$ is neither dependent on $q,\tilde{t}$, nor on $s_1,s_2,d_1, d_2$.\end{remark}

 By applying   Proposition \ref{pr-12}, we can determine the semi-quasi homogeneous polynomial $Q$.
  Moreover, we can obtain a semi-quasi homogeneous  polynomial vector field $X=(P, Q)$ of degree $n\geq2$
  with weight vector $w=(s_{1}, s_{2}, d_{1}, d_{2})$,   $d_{1}>1$.

 Based on the above discussion, we will provide an algorithm for determining  {\itshape PSQHPDS}
  with $|A|\geq2$, $|B|\geq1$ and with weight degree $d_{1}>1$. We fix $n\geq2$, $t\in\{0, . . . , n-2\}$ and $p\in\{0, . . . , n-t-2\}$.
For convenience,  we define the following set
\begin{equation}\label{101}
\tilde{ I}= \bigcup\limits_{\tilde{t}\in\{0, . . . , n-1\}\;\text{such that}\;t\cdot \tilde{t}=0}\tilde{I}(\tilde{t}),
\end{equation}
for each  fixed $t\in\{0, . . . , n-2\}$, where $\tilde{I}(\tilde{t})=\{\tilde{t}\}\times\{0, . . . , n-\tilde{t}\}$.
\begin{algorithm}\label{al-1}
 \mbox{}\par
  \begin{enumerate}[(step 1). ]
  \item\emph{We take the equation $e_{p}^{t, 0}[0]$ as the first equation. }
  \item\emph{We fix $\triangle t\in\{1, . . . , n-t-p-1\}$ and an equation of $A_{p}^{t}(\triangle t)\bigcap \varepsilon_{p}^{t}$ , that is, an
equation of the form $e_{p}^{t, \triangle t}[k]$ with a $k\in\{1, . . . , n-t-\triangle t-p\}$. }
  \item\emph{For each $\triangle t^{*}\in\{1, . . . , n-t-p-1\}\backslash\{\triangle t\}$,
   if there exist $k_{\triangle t^{*}}\in\{1, . . . , n-t-\triangle t^{*}-p\}$ satisfying \eqref{51}, i.e.,
    $k_{\triangle t^{*}}\triangle t=k\triangle t^{*}$, then we have an equation $e_{p}^{t, \triangle t^{*}}[k_{\triangle t^{*}}]$. }
  \item\emph{Working in} step 1, 2 \emph{and} 3, \emph{we can get the following set of equations}
  \begin{equation}\label{56}
    \varepsilon_{p}^{t, \triangle t, k}:=\bigcup\limits_{\triangle t^{*}\in\{1, . . . , n-t-p-1\}\backslash\{\triangle t\}}\{e_{p}^{t, \triangle t^{*}}[k_{\triangle t^{*}}]:k_{\triangle t^{*}}\triangle t=k\triangle t^{*}\}\bigcup\{e_{p}^{t, 0}[0], e_{p}^{t, \triangle t}[k]\}.
  \end{equation}
  \emph{Form   equation \eqref{25}, we know that each equation $e_{p}^{t, \triangle t^{*}}[k_{\triangle t^{*}}]$ determines
a vector field $A_{n-\triangle t^{*}-t}$. For convenience, using $A_{n-\triangle t^{*}-t}^{\triangle t^{*}, k_{\triangle t^{*}}}$
to represent $A_{n-\triangle t^{*}-t}$, we define $A_{n-\triangle t^{*}-t}^{\triangle t^{*}, k_{\triangle t^{*}}}$ as}
\[A_{n-\triangle t^{*}-t}^{\triangle t^{*}, k_{\triangle t^{*}}}=
\left(a_{p+k_{\triangle t^{*}}, n-\triangle t^{*}-t-p-k_{\triangle t^{*}}}x^{p+k_{\triangle t^{*}}}y^{n-\triangle t^{*}-t-p-k_{\triangle t^{*}}}, 0\right). \]
\emph{The equation $e_{p}^{t, \triangle t}[k]$ and equation $e_{p}^{t, 0}[0]$ determine vector field $A_{n-\triangle t-t}$ and $A_{n-t}$,
respectively, that is, }
\[A_{n-\triangle t-t}^{\triangle t, k}=\left(a_{p+k, n-\triangle t-t-p-k}x^{p+k}y^{n-\triangle t-t-p-k}, 0\right)\] \emph{and}
\[A_{n-t}=(a_{p, n-t-p}x^{p}y^{n-t-p}, 0). \]
\emph{Denoted by $P_{t, p, \triangle t, k}$ the semi-quasi homogeneous polynomial
determined by the equations set $\varepsilon_{p}^{t, \triangle t, k}$. We have}
\begin{equation}\label{57}
  (P_{t, p, \triangle t, k}, 0)=A_{n-t}+A_{n-\triangle t-t}^{\triangle t, k}
  +\sum\limits_{\triangle t^{*}\in\{1, . . . , n-t-p-1\}\backslash\{\triangle t\}\;\mbox{and} \;k_{\triangle t^{*}}\triangle t=k\triangle t^{*}}A_{n-\triangle t^{*}-t}^{\triangle t^{*}, k_{\triangle t^{*}}}.
\end{equation}
  \item\emph{Chose an ordered pair $(\tilde{t}, q)\in \tilde{I}$,  we calculate the following set}
  \begin{small}
  \begin{equation}\label{58}
    I_{\tilde{t}, q, \triangle t, k}:=\{(\tilde{t}^{*}, q^{*}):(q-q^{*})\triangle t=(\tilde{t}-\tilde{t}^{*})k,
     \tilde{t}^{*}\in\{0, . . . , n-1\}\backslash\{\tilde{t}\}\;\mbox{and} \;q^{*}\in\{0, . . . , n-\tilde{t}^{*}\}\}.
  \end{equation}
  \end{small}

  \item\emph{We remove from $\tilde{I}$ the set $I_{\tilde{t}, q, \triangle t, k}$, and then  go back to} step 5 \emph{and repeat this process until the set $\tilde{I}=\varnothing$. }
 \item\emph{ From  equation \eqref{24}, we know that each ordered pair $(\tilde{t}^{*}, q^{*})\in I_{\tilde{t}, q, \triangle t, k}$
  determines a vector field $B_{n-\tilde{t}^{*}}$. For convenience,  using $B_{n-\tilde{t}^{*}}^{\tilde{t}^{*}, q^{*}}$
  to represent $B_{n-\tilde{t}^{*}}$, that is, we define $B_{n-\tilde{t}^{*}}^{\tilde{t}^{*}, q^{*}}$ as}
 \[B_{n-\tilde{t}^{*}}^{\tilde{t}^{*}, q^{*}}=\left(0, b_{q^{*}, n-\tilde{t}^{*}-q^{*}}x^{q^{*}}y^{n-\tilde{t}^{*}-q^{*}}\right). \]
 \emph{The ordered pair $(\tilde{t}, q)$  determines a vector field $B_{n-\tilde{t}}^{\tilde{t}, q}$, i.e.,  }
 \[B_{n-\tilde{t}}^{\tilde{t}, q}=(0, b_{q, n-\tilde{t}-q}x^{q}y^{n-\tilde{t}-q}). \]
 \emph{Denoted by $Q_{\tilde{t}, q, \triangle t, k}$ the semi-quasi homogeneous polynomial determined by the set $I_{\tilde{t}, q, \triangle t, k}$ .
 Moreover, we get that}
 \begin{equation}\label{59}
  (0, Q_{\tilde{t}, q, \triangle t, k})=B_{n-\tilde{t}}^{\tilde{t}, q}+\sum\limits_{(\tilde{t}^{*}, q^{*})\in I_{\tilde{t}, q, \triangle t, k}}B_{n-\tilde{t}^{*}}^{\tilde{t}^{*}, q^{*}}.
 \end{equation}
 \emph{Together with \eqref{57}, we   obtain the semi-quasi homogeneous   vector field $X_{t, \tilde{t}, p, q, \triangle t, k}$ as follows}
 \begin{equation}\label{61}
   X_{t, \tilde{t}, p, q, \triangle t, k}=(P_{t, p, \triangle t, k}, Q_{\tilde{t}, q, \triangle t, k}).
 \end{equation}
 \emph{Using statement $c)$ of Proposition \ref{pr-4} and Remark \ref{re-3},
 we can obtain the vector field $X_{t, \tilde{t}, p, q, \triangle t, k}$ corresponding to
 weight vector $w$, the minimal weight vector $w_{m}$ and the index $\lambda$. }\\
 \emph{We note that if index $\lambda=0$, then system is quasi-homogenous and should be excluded. }
  \item\emph{We remove from $\varepsilon_{p}^{t}$ the equations $\varepsilon_{p}^{t, \triangle t, k}\backslash\{e_{p}^{t, 0}[0]\}$. }
  \item\emph{We return to} step 1. \emph{We repeat } step 1 \emph{to} step 8 \emph{until $\varepsilon_{p}^{t}=\{e_{p}^{t, 0}[0]\} $. }
  \end{enumerate}
\end{algorithm}
\begin{remark}\label{re-8}Fixed $t$ and $p$, one can use the equation $e_{p}^{t, 0}[0]$ and
the set $I_{\tilde{t}, q, \triangle t, k}$ to obtain all the semi-quasi homogeneous
  vector field $X_{t, \tilde{t}, p, q, \triangle t, k}$ with $d_{1}>1$.
  Firstly,  take the value of $t=0, . . . , n-2$ in proper order, and then
   successively take the value of $p=0, . . . , n-t-2$ for each $t$.
   In this way,  everyone can   obtain all the  {\itshape PSQHPDS} of degree $n$ for the case that $|A|\geq2$, $|B|\geq1$ and $d_{1}>1$.
\end{remark}
\begin{remark}\label{re-9}
In particular, if $\lambda=0$, then the function of Algorithm Part \ref{al-1} provides quasi homogeneous but non-homogeneous polynomial vector field the same as algorithm given by \cite{1}.
\end{remark}
\subsection{The algorithm of $|A|=1$ and $|B|\geq2$}
 From Proposition \ref{pr-8} and Remark \ref{re-5},
 we can obtain that $\tilde{t}\in\{0, . . . , n-2\}$, $t\in\{0, . . . , n-1\}$, $p\in\{0, . . . , n-t\}$,
  $q\in\{0, . . . , n-\tilde{t}-2\}$, $\triangle\tilde{t}\in\{1, . . . , n-\tilde{t}-q-1\}$
   and $k\in\{1, . . . , n-\tilde{t}-\triangle\tilde{t}-q\}$  satisfy the following equations
\begin{equation}\label{68}
 \begin{split}
 t\cdot\tilde{t}&=0, \\
    \pi _{q}^{\tilde{t}, 0}\left[ 0 \right]&:q{{s}_{1}}+\left( n-\tilde{t}-q-1 \right){{s}_{2}}+1-{{d}_{2}}=0, \\
   \pi _{q}^{\tilde{t}, \triangle\tilde{t}}\left[ k \right]&:\left( k+q \right){{s}_{1}}+\left( n-\triangle\tilde{t}-\tilde{t}-k-q-1 \right){{s}_{2}}+1-{{d}_{2}}=0, \\
   (n-t&-1)k\geq\triangle\tilde{t}\;\text{if}\;p=0, \\
    \lambda=&p-q-1-\frac{(t-\tilde{t}\;)k}{\triangle\tilde{t}}.
 \end{split}
 \end{equation}

 For every $\tilde{t}\in\{0, . . . , n-2\}$,  $q\in\{0, . . . , n-\tilde{t}-2\}$ and $\triangle \tilde{t}\in\{1, . . . , n-\tilde{t}-q-1\}$, we define the set of equations $B_{q}^{\tilde{t}}(\triangle \tilde{t})$ as
\[B_{q}^{\tilde{t}}\left( \triangle\tilde{t} \right)=\left\{ \pi _{q}^{\tilde{t}, \triangle\tilde{t}}\left[ k \right]:k=1, \ldots , n-\triangle\tilde{t}-\tilde{t}-q \right\}. \]

Fixed $\tilde{t}\in\{0, . . . , n-2\}$ and $q\in\{0, . . . , n-\tilde{t}-2\}$, we define the following set of equations
\begin{equation}\label{69}
 E_{q}^{\tilde{t}}=\{ \pi _{q}^{\tilde{t}, 0}\left[ 0 \right]\}\bigcup B_{q}^{\tilde{t}}\left( 1 \right)\bigcup \cdots \bigcup B_{q}^{\tilde{t}}\left( n-\tilde{t}-q-1 \right).
\end{equation}

 In the same way as in the Subsection \ref{sub-4.2}, we can also define the \emph{maximal linear systems} of $E_{q}^{\tilde{t}}$.
 Each maximal linear system  will determine a  semi-quasi homogeneous polynomial $Q$.

 Similar to Proposition \ref{pr-10} and Proposition \ref{pr-12}, we have the following Proposition \ref{pr-13} and \ref{pr-14}, respectively.
\begin{proposition}\label{pr-13}
  Suppose that  system \eqref{1} is a   {\itshape PSQHPDS} of degree $n\ge 2$
  with weight vector $w=\left( s_{1}, s_{2}, d_{1}, d_{2} \right)$.
  Let $\tilde{t}\in\{0, . . . , n-2\}$, $q\in\{0, . . . , n-\tilde{t}\}$
  and $ \triangle \tilde{t}_{1}, \triangle \tilde{t}_{2}\in\{1, . . . , n-\tilde{t}-1\}$.
   If the three equations $ \pi _{q}^{\tilde{t}, 0}\left[ 0 \right]$,
   $\pi _{q}^{\tilde{t}, \triangle\tilde{t}_{1}}\left[ k_{1} \right]$
   and $\pi _{q}^{\tilde{t}, \triangle\tilde{t}_{2}}\left[ k_{2} \right]$ are compatible, then
\begin{equation}\label{70}
  k_{1}\triangle\tilde{ t}_{2}=k_{2}\triangle \tilde{t}_{1}.
\end{equation}
\end{proposition}
\begin{proof}
  The proof is completely analogous to the proof of Proposition \ref{pr-10} and is omitted for the sake of brevity.
\end{proof}
\begin{proposition}\label{pr-14}
  Suppose that   system \eqref{1} is a  {\itshape PSQHPDS} of degree $n\ge 2$
   with weight vector $w=\left( s_{1}, s_{2}, d_{1}, d_{2} \right)$.
    Let $\tilde{t}\in\{0, . . . , n-2\}$, $q\in\{0, . . . , n-\tilde{t}\}$,
     $\triangle \tilde{t}_{1}\in\{1, . . . , n-\tilde{t}-1\}$,
$k_{1}\in\{1, . . . , n-\tilde{t}-\triangle \tilde{t}_{1}-q\}$, $t\in\{0, . . . , n-1\}$
 and $p\in\{0, . . . , n-t\}$.
If the four equations $ \pi _{q}^{\tilde{t}, 0}\left[ 0 \right]$,
 $\pi _{q}^{\tilde{t}, \triangle\tilde{t}_{1}}\left[ k_{1} \right]$ , $e_{p}^{t, 0}[0]$ and $e_{p_{1}}^{t_{1}, 0}[0]$  are compatible, then
\begin{equation}\label{71}
  p-p_{1}=\frac{(t-t_{1})k_{1}}{\triangle \tilde{t}_{1}}
\end{equation}
\end{proposition}
\begin{proof}
  The proof is analogous to the proof of Proposition \ref{pr-12} and is omitted.
\end{proof}
\begin{remark}\label{re-11}
  Using Remak \ref{re-5} and Proposition \ref{pr-13}, one can easily get that
  \begin{equation}\label{liangadd2}\lambda=p-q-1-\frac{(t-\tilde{t}\;)k}{\triangle\tilde{t}}=p_{1}-q-1-\frac{(t_{1}-\tilde{t}\;)k_{1}}{\triangle\tilde{t}_{1}}.\end{equation}
   Obviously,   $|A|=1$ and $|B|\geq2$ if and only if for the equation $(p-p_{1})\triangle \tilde{t}_{1}=(t-t_{1})k_{1}$
    with unknowns $p_{1}$ and $t_{1}$  has a unique non-negative integer solution $t_{1}=t$, $p_{1}=p$.
Moreover, from \eqref{liangadd2}
it is clear that  $\lambda$ is neither dependent on $p,t$, nor on $s_1,s_2,d_1, d_2$.
\end{remark}

 Using Proposition \ref{pr-14}, we can determine the semi-quasi homogeneous polynomial $P$.
  Moreover, we can obtain a semi-quasi homogeneous   polynomial vector field $X=(P, Q)$ of degree $n\geq2$
  with weight vector $w=(s_{1}, s_{2}, d_{1}, d_{2})$
  as follows.

 For every $\triangle \tilde{t}\in\{1, . . . , n-\tilde{t}-q-1\}$, we define the set  $I_{\tilde{t}, q}(\triangle t)$ as
\begin{equation}\label{75}
  I_{\tilde{t}, q}(\triangle \tilde{t})=\{\triangle \tilde{t}\;\}\times\{1, . . . , n-\tilde{t}-\triangle \tilde{t}-q\},
\end{equation}
where $\times$ is  the Cartesian product.
Furthermore, fixed $\tilde{t}\in\{0, . . . , n-2\}$ and $q\in\{0, . . . , n-\tilde{t}-2\}$, we  introduce the following set
\begin{equation}\label{76}
  I_{\tilde{t}, q}= I_{\tilde{t}, q}(1)\bigcup. . . \bigcup I_{\tilde{t}, q}(n-\tilde{t}-q-1).
\end{equation}

On the other hand, for each fixed $\tilde{t}\in\{0, . . . , n-2\}$, we define the following set
\begin{equation}\label{89}
 I= \bigcup\limits_{t\in\{0, . . . , n-1\}\;\text{such that}\;t\cdot \tilde{t}=0}I(t),
\end{equation}
where $I(t)=\{t\}\times\{0, . . . , n-t\}$.

Similar to Algorithm Part \ref{al-1}, fixed $n\geq2$, $\tilde{t}\in\{0, . . . , n-2\}$ and $q\in\{0, . . . , n-\tilde{t}-2\}$,
we can construct Algorithm Part \ref{al-2}.
\begin{algorithm}\label{al-2}
\mbox{}\par
\begin{enumerate}[(step 1). ]
\item\emph{We choose ordered pair $(\triangle\tilde{t}, k)\in I_{\tilde{t}, q}$. }
\item\emph{For every ordered pair $(t, p)\in I$, we calculate the following set}
  \begin{small}
     \begin{equation}\label{74}
    I_{t, p, \triangle \tilde{t}, k}:=\{(t^{*}, p^{*}):(p-p^{*})\triangle \tilde{t}=(t-t^{*})k, t^{*}\in\{0, . . . , n-1\}\backslash\{t\}\;\mbox{and}\;p^{*}\in\{0, . . . , n-t^{*}\}\}.
  \end{equation}
  \end{small}
  \emph{Furthermore, we can get the following set}
  \begin{equation}\label{90}
    S_{\triangle \tilde{t}, k}:=\{(t, p):(t, p)\in I\;\mbox{such \ that}\;I_{t, p, \triangle \tilde{t}, k}=\varnothing\;and\;(n-t-1)k\geq\triangle\tilde{t}\;if\;p=0\}.
  \end{equation}
  \emph{If the set $S_{\triangle \tilde{t}, k}\neq\varnothing$, then go to }step 3; \emph{if not then jump to} step 6.
\item\emph{We take   equation $ \pi _{q}^{\tilde{t}, 0}\left[ 0 \right]$ as the first equation. The ordered pair $(\triangle\tilde{t}, k)$ determines an equation $\pi_{q}^{\tilde{t}, \triangle \tilde{t}}[k]$. }
\item\emph{If there exist a ordered pair $(\triangle \tilde{t}^{*}, k_{\triangle \tilde{t}^{*}})\in I_{\tilde{t}, q}\backslash\{(\triangle\tilde{t}, k)\}$ satisfying \eqref{70}, i.e.,  $k_{\triangle \tilde{t}^{*}}\triangle \tilde{t}=k\triangle \tilde{t}^{*}$, then we have an equation $\pi_{q}^{\tilde{t}, \triangle \tilde{t}^{*}}[k_{\triangle \tilde{t}^{*}}]$. }
\item\emph{Working in} step 3 \emph{and} 4, \emph{we can get the following set of equations}
  \begin{equation}\label{72}
    \tau_{q}^{\tilde{t}, \triangle \tilde{t}, k}:=\bigcup\limits_{I_{\tilde{t}, q}\backslash\{(\triangle\tilde{t}, k)\}}\{\pi_{q}^{\tilde{t}, \triangle \tilde{t}^{*}}[k_{\triangle \tilde{t}^{*}}]:k_{\triangle \tilde{t}^{*}}\triangle\tilde{ t}=k\triangle \tilde{t}^{*}\}\bigcup\{\pi_{q}^{\tilde{t}, 0}[0], \pi_{q}^{\tilde{t}, \triangle \tilde{t}}[k]\}.
  \end{equation}
  \emph{Form   equation \eqref{41}, we know that each equation $\pi_{q}^{\tilde{t}, \triangle \tilde{t}^{*}}[k_{\triangle \tilde{t}^{*}}]$ determines
a vector field $B_{n-\triangle \tilde{t}^{*}-\tilde{t}}$. For convenience, we use $B_{n-\triangle \tilde{t}^{*}-\tilde{t}}^{\triangle \tilde{t}^{*}, k_{\triangle \tilde{t}^{*}}}$ to represent $B_{n-\triangle \tilde{t}^{*}-\tilde{t}}$. That is, we define $B_{n-\triangle \tilde{t}^{*}-\tilde{t}}^{\triangle \tilde{t}^{*}, k_{\triangle \tilde{t}^{*}}}$ as}
\[B_{n-\triangle \tilde{t}^{*}-\tilde{t}}^{\triangle \tilde{t}^{*}, k_{\triangle \tilde{t}^{*}}}=(0, b_{q+k_{\triangle \tilde{t}^{*}}, n-\triangle \tilde{t}^{*}-\tilde{t}-q-k_{\triangle \tilde{t}^{*}}}x^{q+k_{\triangle \tilde{t}^{*}}}y^{n-\triangle \tilde{t}^{*}-\tilde{t}-q-k_{\triangle\tilde{ t}^{*}}}). \]
\emph{The equation $\pi_{q}^{\tilde{t}, \triangle \tilde{t}}[k]$ and equation $\pi_{q}^{\tilde{t}, 0}[0]$ determine vector field $B_{n-\triangle \tilde{t}-\tilde{t}}$ and $B_{n-\tilde{t}}$,  respectively, that is, }
\[B_{n-\triangle \tilde{t}-\tilde{t}}^{\triangle \tilde{t}, k}=(0, b_{q+k, n-\triangle \tilde{t}-\tilde{t}-q-k}x^{q+k_{\triangle \tilde{t}}}y^{n-\triangle \tilde{t}-\tilde{t}-q-k})\] \emph{and}
\[B_{n-\tilde{t}}=(0, b_{q, n-\tilde{t}-q}x^{q}y^{n-\tilde{t}-q}). \]
\emph{Denoted by $Q_{\tilde{t}, q, \triangle \tilde{t}, k}$ the semi-quasi homogeneous polynomial
determined by the  equations set $\tau_{q}^{\tilde{t}, \triangle \tilde{t}, k}$. Moreover, we have}
\begin{equation}\label{73}
  (0, Q_{\tilde{t}, q, \triangle \tilde{t}, k})=B_{n-\tilde{t}}+B_{n-\triangle \tilde{t}-\tilde{t}}^{\triangle \tilde{t}, k}+\sum\limits_{(\triangle \tilde{t}^{*}, k_{\triangle \tilde{t}^{*}})\in I_{\tilde{t}, q}\backslash\{(\triangle\tilde{t}, k)\}\;and\;k_{\triangle \tilde{t}^{*}}\triangle \tilde{t}=k\triangle \tilde{t}^{*}}B_{n-\triangle \tilde{t}^{*}-\tilde{t}}^{\triangle \tilde{t}^{*}, k_{\triangle \tilde{t}^{*}}}.
\end{equation}
\emph{Each  ordered pair $(t, p)\in S_{\triangle \tilde{t}, k}$ determines a vector field $A_{n-t}$, i.e.,  }
\begin{equation}\label{77}
  A_{n-t}=(a_{p, n-t-p}x^{p}y^{n-t-p}, 0).
\end{equation}
\emph{Together with \eqref{73}, we obtain the semi-quasi homogeneous   vector field $X_{t, \tilde{t}, p, q, \triangle t, k}$ as follows}
\begin{equation}\label{78}
  X_{t, \tilde{t}, p, q, \triangle t, k}=(a_{p, n-t-p}x^{p}y^{n-t-p}, Q_{\tilde{t}, q, \triangle \tilde{t}, k}).
\end{equation}
\emph {Applying Remark \ref{re-5},  one can get the index $\lambda$. For semi-quasi homogenous vector field $X_{t, \tilde{t}, p, q, \triangle t, k}$ , the weight vector $w$ and the minimal weight vector $w_{m}$ are given in statement $c)$ of Proposition \ref{pr-8}. }
\emph{Note that we will next  go to  } step 7 \emph{directly}. \\
\emph{We observe that if index $\lambda=0$, then the system is quasi-homogenous and should be excluded. }
\item\emph{We remove from $I_{\tilde{t}, q}$ the ordered pair $(\triangle\tilde{t}, k)$ and go back to} step 1.
\item\emph{We remove from $I_{\tilde{t}, q}$ the set $\{(\triangle\tilde{t}, k)\}\bigcup\{(\triangle \tilde{t}^{*}, k_{\triangle \tilde{t}^{*}}):k_{\triangle \tilde{t}^{*}}\triangle\tilde{ t}=k\triangle \tilde{t}^{*}\}$. }
\item\emph{We return to} step 1. \emph{We repeat } step 1 \emph{to} step 7 \emph{until $I_{\tilde{t}, q}=\varnothing $. }
\end{enumerate}
\end{algorithm}

If we take the value of $\tilde{t}$ form $0$ to $n-2$ in turn, and then successively take the value of $q=0, . . . , n-\tilde{t}-2$ for each $\tilde{t}$, we will obtain   all the   {\itshape PSQHPDS} of degree $n$ with $|A|=1$ and $|B|\geq2$.

 In a word, based on the above Algorithm which include Part \ref{al-4}, Part  \ref{al-1}, Part  \ref{al-2},  everyone can
  obtain directly all the {\itshape PSQHPDS}  with a given degree.

  In the end of this section,   we would like to point out that the concept "the index of a given {\itshape PSQHPDS}" is well-defined.

\begin{remark}\label{re-12}
For any given {\itshape PSQHPDS} (don't forget that we have exclude the semihomogeneous systems),  the index is unique, although
the weight vector is not unique. This conclusion  can be confirmed  by remark \ref{re-10}, \ref{re-7}, \ref{re-11}, where we show that in any case,
 $\lambda$ is independent of $s_1,s_2,d_1,d_2$.   \label{re-12}
\end{remark}
\section{Application of the algorithm}\label{se-5}

 The main purpose of this section is to illustrate the application of  the algorithm described in the   Section \ref{se-4}.
 To this end, we will calculate
    all the {\itshape PSQHPDS} of degree $2$ and $3$ by employing our algorithm.

\subsection{The case $n=2$}

 By   Algorithm Part \ref{al-4}, we   find that this part can  only provides one {\itshape PSQHPDS} of degree $2$, i.e., $X_{0,0,0}$:
\begin{equation}\label{81}
  \begin{split}
  &\dot{x}=a_{0, 2}y^{2}+a_{1, 0}x, \\
  &\dot{y}=b_{0, 2}y^{2}+b_{1, 0}x, \\
  \end{split}
\end{equation}
with $w=(2s_{2}, s_{2}, 1, s_{2}+1)$ and thus $w_{m}=(2, 1, 1, 2)$.

 In Algorithm Part \ref{al-1}, we have $t=0$, $p=0$, $\triangle t=1$, $k=1$. This means that the equations set is very simple: $\varepsilon_{0}^{0}=\{e_{0}^{0, 0}[0]\}$. Hence, Algorithm Part \ref{al-1} is ended.

 In order to apply Algorithm Part \ref{al-2}, we  introduce the following matrix:
\begin{equation*}
{\displaystyle\tilde{R}(n, \tilde{t}, \triangle\tilde{t}, k)}=\begin{pmatrix}
  (t, p)&I_{t, p, \triangle \tilde{t}, k}\\
  (0, 0)&I_{0, 0, \triangle \tilde{t}, k}\\
  (0, 1)&I_{0, 1, \triangle \tilde{t}, k}\\
  . . . &. . . \\
   (0, n)&I_{0, n, \triangle \tilde{t}, k}\\
   (t_{i}, p_{i})& I_{t_{i}, p_{i}, \triangle \tilde{t}, k}
  \end{pmatrix},
\end{equation*}
where, $(t_{i}, p_{i})\in I$ (definition of $I$ see \eqref{89}), $I_{t_{i}, p_{i}, \triangle \tilde{t}, k}=\{(t^{*}, p^{*}):(p_{i}-p^{*})\triangle \tilde{t}=(t_{i}-t^{*})k, t^{*}\in\{0, . . . , n-1\}\backslash\{t\}\;and\;p^{*}\in\{0, . . . , n-t^{*}\}\}$.

 For Algorithm Part \ref{al-2}, since we can  only choose $\tilde{t}=0$, $q=0$ and $\triangle \tilde{t}=1$,  then the set $I_{\tilde{t}, q}=I_{0, 0}=I_{0, 0}(1)=\{(1, 1)\}$ and $I=I(0)\bigcup I(1)=\{(0, 0), (0, 1), (0, 2), (1, 0), (1, 1)\}$. In \emph{step 1}, we choose $(\triangle \tilde{t}, k)=(1, 1)\in I_{0, 0}$. We can write the matrix $\tilde{R}(n, \tilde{t}, \triangle\tilde{t}, k)$ as
\begin{equation}\label{129}
\tilde{R}(2, 0, 1, 1)=
  \begin{psmallmatrix}
  (t, p)&I_{t, p, 1, 1}\\
  (0, 0)&I_{0, 0, 1, 1}=\{(t^{*}, p^{*}):&p^{*}=t^{*}, t^{*}\in\{0, 1\}\backslash\{0\}\;and\;p^{*}\in\{0, . . . , 2-t^{*}\}\}=\{(1, 1)\}\\
  (0, 1)&I_{0, 1, 1, 1}=\{(t^{*}, p^{*}):&p^{*}-1=t^{*}, t^{*}\in\{0, 1\}\backslash\{0\}\;and\;p^{*}\in\{0, . . . , 2-t^{*}\}\}=\varnothing\\
  (0, 2)&I_{0, 2, 1, 1}=\{(t^{*}, p^{*}):&p^{*}-2=t^{*}, t^{*}\in\{0, 1\}\backslash\{0\}\;and\;p^{*}\in\{0, . . . , 2-t^{*}\}=\varnothing\\
  (1, 0)&I_{1, 0, 1, 1}=\{(t^{*}, p^{*}):&p^{*}=t^{*}-1, t^{*}\in\{0, 1\}\backslash\{1\}\;and\;p^{*}\in\{0, . . . , 2-t^{*}\}\}=\varnothing\\
  (1, 1)&I_{1, 1, 1, 1}=\{(t^{*}, p^{*}):&p^{*}=t^{*}, t^{*}\in\{0, 1\}\backslash\{1\}\;and\;p^{*}\in\{0, . . . , 2-t^{*}\}\}=\{(0, 0)\}
  \end{psmallmatrix}.
\end{equation}
In \emph{step 2}, by \eqref{129}, we have $S_{1, 1}=\{(0, 1), (0, 2)\}$.
By \emph{step 3 } and \emph{4} of Algorithm Part \ref{al-2}, we can get the equations set
   \begin{equation}\label{87}
    \tau_{0}^{0, 1, 1}=\{\pi_{0}^{0, 0}[0], \pi_{0}^{0, 1}[1]\}.
   \end{equation}
 In \emph{step 5}, the equations set $\tau_{0}^{0, 1, 1}$ determines the  semi-quasi homogeneous polynomial
 \begin{equation}\label{88}
   (0, Q_{0, 0, 1, 1})=B_{2}+B_{1}^{1, 1}=(0, b_{0, 2}y^{2}+b_{1, 0}x),
 \end{equation}
 that is,   $Q_{0, 0, 1, 1}=b_{0, 2}y^{2}+b_{1, 0}x$.

 Since each ordered pair $(t, p)\in S_{\triangle \tilde{t}, k}$ determines a vector field $A_{n-t}$, we have
 \begin{align*}
&(t, p)=(0, 1):\mspace{125mu}(a_{1, 1}xy, 0);\\
&(t, p)=(0, 2):\mspace{125mu}(a_{2, 0}x^{2}, 0).
 \end{align*}
Therefore,  its corresponding semi-quasi homogeneous   vector field $X_{t, 0, p, 0, 1, 1}$ is\\
\begin{equation}\label{128}
\begin{split}
&X_{0, 0, 1, 0, 1, 1}:\dot{x}=a_{1, 1}xy, \dot{y}=b_{0, 2}y^{2}+b_{1, 0}x, \\
&\lambda=0, w=({2\, d_{{2}}-2, d_{{2}}-1, d_{{2}}, d_{{2}}})\;with\;w_{{m}}=({2, 1, 2, 2}) \;(remove\; this\;system)\;;\\
&X_{0, 0, 2, 0, 1, 1}:\dot{x}=a_{2, 0}x^{2}, \dot{y}=b_{0, 2}y^{2}+b_{1, 0}x, \\
&\lambda=1, w=({2\, d_{{2}}-2, d_{{2}}-1, 2\, d_{{2}}-1, d_{{2}}})\;with\;w_{{m}}=({2, 1, 3, 2});\\
\end{split}
\end{equation}
respectively, where the weight vector $w$ and the minimal weight vector $w_{m}$ are given in statement $c)$ of Proposition \ref{pr-8}, and index $\lambda$ is given in Remark \ref{re-5}. In \emph{step 7}, $I_{\tilde{t}, q}=\{(1, 1)\}\backslash\{(1, 1)\}=\varnothing$,
the Algorithm Part \ref{al-2} has ended.

   Consequently,   we have the  following result.
\begin{proposition}\label{pr-15}
  A   {\itshape PSQHPDS} of degree $2$ (with $s_{1}>s_{2}$) is one of the following systems:
  {\allowdisplaybreaks\begin{align*}
&X_{0,0,0}:\dot{x}=a_{0, 2}y^{2}+a_{1, 0}x, \dot{y}=b_{0, 2}y^{2}+b_{1, 0}x, \\
&\lambda=-1, w=(2s_{2}, s_{2}, 1, s_{2}+1) \;with \;w_{m}=(2, 1, 1, 2);\\
&X_{0, 0, 2, 0, 1, 1}:\dot{x}=a_{2, 0}x^{2}, \dot{y}=b_{0, 2}y^{2}+b_{1, 0}x, \\
&\lambda=1, w=({2\, d_{{2}}-2, d_{{2}}-1, 2\, d_{{2}}-1, d_{{2}}})\;with\;w_{{m}}=({2, 1, 3, 2}).
  \end{align*}}
\end{proposition}

\subsection{The case $n=3$}
Firstly, from \eqref{80}, we obtain a   {\itshape PSQHPDS} as follows
\begin{equation}\label{91}
  \begin{split}
  &\dot{x}=a_{0, 3}y^{3}+a_{1, 0}x, \\
  &\dot{y}=b_{0, 3}y^{3}+b_{1, 0}x, \\
  \end{split}
\end{equation}
with $w=(3s_{2}, s_{2}, 1, 2s_{2}+1)$ and $w_{m}=(3, 1, 1, 3)$.
\par For   Algorithm Part \ref{al-4}, we define the following matrix:
\begin{equation*}
  R(n, t)=\begin{pmatrix}
    q&I_{t, 0, q}\\
    0&I_{t, 0, 0}\\
    1&I_{t, 0, 1}\\
    . . . &. . . \\
    n&I_{t, 0, n}
  \end{pmatrix},
\end{equation*}
where $I_{t, 0, q}=\{(\tilde{t}^{*}, q^{*}):\tilde{t}^{*}=(q^{*}-q)(n-t-1), \tilde{t}^{*}\in\{1, . . . , n-1\}\;and\;q^{*}\in\{0, . . . , n-\tilde{t}^{*}\}\}$.

 Since $n=3$, we can only choose $t=1$ in this part. We must consider the values of $q=0, 1, 2, 3$. Substituting $n=3$ and $t=1$ into the matrix $R(n, t)$ gives:
\begin{equation*}
   R(3,1)=
   \begin{psmallmatrix}
     q&I_{1, 0, q}&\\
      0& I_{1, 0, 0}=\{(\tilde{t}^{*}, q^{*}):&\tilde{t}^{*}=q^{*}, \tilde{t}^{*}\in\{1, 2\}\;and\;q^{*}\in\{0, . . . , 3-\tilde{t}^{*}\}\}=\{(1, 1)\}\\
      1&I_{1, 0, 1}=\{(\tilde{t}^{*}, q^{*}):&\tilde{t}^{*}=q^{*}-1, \tilde{t}^{*}\in\{1, 2\}\;and\;q^{*}\in\{0, . . . , 3-\tilde{t}^{*}\}\}=\{(1, 2)\}\\
      2&I_{1, 0, 2}=\{(\tilde{t}^{*}, q^{*}):&\tilde{t}^{*}=q^{*}-2, \tilde{t}^{*}\in\{1, 2\}\;and\;q^{*}\in\{0, . . . , 3-\tilde{t}^{*}\}\}=\varnothing\\
      3&I_{1, 0, 3}=\{(\tilde{t}^{*}, q^{*}):&\tilde{t}^{*}=q^{*}-3, \tilde{t}^{*}\in\{1, 2\}\;and\;q^{*}\in\{0, . . . , 3-\tilde{t}^{*}\}\}=\varnothing
   \end{psmallmatrix}.
\end{equation*}
In \emph{step 2}, for $q=0$,  the set $I_{1, 0, 0}$ determines the semi-quasi homogeneous polynomial
\begin{equation*}
 (0, Q_{1, 0, 0})=B_{3}^{0, 0}+B_{2}^{1, 1}=(0, b_{0, 3}y^{3}+b_{1, 1}xy),
\end{equation*}
that is, $Q_{1, 0, 0}=b_{0, 3}y^{3}+b_{1, 1}xy$.
\par Therefore,  the  corresponding semi-quasi homogeneous   vector field $X_{1, 0, 0}$ is
\begin{equation}\label{96}
  X_{1, 0, 0}=(a_{0, 2}y^{2}+a_{1, 0}x, Q_{1, 0, 0})=(a_{0, 2}y^{2}+a_{1, 0}x, b_{0, 3}y^{3}+b_{1, 1}xy).
\end{equation}
Applying Remark \ref{re-3} and statement $b)$ of Proposition \ref{pr-4}, we obtain $\lambda=-2$, $w=(2\, s_{{2}}, s_{{2}}, 1, 1+2\, s_{{2}})$ and $w_{{m}}=({2, 1, 1, 3})$.
\par In \emph{step 2}, for $q=1$, we get $(0, Q_{1, 0, 1})=B_{3}^{0, 1}+B_{2}^{1, 2}=(0, b_{1, 2}xy^{2}+b_{2, 0}x^{2})$.
\par So,  the corresponding   vector field $X_{1, 0, 1}$ is
\begin{equation}\label{97}
  X_{1, 0, 1}=(a_{0, 2}y^{2}+a_{1, 0}x, Q_{1, 0, 1})=(a_{0, 2}y^{2}+a_{1, 0}x, b_{1, 2}xy^{2}+b_{2, 0}x^{2}).
\end{equation}
By Remark \ref{re-3} and statement $b)$ of Proposition \ref{pr-4}, we get $\lambda=-3$, $w=({2\, s_{{2}}, s_{{2}}, 1, 3\, s_{{2}}+1})$ and $w_{{m}}=({2, 1, 1, 4})$.
\par In \emph{step 2}, for $q=2$, we have $(0, Q_{1, 0, 2})=B_{3}^{0, 2}=(0, b_{2, 1}x^{2}y)$.
\par So,  the corresponding   vector field $X_{1, 0, 2}$ is
\begin{equation}\label{98}
 X_{1, 0, 2}=(a_{0, 2}y^{2}+a_{1, 0}x, Q_{1, 0, 2})=(a_{0, 2}y^{2}+a_{1, 0}x, b_{2, 1}x^{2}y).
\end{equation}
Using Remark \ref{re-3} and statement $b)$ of Proposition \ref{pr-4}, we have $\lambda=-4$, $w=({2\, s_{{2}}, s_{{2}}, 1, 4\, s_{{2}}+1})$ and $w_{{m}}=({2, 1, 1, 5})$.
\par In \emph{step 2}, for  $q=3$, we obtain $(0, Q_{1, 0, 3})=B_{3}^{0, 3}=(0, b_{3, 0}x^{3})$.
\par Thus, the corresponding   vector field $X_{1, 0, 3}$ is
\begin{equation}\label{99}
  X_{1, 0, 3}=(a_{0, 2}y^{2}+a_{1, 0}x, Q_{1, 0, 3})=(a_{0, 2}y^{2}+a_{1, 0}x, b_{3, 0}x^{3}).
\end{equation}
Applying Remark \ref{re-3} and statement $b)$ of Proposition \ref{pr-4}, we obtain $\lambda=-5$, $w=(2\, s_{{2}}, s_{{2}}, 1, 5\, s_{{2}}$\\$ +1)$ and $w_{{m}}=({2, 1, 1, 6})$. Algorithm Part \ref{al-4} is ended.

 In order to apply Algorithm Part \ref{al-1} expediently, for each $t=0, . . . , n-2$ and each $p=0, . . . , n-t-2$, we construct a matrix $M(n, t, p)$ as following
\begin{equation*}
  M(n, t, p)=\begin{pmatrix}
    \triangle t&k&\varepsilon_{p}^{t}&(P_{t, p, \triangle t, k}, 0)\\
    -&-&e_{p}^{t, 0}[0]&A_{n-t}\\
    1&1&e_{p}^{t, 1}[1]&A_{n-t-1}^{1, 1}\\
    . . . &. . . &. . . &. . . \\
    1&n-t-1-p&e_{p}^{t, 1}[n-t-1-p]&A_{n-t-1}^{1, n-t-1-p}\\
2&1&e_{p}^{t, 2}[1]&A_{n-t-2}^{2, 1}\\
. . . &. . . &. . . &. . . \\
2&n-t-2-p&e_{p}^{t, 2}[n-t-2-p]&A_{n-t-2}^{2, n-t-1-p}\\
 \triangle t_{i}&k_{i}&e_{p}^{t, \triangle t_{i}}[k_{i}]& A_{n-\triangle t_{i}-t}^{\triangle t_{i}, k_{i}}
  \end{pmatrix},
\end{equation*}
where, $\triangle t_{i}\in\{\triangle t:\triangle t=1, . . . , n-t-p-1\}$, $k_{i}\in\{k:k=1, . . . , n-t-\triangle t-p\}$, $\varepsilon_{p}^{t}=\{e_{p}^{t, 0}[0], e_{p}^{t, \triangle t_{i}}[k_{i}]\}$, $(P_{t, p, \triangle t, k}, 0)=\{A_{n-t}, A_{n-\triangle t_{i}-t}^{\triangle t_{i}, k_{i}}\}$. We remove from $\varepsilon_{0}^{t}$ the equation $e_{0}^{t, n-t-1}[1]$.
\par Now, we employ   Algorithm Part \ref{al-1} for the case $n=3$. In this case it suffices to  consider the values of $t=0, 1$.
\\\textbf{(i)} If chooses $t=0$, we must take the values of $p=0, 1$. \\
\;\textbf{(i.1)} For $p=0$, we can write the matrix $M(n, t, p)$ as
\begin{equation*}
 M(3, 0, 0)=\begin{pmatrix}
    \triangle t&k&\varepsilon_{0}^{0}&&(P_{0, 0, \triangle t, k}, 0)\\
    -&-&e_{0}^{0, 0}[0]:&-s_{1}+3s_{2}=d_{1}-1&A_{3}=(a_{0, 3}y^{3}, 0)\\
    1&1&e_{0}^{0, 1}[1]:&s_{2}=d_{1}-1&A_{2}^{1, 1}=(a_{1, 1}xy, 0)\\
1&2&e_{0}^{0, 1}[2]:&s_{1}=d_{1}-1&A_{2}^{1, 2}=(a_{2, 0}x^2, 0)
  \end{pmatrix}.
\end{equation*}
In \emph{step 1}, we take   $e_{0}^{0, 0}[0]$  as the first equation. For \emph{step 2}, we choose   $e_{0}^{0, 1}[1]$ as the second equation from the matrix $ M(3, 0, 0)$, i.e.,  $ \triangle t=1$ and $k=1$. In \emph{step 3} we cannot choose the equation $e_{0}^{0, 1}[2]$ because it does not satisfy   equation \eqref{51}. By \emph{step 4} we have $\varepsilon_{0}^{0, 1, 1}=\{e_{0}^{0, 0}[0], e_{0}^{0, 1}[1]\}$. From   equation \eqref{57}, we obtain that
\begin{equation}\label{100}
  (P_{0, 0, 1, 1}, 0)=A_{3}+A_{2}^{1, 1}=(a_{0, 3}y^{3}+a_{1, 1}xy, 0).
\end{equation}
\par In \emph{step 5}, we define the following matrix:
\begin{equation*}
 R(n, t, \triangle t, k)= \begin{pmatrix}
  (\tilde{t}, q)&I_{\tilde{t}, q, \triangle t, k}&(0, Q_{\tilde{t}, q, \triangle t, k})\\
  (0, 0)&I_{0, 0, \triangle t, k}&(0, Q_{0, 0, \triangle t, k})\\
  (0, 1)&I_{0, 1, \triangle t, k}&(0, Q_{0, 1, \triangle t, k})\\
  . . . &. . . &. . . \\
  (0, n)&I_{0, n, \triangle t, k}&(0, Q_{0, n, \triangle t, k})\\
  (\tilde{t}_{i}, q_{i})&I_{\tilde{t}_{i}, q_{i}, \triangle t, k}&(0, Q_{\tilde{t}_{i}, q_{i}, \triangle t, k})
  \end{pmatrix},
\end{equation*}
where, $(\tilde{t}_{i}, q_{i})\in\tilde{I}$ (the definition of $\tilde{I}$ can be seen in  \eqref{101}), $I_{\tilde{t}_{i}, q_{i}, \triangle t, k}=\{(\tilde{t}^{*}, q^{*}):(q_{i}-q^{*})\triangle t=(\tilde{t}_{i}-\tilde{t}^{*}\;)k, \tilde{t}^{*}\in\{0, . . . , n-1\}\backslash\{\;\tilde{t}_{i}\;\}\;and\;q^{*}\in\{0, . . . , n-\tilde{t}^{*}\}\}$.
\par Since $t=0$,   the set $\tilde{I}=\tilde{I}(0)\bigcup\tilde{I}(1)\bigcup\tilde{I}(2)=\{(0, 0), (0, 1), (0, 2), (0, 3), (1, 0), (1, 1),\\(1, 2), (2, 0), (2, 1)\}$. Substituting  $t=0$, $\triangle t=1$ and $k=1$ into the matrix  $R(n, t, \triangle t, k)$ gives:
{\allowdisplaybreaks\begin{eqnarray}\label{102}
  R(3,0,1,1)=\begin{psmallmatrix}
    (\tilde{t}, q)&I_{\tilde{t}, q, 1, 1}&(0, Q_{\tilde{t}, q, 1, 1})\\
    (0, 0)&I_{0, 0, 1, 1}=\{(1, 1)\}&(0, Q_{0, 0, 1, 1})\\
    (0, 1)&I_{0, 1, 1, 1}=\{(1,  2)\}&(0, Q_{0, 1, 1, 1})\\
    (0, 2)&I_{0, 2, 1, 1}=\varnothing&(0, Q_{0, 2, 1, 1})\\
    (0, 3)&I_{0, 3, 1, 1}=\varnothing&(0, Q_{0, 3, 1, 1})\\
    (1, 0)&I_{1, 0, 1, 1}=\{(2,  1)\}&(0, Q_{1, 0, 1, 1})\\
    (1, 1)&I_{1, 1, 1, 1}=\{(0,  0)\}&(0, Q_{1, 1, 1, 1})\\
    (1, 2)&I_{1, 2, 1, 1}=\{(0,  1)\}&(0, Q_{1, 2, 1, 1})\\
    (2, 0)&I_{2, 0, 1, 1}=\varnothing&(0, Q_{2, 0, 1, 1})\\
    (2, 1)&I_{2, 1, 1, 1}=\{(1,  0)\}&(0, Q_{2, 1, 1, 1})
  \end{psmallmatrix}.
\end{eqnarray}}
Doing \emph{step 6}, we obtain the new matrix
\begin{eqnarray}\label{130}
  \begin{pmatrix}
    (\tilde{t}, q)&I_{\tilde{t}, q, 1, 1}&(0, Q_{\tilde{t}, q, 1, 1})\\
    (0, 0)&I_{0, 0, 1, 1}=\{(1, 1)\}&(0, Q_{0, 0, 1, 1})\\
    (0, 1)&I_{0, 1, 1, 1}=\{(1,  2)\}&(0, Q_{0, 1, 1, 1})\\
    (0, 2)&I_{0, 2, 1, 1}=\varnothing&(0, Q_{0, 2, 1, 1})\\
    (0, 3)&I_{0, 3, 1, 1}=\varnothing&(0, Q_{0, 3, 1, 1})\\
    (1, 0)&I_{1, 0, 1, 1}=\{(2,  1)\}&(0, Q_{1, 0, 1, 1})\\
    (2, 0)&I_{2, 0, 1, 1}=\varnothing&(0, Q_{2, 0, 1, 1})\\
  \end{pmatrix}.
\end{eqnarray}
In \emph{step 7}, by   equation \eqref{59}, we have
\begin{equation}\label{108}
   \begin{split}
     &Q_{{0, 0, 1, 1}}=b_{{0, 3}}{y}^{3}+b_{{1, 1}}xy;\\
     &Q_{{0, 1, 1, 1}}=b_{{1, 2}}x{y}^{2}+b_{{2, 0}}{x}^{2};\\
     &Q_{{0, 2, 1, 1}}=b_{{2, 1}}{x}^{2}y;\\
     &Q_{{0, 3, 1, 1}}=b_{{3, 0}}{x}^{3};\\
     &Q_{{1, 0, 1, 1}}=b_{{0, 2}}{y}^{2}+b_{{1, 0}}x;\\
     &Q_{{2, 0, 1, 1}}=b_{{0, 1}}y.
   \end{split}
\end{equation}
Therefore,  together with \eqref{100}, the  corresponding semi-quasi homogeneous   vector field \\$X_{0, \tilde{t}, 0, q, 1, 1}$ is
\begin{equation}\label{104}
\begin{split}
 &{\it X}_{{0, 0, 0, 0, 1, 1}}:\dot{x}=a_{0, 3}y^{3}+a_{1, 1}xy, \dot{y}=b_{{0, 3}}{y}^{3}+b_{{1, 1}}xy, \\
 &\lambda=-1, w=({2\, d_{{1}}-2, d_{{1}}-1, d_{{1}}, 2\, d_{{1}}-1})\;with\;w_{{m}}=({2, 1, 2, 3});\\
 &{\it X}_{{0, 0, 0, 1, 1, 1}}:\dot{x}=a_{0, 3}y^{3}+a_{1, 1}xy, \dot{y}=b_{{1, 2}}x{y}^{2}+b_{{2, 0}}{x}^{2}, \\
 &\lambda=-2, w=({2\, d_{{1}}-2, d_{{1}}-1, d_{{1}}, 3\, d_{{1}}-2})\;with\;w_{{m}}=({2, 1, 2, 4});\\
 &{\it X}_{{0, 0, 0, 2, 1, 1}}: \dot{x}=a_{0, 3}y^{3}+a_{1, 1}xy, \dot{y}=b_{{2, 1}}{x}^{2}y, \\
 &\lambda=-3, w=({2\, d_{{1}}-2, d_{{1}}-1, d_{{1}}, 4\, d_{{1}}-3})\;with\;w_{{m}}=({2, 1, 2, 5});\\
 &{\it X}_{{0, 0, 0, 3, 1, 1}}: \dot{x}=a_{0, 3}y^{3}+a_{1, 1}xy, \dot{y}=b_{{3, 0}}{x}^{3}, \\
 &\lambda=-4, w=({2\, d_{{1}}-2, d_{{1}}-1, d_{{1}}, 5\, d_{{1}}-4})\;with\;w_{{m}}=({2, 1, 2, 6});\\
 &{\it X}_{{0, 1, 0, 0, 1, 1}}:\dot{x}=a_{0, 3}y^{3}+a_{1, 1}xy, \dot{y}=b_{{0, 2}}{y}^{2}+b_{{1, 0}}x, \\
 &\lambda=0, w=({2\, d_{{1}}-2, d_{{1}}-1, d_{{1}}, d_{{1}}})\;with\;w_{{m}}=({2, 1, 2, 2})\;(remove\; this\; system)\;;\\
 &{\it X}_{{0, 2, 0, 0, 1, 1}}:\dot{x}=a_{0, 3}y^{3}+a_{1, 1}xy, \dot{y}=b_{{0, 1}}y, \\
 &\lambda=1, w=({2\, d_{{1}}-2, d_{{1}}-1, d_{{1}}, 1})\;with\;w_{{m}}=({2, 1, 2, 1});
\end{split}
\end{equation}
 respectively, where  the weight vector $w$ and  minimal weight vector $w_{m}$ are given in statement $c)$ of Proposition \ref{pr-4}, and
 the index $\lambda$ is given in Remark \ref{re-3}.
 \par In \emph{step 8} we remove from $\varepsilon_{0}^{0}$ the equations $\varepsilon_{0}^{0, 1, 1}\backslash\{e_{0}^{0, 0}[0]\}$. Hence the matrix $ M(3, 0, 0)$ is changed  into
 \begin{equation*}
 \begin{pmatrix}
    \triangle t&k&\varepsilon_{0}^{0}&&(P_{0, 0, \triangle t, k}, 0)\\
    -&-&e_{0}^{0, 0}[0]:&-s_{1}+3s_{2}=d_{1}-1&A_{3}=(a_{0, 3}y^{3}, 0)\\
    1&2&e_{0}^{0, 1}[2]:&s_{1}=d_{1}-1&A_{2}^{1, 2}=(a_{2, 0}x^2, 0)
  \end{pmatrix}.
 \end{equation*}
   We begin to choose   $e_{0}^{0, 0}[0]$ as the first equation in \emph{step 1}. In \emph{step 2} we can only take $\triangle t=1$ and $k=2$, that is, the equation $e_{0}^{0, 1}[2]$. No other equation can be chose in \emph{step 3}. In \emph{step 4} we have $\varepsilon_{0}^{0, 1, 2}=\{e_{0}^{0, 0}[0], e_{0}^{0, 1}[2]\}$. Using  equation \eqref{57}, we get that
   \begin{equation}\label{105}
     (P_{0, 0, 1, 2}, 0)=A_{3}+A_{2}^{1, 2}=(a_{0, 3}y^{3}+a_{2, 0}x^2, 0).
   \end{equation}
  \par In \emph{step 5},  we calculate the  matrix $ R(3, 0, 1, 2)$ as following
\begin{equation}\label{103}
  R(3, 0, 1, 2)=\begin{pmatrix}
    (\tilde{t}, q)&I_{\tilde{t}, q, 1, 2}&(0, Q_{\tilde{t}, q, 1, 2})\\
    (0, 0)&I_{0, 0, 1, 2}=\{(1,  2)\}&(0, Q_{0, 0, 1, 2})\\
    (0, 1)&I_{0, 1, 1, 2}=\varnothing&(0, Q_{0, 1, 1, 2})\\
    (0, 2)&I_{0, 2, 1, 2}=\varnothing&(0, Q_{0, 2, 1, 2})\\
    (0, 3)&I_{0, 3, 1, 2}=\varnothing&(0, Q_{0, 3, 1, 2})\\
    (1, 0)&I_{1, 0, 1, 2}=\varnothing&(0, Q_{1, 0, 1, 2})\\
    (1, 1)&I_{1, 1, 1, 2}=\varnothing&(0, Q_{1, 1, 1, 2})\\
    (1, 2)&I_{1, 2, 1, 2}=\{(0,  0)\}&(0, Q_{1, 2, 1, 2})\\
    (2, 0)&I_{2, 0, 1, 2}=\varnothing&(0, Q_{2, 0, 1, 2})\\
    (2, 1)&I_{2, 1, 1, 2}=\varnothing&(0, Q_{2, 1, 1, 2})
  \end{pmatrix}.
\end{equation}
Doing \emph{step 6}, the matrix $R(3, 0, 1, 2)$ changes into
\begin{equation}\label{131}
  R(3, 0, 1, 2)=\begin{pmatrix}
    (\tilde{t}, q)&I_{\tilde{t}, q, 1, 2}&(0, Q_{\tilde{t}, q, 1, 2})\\
    (0, 0)&I_{0, 0, 1, 2}=\{(1,  2)\}&(0, Q_{0, 0, 1, 2})\\
    (0, 1)&I_{0, 1, 1, 2}=\varnothing&(0, Q_{0, 1, 1, 2})\\
    (0, 2)&I_{0, 2, 1, 2}=\varnothing&(0, Q_{0, 2, 1, 2})\\
    (0, 3)&I_{0, 3, 1, 2}=\varnothing&(0, Q_{0, 3, 1, 2})\\
    (1, 0)&I_{1, 0, 1, 2}=\varnothing&(0, Q_{1, 0, 1, 2})\\
    (1, 1)&I_{1, 1, 1, 2}=\varnothing&(0, Q_{1, 1, 1, 2})\\
    (2, 0)&I_{2, 0, 1, 2}=\varnothing&(0, Q_{2, 0, 1, 2})\\
    (2, 1)&I_{2, 1, 1, 2}=\varnothing&(0, Q_{2, 1, 1, 2})
  \end{pmatrix}.
\end{equation}
From   equation \eqref{59}, we obtain that
\begin{align*}
   &Q_{{0, 0, 1, 2}}=b_{{0, 3}}{y}^{3}+b_{{2, 0}}{x}^{2};\\
    &Q_{{0, 1, 1, 2}}=b_{{1, 2}}x{y}^{2};\\
     &Q_{{0, 2, 1, 2}}=b_{{2, 1}}{x}^{2}y;\\
     &Q_{{0, 3, 1, 2}}=b_{{3, 0}}{x}^{3};\\
     &Q_{{1, 0, 1, 2}}=b_{{0, 2}}{y}^{2};\\
     &Q_{{1, 1, 1, 2}}=b_{{1, 1}}xy;\\
     &Q_{{2, 0, 1, 2}}=b_{{0, 1}}y;\\
     &Q_{{2, 1, 1, 2}}=b_{{1, 0}}x.
\end{align*}
So,  together with \eqref{105}, the  corresponding semi-quasi homogeneous vector field  $X_{0, \tilde{t}, 0, q, 1, 2}$ is
\begin{equation}\label{106}
  \begin{split}
&X_{{0, 0, 0, 0, 1, 2}}:\dot{x}=a_{0, 3}y^{3}+a_{2, 0}x^2;\dot{y}=b_{{0, 3}}{y}^{3}+b_{{2, 0}}{x}^{2}, \\
&\lambda=-1, w=(d_{{1}}-1, 2(d_1-1)/3, d_{{1}},(4d_1-1)/3)\;with\;w_{{m}}=({3, 2, 4, 5}), \\
&X_{{0, 0, 0, 1, 1, 2}}:\dot{x}=a_{0, 3}y^{3}+a_{2, 0}x^2;\dot{y}=b_{{1, 2}}x{y}^{2}, \\
&\lambda=-2, w=({d_{{1}}-1, 2(d_1-1)/3, d_{{1}}, (5d_1-2)/3}), \;with\;w_{{m}}=({3, 2, 4, 6}), \\
&X_{{0, 0, 0, 2, 1, 2}}:\dot{x}=a_{0, 3}y^{3}+a_{2, 0}x^2;\dot{y}=b_{{2, 1}}{x}^{2}y, \\
&\lambda=-3, w=({d_{{1}}-1, 2(d_1-1)/3, d_{{1}}, 2\, d_{{1}}-1})\;with\;w_{{m}}=({3, 2, 4, 7});\\
&X_{{0, 0, 0, 3, 1, 2}}:\dot{x}=a_{0, 3}y^{3}+a_{2, 0}x^2;\dot{y}=b_{{3, 0}}{x}^{3}, \\
&\lambda=-4, w=({d_{{1}}-1, 2(d_1-1)/3, d_{{1}}, (7d_1-4)/3})\;with\;w_{{m}}=({3, 2, 4, 8});\\
&X_{{0, 1, 0, 0, 1, 2}}:\dot{x}=a_{0, 3}y^{3}+a_{2, 0}x^2;\dot{y}=b_{{0, 2}}{y}^{2}, \\
&\lambda=1, w=({d_{{1}}-1, 2(d_1-1)/3, d_{{1}}, (2d_1+1)/3})\;with\;w_{{m}}=({3, 2, 4, 3});\\
&X_{{0, 1, 0, 1, 1, 2}}:\dot{x}=a_{0, 3}y^{3}+a_{2, 0}x^2;\dot{y}=b_{{1, 1}}xy, \\
&\lambda=0, w=({d_{{1}}-1, 2(d_1-1)/3\, d_{{1}}, d_{{1}}, d_{{1}}})\;with\;w_{{m}}=({3, 2, 4, 4})\;(remove \; this \; system)\;;\\
&X_{{0, 2, 0, 0, 1, 2}}:\dot{x}=a_{0, 3}y^{3}+a_{2, 0}x^2;\dot{y}=b_{{0, 1}}y, \\
&\lambda=3, w=({d_{{1}}-1, 2(d_1-1)/3, d_{{1}}, 1})\;with\;w_{{m}}=({3, 2, 4, 1});\\
&X_{{0, 2, 0, 1, 1, 2}}:\dot{x}=a_{0, 3}y^{3}+a_{2, 0}x^2;\dot{y}=b_{{1, 0}}x, \\
&\lambda=2, w=({d_{{1}}-1, 2(d_1-1)/3, d_{{1}}, (d_1+2)/3})\;with\;w_{{m}}=({3, 2, 4, 2});\\
  \end{split}
\end{equation}
respectively, where  weight vector $w$ and  minimal weight vector $w_{m}$ are given in statement $c)$ of Proposition \ref{pr-4}, and the index $\lambda$ is given in Remark \ref{re-3}.

 In the end we remove from $\varepsilon_{0}^{0}$ the equations $\varepsilon_{0}^{0, 1, 2}\backslash \{e_{0}^{0, 0}[0]\}$. This implies
that $\varepsilon_{0}^{0}=\{e_{0}^{0, 0}[0]\}$.  The  Algorithm Part \ref{al-1} has finished for $p=0$. \\
\textbf{(i.2)} $p=1$. Substituting  $n=3$, $t=0$ and $p=1$ into the matrix  $M(n, t, p)$ gives:
\begin{equation*}
  M(3, 0, 1)=\begin{pmatrix}
    \triangle t&k&\varepsilon_{1}^{0}&&(P_{0, 1, \triangle t, k}, 0)\\
    -&-&e_{1}^{0, 0}[0]:&2s_{2}=d_{1}-1&A_{3}=(a_{1, 2}xy^{2}, 0)\\
    1&1&e_{1}^{0, 1}[1]:&s_{1}=d_{1}-1&A_{2}^{1, 1}=(a_{2, 0}x^{2}, 0)
 \end{pmatrix}.
\end{equation*}
\par We start choosing the equation $e_{1}^{0, 0}[0]$ in \emph{step 1}. In \emph {step 2} we can only select $\triangle t=1$ and $k=1$, i.e.,  the equation $e_{1}^{0, 1}[1]$. For \emph{step 3}, no other equation can be chosen. By \emph{step 4} we have $\varepsilon_{1}^{0, 1, 1}=\{e_{1}^{0, 0}[0], e_{1}^{0, 1}[1]\}$. Applying    equation \eqref{57}, we have that
\begin{equation}\label{107}
  (P_{0, 1, 1, 1}, 0)=A_{3}+A_{2}^{1, 1}=(a_{1, 2}xy^{2}+a_{2, 0}x^{2}, 0).
\end{equation}
\par In \emph{step 5}, substituting  $t=0$, $\triangle t=1$ and $k=1$ into the matrix  $R(n, t, \triangle t, k)$, we can find that the matrix  $R(n, t, \triangle t, k)$ is the same as \eqref{102}. From   equation \eqref{59}, we also have \eqref{108}.
\par Therefore,  together with \eqref{107}, the corresponding semi-quasi homogeneous  vector field $X_{0, \tilde{t}, 1, q, 1, 1}$ is
\begin{equation}\label{109}
  \begin{split}
 &X_{{0, 0, 1, 0, 1, 1}}:\dot{x}=a_{1, 2}xy^{2}+a_{2, 0}x^{2}, \dot{y}=b_{{0, 3}}{y}^{3}+b_{{1, 1}}xy, \\
 &\lambda=0, w=({d_{{1}}-1, (d_1-1)/2, d_{{1}}, d_{{1}}})\;with\;w_{{m}}=({2, 1, 3, 3})\;(remove \; this \; system)\;;\\
 &X_{{0, 0, 1, 1, 1, 1}}:\dot{x}=a_{1, 2}xy^{2}+a_{2, 0}x^{2}, \dot{y}=b_{{1, 2}}x{y}^{2}+b_{{2, 0}}{x}^{2}, \\
 &\lambda=-1, w=({d_{{1}}-1, (d_1-1)/2, d_{{1}}, (3d_1-1)/2}), \;with\;w_{{m}}=({2, 1, 3, 4});\\
 &X_{{0, 0, 1, 2, 1, 1}}: \dot{x}=a_{1, 2}xy^{2}+a_{2, 0}x^{2}, \dot{y}=b_{{2, 1}}{x}^{2}y, \\
 &\lambda=-2, w=({d_{{1}}-1, (d_1-1)/2, d_{{1}}, 2\, d_{{1}}-1})\;with\;w_{{m}}=({2, 1, 3, 5});\\
 &X_{{0, 0, 1, 3, 1, 1}}: \dot{x}=a_{1, 2}xy^{2}+a_{2, 0}x^{2}, \dot{y}=b_{{3, 0}}{x}^{3}, \\
 &\lambda=-3, w=({d_{{1}}-1, (d_1-1)/2, d_{{1}}, (5d_1-3)/2})\;with\;w_{{m}}=({2, 1, 3, 6});\\
 &X_{{0, 1, 1, 0, 1, 1}}:\dot{x}=a_{1, 2}xy^{2}+a_{2, 0}x^{2}, \dot{y}=b_{{0, 2}}{y}^{2}+b_{{1, 0}}x, \\
 &\lambda=1, w=({d_{{1}}-1, (d_1-1)/2, d_{{1}}, (d_1+1)/2})\;with\;w_{{m}}=({2, 1, 3, 2});\\
 &X_{{0, 2, 1, 0, 1, 1}}:\dot{x}=a_{1, 2}xy^{2}+a_{2, 0}x^{2}, \dot{y}=b_{{0, 1}}y, \\
 &\lambda=2, w=({d_{{1}}-1, (d_1-1)/2, d_{{1}}, 1})\;with\;w_{{m}}=({2, 1, 3, 1});
\end{split}
\end{equation}
respectively, where  the weight vector $w$ and  the minimal weight vector $w_{m}$ are given in statement $c)$ of Proposition \ref{pr-4}, and the index $\lambda$ is given in Remark \ref{re-3}. Finally, we remove from $\varepsilon_{1}^{0}$ the equations $\varepsilon_{1}^{0, 1, 1}\backslash\{e_{1}^{0, 0}[0]\}$. This means that $\varepsilon_{1}^{0}=\{e_{1}^{0, 0}[0]\}$. The Algorithm Part \ref{al-1} has ended for $p=1$.
\\\textbf{(ii)} If chooses $t=1$, we can only choose $p=0$ in this case. One can easily compute that $\varepsilon_{0}^{1}=\{e_{0}^{1, 0}[0]\}$. In \emph{step 9} the Algorithm Part \ref{al-1} has finished for $t = 1$.

 Now we apply   Algorithm Part \ref{al-2} to get all  {\itshape PSQHPDS}
 of degree $3$ with $|A|=1$ and $|B|\geq2$. In this case we must take the values of $\tilde{t} = 0, 1$. \\
\textbf{(i)} If chooses $\tilde{t}=0$, then  $q$ can be $0$ and $1$. \\
\;\textbf{(i.1)} $q=0$. We have
$$I_{\tilde{t}, q}=I_{0, 0}=\{(1,1), (1, 2), (2, 1)\}$$
and
$$I=\{(0,0), (0,1), (0,2), (0,3),(1,0),(1,1),(1,2), (2,0), ( 2,1)\}.$$
In \emph{step 1} we choose $(\triangle t, k)=(1, 1)\in I_{0, 0}$. We can write the matrix $\tilde{R}(n, \tilde{t}, \triangle\tilde{t}, k)$ as
\begin{equation}\label{110}
  \tilde{R}(3, 0, 1, 1)=\begin{pmatrix}
  (t, p)&I_{t, p, 1, 1}\\
  (0, 0)&I_{0, 0, 1, 1}=\{(1, 1)\}\\
    (0, 1)&I_{0, 1, 1, 1}=\{(1,  2)\}\\
    (0, 2)&I_{0, 2, 1, 1}=\varnothing\\
    (0, 3)&I_{0, 3, 1, 1}=\varnothing\\
    (1, 0)&I_{1, 0, 1, 1}=\{(2,  1)\}\\
    (1, 1)&I_{1, 1, 1, 1}=\{(0,  0)\}\\
    (1, 2)&I_{1, 2, 1, 1}=\{(0,  1)\}\\
    (2, 0)&I_{2, 0, 1, 1}=\varnothing\\
    (2, 1)&I_{2, 1, 1, 1}=\{(1,  0)\}
  \end{pmatrix}.
\end{equation}
In \emph{step 2} we have $S_{1, 1}=\{(0, 2), (0, 3)\}$. By \emph{step 3} and \emph{ 4} of  Algorithm Part \ref{al-2}, we can get the equations set
\begin{equation}\label{111}
  \tau_{0}^{0, 1, 1}=\{\pi_{0}^{0, 0}[0], \pi_{0}^{0, 1}[1]\}.
\end{equation}
In \emph{step 5} the equations set $\tau_{0}^{0, 1, 1}$ determines a  semi-quasi homogeneous polynomial
 \begin{equation}\label{112}
   (0, Q_{0, 0, 1, 1})=B_{3}+B_{2}^{1, 1}=(0, b_{0, 3}y^{3}+b_{1, 1}xy),
 \end{equation}
 that is,  $Q_{0, 0, 1, 1}=b_{0, 3}y^{3}+b_{1, 1}xy$.
 \par Since each ordered pair $(t, p)\in S_{\triangle \tilde{t}, k}$ determines a vector field $A_{n-t}$, we have\\
 \begin{equation}\label{123}
 \begin{split}
 &(t, p)=(0,  2):\mspace{125mu}(a_{{2, 1}}{x}^{2}y, 0);\\
&(t, p)=(0, 3):\mspace{125mu}(a_{{3, 0}}{x}^{3}, 0).
 \end{split}
 \end{equation}
Therefore,  the  corresponding semi-quasi homogeneous  vector field $X_{t, 0, p, 0, 1, 1}$ is
\begin{equation}\label{124}
  \begin{split}
  &X_{{0, 0, 2, 0, 1, 1}}:\dot{x}=a_{{2, 1}}{x}^{2}y, \dot{y}=b_{0, 3}y^{3}+b_{1, 1}xy, \\
  &\lambda=1, w=({d_{{2}}-1, 1/2\, d_{{2}}-1/2, 3/2\, d_{{2}}-1/2, d_{{2}}})\;with\;w_{{m}}=({2, 1, 4, 3});\\
  &X_{{0, 0, 3, 0, 1, 1}}:\dot{x}=a_{{3, 0}}{x}^{3}, \dot{y}=b_{0, 3}y^{3}+b_{1, 1}xy, \\
  &\lambda=2, w=({d_{{2}}-1, 1/2\, d_{{2}}-1/2, 2\, d_{{2}}-1, d_{{2}}})\;with\;w_{{m}}=({2, 1, 5, 3});
  \end{split}
\end{equation}
respectively, where the weight vector $w$ and the  minimal weight vector $w_{m}$ are given in statement $c)$ of Proposition \ref{pr-8},
 and the index $\lambda$ is given in Remark \ref{re-5}.

In \emph{step 7} we remove from $I_{0, 0}$ the set $\{(1, 1)\}$. Hence  the set $I_{0, 0}$ changes into $I_{0, 0}=\{(1,  2), (2,  1)\}$.

We go back to \emph{step 1} and select $(\triangle\tilde{t}, k)=(1,  2)\in I_{0, 0}$.

In \emph{step 2},  we calculate the  matrix $ \tilde{R}(3, 0, 1, 2)$ as following
\begin{equation}\label{113}
  \tilde{R}(3, 0, 1, 2)=\begin{psmallmatrix}
    (t, p)&I_{t, p, 1, 2}\\
    (0, 0)&I_{0, 0, 1, 2}=\{(1,  2)\}\\
    (0, 1)&I_{0, 1, 1, 2}=\varnothing\\
    (0, 2)&I_{0, 2, 1, 2}=\varnothing\\
    (0, 3)&I_{0, 3, 1, 2}=\varnothing\\
    (1, 0)&I_{1, 0, 1, 2}=\varnothing\\
    (1, 1)&I_{1, 1, 1, 2}=\varnothing\\
    (1, 2)&I_{1, 2, 1, 2}=\{(0,  0)\}\\
    (2, 0)&I_{2, 0, 1, 2}=\varnothing\\
    (2, 1)&I_{2, 1, 1, 2}=\varnothing
  \end{psmallmatrix}.
\end{equation}
So, $S_{1, 2}=\{ (0, 1), (0, 2), (0, 3), (1, 0), (1, 1), (2, 1)\}$. From \emph{step 3} and \emph{4} of Algorithm Part \ref{al-2}, we can obtain the equations set
\begin{equation}\label{114}
  \tau_{0}^{0, 1, 2}=\{\pi_{0}^{0, 0}[0], \pi_{0}^{0, 1}[2]\}.
\end{equation}
In \emph{step 5} the semi-quasi homogeneous  polynomial corresponding to equations set $\tau_{0}^{0, 1, 2}$ is
\begin{equation}\label{115}
  (0, Q_{0, 0, 1, 2})=B_{3}+B_{2}^{1, 2}=(0, b_{0, 3}y^{3}+b_{2, 0}x^{2}),
\end{equation}
i.e.,    $Q_{0, 0, 1, 2}=b_{0, 3}y^{3}+b_{2, 0}x^{2}$. Each ordered pair $(t, p)\in S_{1, 2}$ determines vector field $A_{n-t}$ as following
\begin{align*}
&(t, p)=(0,  1):\mspace{125mu}(a_{{1, 2}}x{y}^{2}, 0);\\
&(t, p)=(0,  2):\mspace{125mu}(a_{{2, 1}}{x}^{2}y, 0);\\
&(t, p)=(0,  3):\mspace{125mu}(a_{{3, 0}}{x}^{3}, 0);\\
&(t, p)=(1,  0):\mspace{125mu}(a_{{0, 2}}{y}^{2}, 0);\\
&(t, p)=(1,  1):\mspace{125mu}(a_{{1, 1}}xy, 0);\\
&(t, p)=(2,  1):\mspace{125mu}(a_{{1, 0}}x, 0).
\end{align*}
Consequently, the  corresponding semi-quasi homogeneous   vector field $X_{t, 0, p, 0, 1, 2}$ is
\begin{equation}\label{116}
  \begin{split}
   &X_{{0, 0, 1, 0, 1, 2}}:\dot{x}=a_{{1, 2}}x{y}^{2}, \dot{y}=b_{0, 3}y^{3}+b_{2, 0}x^{2}, \\
   &\lambda=0,  w=({3(d_2-1)/4, (d_2-1)/2, d_{{2}}, d_{{2}}})\;with\;w_{{m}}=({3, 2, 5, 5})\;(remove\; this\;system)\;;\\
   &X_{{0, 0, 2, 0, 1, 2}}:\dot{x}=a_{{2, 1}}{x}^{2}y, \dot{y}=b_{0, 3}y^{3}+b_{2, 0}x^{2}, \\
   &\lambda=1, w=({3(d_2-1)/4, (d_2-1)/2, (5d_2-1)/4, d_{{2}}})\;with\;w_{{m}}=({3, 2, 6, 5});\\
   &X_{{0, 0, 3, 0, 1, 2}}:\dot{x}=a_{{3, 0}}{x}^{3}, \dot{y}=b_{0, 3}y^{3}+b_{2, 0}x^{2}, \\
   &\lambda=2, w=({3(d_2-1)/4, (d_2-1)/2, (3d_2-1)/2, d_{{2}}})\;with\;w_{{m}}=({3, 2, 7, 5});\\
   &X_{{1, 0, 0, 0, 1, 2}}:\dot{x}=a_{{0, 2}}{y}^{2}, \dot{y}=b_{0, 3}y^{3}+b_{2, 0}x^{2}, \\
   &\lambda=-3, w=({3(d_2-1)/4, (d_2-1)/2, (d_2+3)/4, d_{{2}}})\;with\;w_{{m}}=({3, 2, 2, 5});\\
   &X_{{1, 0, 1, 0, 1, 2}}:\dot{x}=a_{{1, 1}}xy, \dot{y}=b_{0, 3}y^{3}+b_{2, 0}x^{2}, \\
   &\lambda=-2, w=({3(d_2-1)/4, (d_2-1)/2, (d_2+1)/2, d_{{2}}})\;with\;w_{{m}}=({3, 2, 3, 5});\\
   &X_{{2, 0, 1, 0, 1, 2}}:\dot{x}=a_{{1, 0}}x, \dot{y}=b_{0, 3}y^{3}+b_{2, 0}x^{2}, \\
   &\lambda=-4, w=({3(d_2-1)/4, (d_2-1)/2, 1, d_{{2}}})\;with\;w_{{m}}=({3, 2, 1, 5});
  \end{split}
\end{equation}
respectively, where the  weight vector $w$ and the  minimal weight vector $w_{m}$ are given in statement $c)$ of Proposition \ref{pr-8}, and
the index $\lambda$ is given in Remark \ref{re-5}.

Doing \emph{step 7}, we obtain the set $I_{0, 0}=\{(2, 1)\}$.

We once again go back to \emph{step 1} and choose $(\triangle \tilde{t}, k)=(2, 1)$.
Substituting $\tilde{t} = 0$, $\triangle\tilde{t} = 2 $ and $k = 1$ into the matrix $\tilde{R}(n, \tilde{t}, \triangle\tilde{t}, k)$ gives:
\begin{equation}\label{117}
  \tilde{R}(3, 0, 2, 1)=\begin{pmatrix}
    (t, p)&I_{t, p, 2, 1}\\
    (0, 0)&I_{0, 0, 2, 1}=\{(2,  1)\}\\
    (0, 1)&I_{0, 1, 2, 1}=\varnothing\\
    (0, 2)&I_{0, 2, 2, 1}=\varnothing\\
    (0, 3)&I_{0, 3, 2, 1}=\varnothing\\
    (1, 0)&I_{1, 0, 2, 1}=\varnothing\\
    (1, 1)&I_{1, 1, 2, 1}=\varnothing\\
    (1, 2)&I_{1, 2, 2, 1}=\varnothing\\
    (2, 0)&I_{2, 0, 2, 1}=\varnothing\\
    (2, 1)&I_{2, 1, 2, 1}=\{(0,  0)\}
  \end{pmatrix}.
\end{equation}
In \emph{step 2} we have $S_{2, 1}=\{(0, 1), (0, 2), (0, 3), (1, 1), (1, 2)\}$. By \emph{step 3} and \emph{4} of Algorithm Part \ref{al-2}, we can get the equations set
\begin{equation}\label{118}
  \tau_{0}^{0, 2, 1}=\{\pi_{0}^{0, 0}[0], \pi_{0}^{0, 2}[1]\}.
\end{equation}
In \emph{step 5} the equations set $\tau_{0}^{0, 2, 1}$ determines a semi-quasi homogeneous polynomial
\begin{equation}\label{119}
   (0, Q_{0, 0, 2, 1})=B_{3}+B_{1}^{2, 1}=(0, b_{0, 3}y^{3}+b_{1, 0}x),
 \end{equation}
 that is,  $Q_{0, 0, 2, 1}=b_{0, 3}y^{3}+b_{1, 0}x$.
\par Since each ordered pair $(t, p)\in S_{\triangle \tilde{t}, k}$ determines a vector field $A_{n-t}$, we have
\begin{align*}
&(t, p)=(0, 1):\mspace{125mu}(a_{{1, 2}}x{y}^{2}, 0);\\
&(t, p)=(0, 2):\mspace{125mu}(a_{{2, 1}}{x}^{2}y, 0);\\
&(t, p)=(0, 3):\mspace{125mu}(a_{{3, 0}}{x}^{3}, 0);\\
&(t, p)=(1, 1):\mspace{125mu}(a_{{1, 1}}xy, 0);\\
&(t, p)=(1, 2):\mspace{125mu}(a_{{2, 0}}{x}^{2}, 0).
\end{align*}
Therefore, the  corresponding semi-quasi homogeneous vector field  $X_{t, 0, p, 0, 2, 1}$ is
\begin{equation}\label{120}
\begin{split}
  &X_{{0, 0, 1, 0, 2, 1}}:\dot{x}=a_{{1, 2}}x{y}^{2}, \dot{y}=b_{0, 3}y^{3}+b_{1, 0}x, \\
  &\lambda=0, w=({3(d_2-1)/2, (d_2-1)/2, d_{{2}}, d_{{2}}})\;with\;w_{{m}}=({3, 1, 3, 3})\;(remove\; this\;system)\;;\\
  &X_{{0, 0, 2, 0, 2, 1}}:\dot{x}=a_{{2, 1}}{x}^{2}y, \dot{y}=b_{0, 3}y^{3}+b_{1, 0}x, \\
  &\lambda=1, w=({3(d_2-1)/2, (d_2-1)/2, 2\, d_{{2}}-1, d_{{2}}})\;with\;w_{{m}}=({3, 1, 5, 3});\\
  &X_{{0, 0, 3, 0, 2, 1}}:\dot{x}=a_{{3, 0}}{x}^{3}, \dot{y}=b_{0, 3}y^{3}+b_{1, 0}x, \\
  &\lambda=2, w=({3(d_2-1)/2, (d_2-1)/2, 3d_{{2}}-2, d_{{2}}})\;with\;w_{{m}}=({3, 1, 7, 3});\\
  &X_{{1, 0, 1, 0, 2, 1}}:\dot{x}=a_{{1, 1}}xy, \dot{y}=b_{0, 3}y^{3}+b_{1, 0}x, \\
  &\lambda=-1/2, w=({3(d_2-1)/2, (d_2-1)/2, (d_2+1)/2, d_{{2}}})\;with\;w_{{m}}=({3, 1, 2, 3});\\
  &X_{{1, 0, 2, 0, 2, 1}}:\dot{x}=a_{{2, 0}}{x}^{2}, \dot{y}=b_{0, 3}y^{3}+b_{1, 0}x, \\
  &\lambda=1/2, w=({3(d_2-1)/2, (d_2-1)/2, (3d_2-1)/2, d_{{2}}})\;with\;w_{{m}}=({3, 1, 4, 3});
\end{split}
\end{equation}
respectively, where the weight vector $w$ and the minimal weight vector $w_{m}$ are given in statement $c)$ of Proposition \ref{pr-8}, and the index $\lambda$ is given in Remark \ref{re-5}. In the end we remove from $I_{0, 0}$ the set $\{(2, 1)\}$. This implies that $I_{0, 0}=\varnothing$. The Algorithm Part \ref{al-2} has finished for $q=0$. \\
\textbf{(i.2)} $q=1$. In this case we have $I_{0, 1}=\{(1, 1)\}$ and
 $I=\{(0,  0), (0,  1), (0,  2), (0,  3), (1,  0),\\ (1,  1), (1,  2),(2,  0), ( 2,  1)\}$.

In \emph{step 2}, substituting  $\tilde{t}=0$, $\triangle \tilde{t}=1$ and $k=1$ into the matrix  $\tilde{R}(n, \tilde{t}, \triangle \tilde{t}, k)$, we can find that the matrix $\tilde{R}(n, \tilde{t}, \triangle \tilde{t}, k)$ is the same as \eqref{110}.  We also have $S_{1, 1}=\{(0, 2), (0, 3)\}$. From \emph{step 3} and  \emph{4} of  Algorithm Part \ref{al-2}, we can obtain the equations set
\begin{equation}\label{121}
  \tau_{1}^{0, 1, 1}=\{\pi_{1}^{0, 0}[0], \pi_{1}^{0, 1}[1]\}.
\end{equation}
In \emph{step 5} the semi-quasi homogeneous polynomial corresponding  to the equations set $\tau_{1}^{0, 1, 1}$ is
\begin{equation}\label{122}
  (0, Q_{0, 1, 1, 1})=B_{3}+B_{2}^{1, 1}=(0, b_{1, 2}xy^{2}+b_{2, 0}x^{2}),
\end{equation}
i.e.,    $Q_{0, 1, 1, 1}=b_{1, 2}xy^{2}+b_{2, 0}x^{2}$.
By \eqref{123},  the corresponding semi-quasi homogeneous  vector field $X_{t, 0, p, 1, 1, 1}$ is
\begin{equation}
  \begin{split}
  &X_{{0, 0, 2, 1, 1, 1}}:\dot{x}=a_{{2, 1}}{x}^{2}y, \dot{y}=b_{1, 2}xy^{2}+b_{2, 0}x^{2}, \\
  &\lambda=0, w=({2(d_2-1)/3, (d_2-1)/3, d_{{2}}, d_{{2}}})\;with\;w_{{m}}=({2, 1, 4, 4})\;(remove\; this\;system)\;;\\
  &X_{{0, 0, 3, 1, 1, 1}}:\dot{x}=a_{{3, 0}}{x}^{3}, \dot{y}=b_{1, 2}xy^{2}+b_{2, 0}x^{2}, \\
  &\lambda=1, w=({2(d_2-1)/3, (d_2-1)/3, (4d_2-1)/3, d_{{2}}})\;with\;w_{{m}}=({2, 1, 5, 4});\\
  \end{split}
\end{equation}
respectively, where the weight vector $w$ and the minimal weight vector $w_{m}$ are given in statement $c)$ of Proposition \ref{pr-8}, and the
index $\lambda$ is given in Remark \ref{re-5}. Finally, we remove from $I_{0, 1}$ the set $\{(1, 1)\}$ and we get $I_{0, 1}=\varnothing$. This process has ended for $q=1$.
\\\textbf{(ii)} If chooses $\tilde{t}=1$, we can only choose $q=0$ in this case. One can easily calculate that $I_{1, 0}=\{(1,  1)\}$ and $I_{0, 1}=\{(1, 1)\}$ and $I=\{(0,  0), (0,  1), (0,  2), (0,  3)\}$. In \emph{step 1} we can only select $(\triangle\tilde{t}, k)=(1, 1)\in I_{1, 0}$. Substituting $\tilde{t}=1$, $\triangle\tilde{t}=1$ and $k=1$ into the matrix $\tilde{R}(n, \tilde{t}, \triangle\tilde{t}, k)$ gives:
\begin{equation}\label{125}
  \tilde{R}(3, 1, 1, 1)=\begin{pmatrix}
  (t, p)&I_{t, p, 1, 1}\\
  (0, 0)&I_{0, 0, 1, 1}=\{(1, 1)\}\\
    (0, 1)&I_{0, 1, 1, 1}=\{(1,  2)\}\\
    (0, 2)&I_{0, 2, 1, 1}=\varnothing\\
    (0, 3)&I_{0, 3, 1, 1}=\varnothing
  \end{pmatrix}.
\end{equation}
Thus $S_{1, 1}=\{(0, 2), (0, 3)\}$. From \emph{step 3 and 4} of Algorithm Part  \ref{al-2}, we can obtain the equations set
\begin{equation}\label{114}
  \tau_{0}^{1, 1, 1}=\{\pi_{0}^{1, 0}[0], \pi_{0}^{1, 1}[1]\}.
\end{equation}
\par In \emph{step 5} the   semi-quasi homogeneous polynomial corresponding  to   equations  $\tau_{0}^{1, 1, 1}$  is
\begin{equation}\label{126}
  (0, Q_{1, 0, 1, 1})=B_{2}+B_{1}^{1, 1}=(0, b_{0, 2}y^{2}+b_{1, 0}x),
\end{equation}
i.e., $Q_{1, 0, 1, 1}=b_{0, 2}y^{2}+b_{1, 0}x$. Each ordered pair $(t, p)\in S_{1, 1}$ determines a vector field $A_{n-t}$ as following
\begin{align*}
&(t, p)=(0,  2):\mspace{125mu}(a_{{2, 1}}{x}^{2}y, 0);\\
&(t, p)=(0, 3):\mspace{125mu}(a_{{3, 0}}{x}^{3}, 0).
\end{align*}
Consequently, the  corresponding semi-quasi homogeneous  vector field $X_{t, 1, p, 0, 1, 1}$ is
\begin{equation}\label{127}
  \begin{split}
  &X_{{0, 1, 2, 0, 1, 1}}:\dot{x}=a_{{2, 1}}{x}^{2}y, \dot{y}=b_{0, 2}y^{2}+b_{1, 0}x, \\
  &\lambda=2, w=({2\, d_{{2}}-2, d_{{2}}-1, 3\, d_{{2}}-2, d_{{2}}})\;with\;w_{{m}}=({2, 1, 4, 2});\\
  &X_{{0, 1, 3, 0, 1, 1}}:\dot{x}=a_{{3, 0}}{x}^{3}, \dot{y}=b_{0, 2}y^{2}+b_{1, 0}x, \\
  &\lambda=3, w=({2\, d_{{2}}-2, d_{{2}}-1, 4\, d_{{2}}-3, d_{{2}}})\;with\;w_{{m}}=({2, 1, 5, 2});
\end{split}
\end{equation}
respectively. Finally, we remove from $I_{1, 0}$ the set $\{(1, 1)\}$ and get $I_{1, 0}=\varnothing$. The Algorithm Part \ref{al-2} has finished for $\tilde{t}=1$.

Combing the above result, we have the following conclusion.

\begin{proposition}\label{pr-16}
  A   {\itshape PSQHPDS} of degree $3$ (with $s_{1}>s_{2}$) is one of the following systems:
{\allowdisplaybreaks\begin{align*}
&X_{0,0,0}:\dot{x}=a_{0, 3}y^{3}+a_{1, 0}x, \dot{y}=b_{0, 3}y^{3}+b_{1, 0}x, \\
&\lambda=-1, w=(3s_{2}, s_{2}, 1, 2s_{2}+1)\;with\;w_{m}=(3, 1, 1, 3);\\
&X_{1, 0, 0}:\dot{x}=a_{0, 2}y^{2}+a_{1, 0}x, \dot{y}=b_{0, 3}y^{3}+b_{1, 1}xy, \\
&\lambda=-2, w=(2\, s_{{2}}, s_{{2}}, 1, 1+2\, s_{{2}})\;with\;w_{{m}}=({2, 1, 1, 3});\\
&X_{1, 0, 1}:\dot{x}=a_{0, 2}y^{2}+a_{1, 0}x, \dot{y}=b_{1, 2}xy^{2}+b_{2, 0}x^{2}, \\
&\lambda=-3, w=({2\, s_{{2}}, s_{{2}}, 1, 3\, s_{{2}}+1})\;with\;w_{{m}}=({2, 1, 1, 4});\\
&X_{1, 0, 2}:\dot{x}=a_{0, 2}y^{2}+a_{1, 0}x, \dot{y}=b_{2, 1}x^{2}y, \\
&\lambda=-4, w=({2\, s_{{2}}, s_{{2}}, 1, 4\, s_{{2}}+1})\;with\;w_{{m}}=({2, 1, 1, 5});\\
&X_{1, 0, 3}:\dot{x}=a_{0, 2}y^{2}+a_{1, 0}x, \dot{y}=b_{3, 0}x^{3}, \\
&\lambda=-5, w=({2\, s_{{2}}, s_{{2}}, 1, 5\, s_{{2}}+1})\;with\;w_{{m}}=({2, 1, 1, 6});\\
  &{\it X}_{{0, 0, 0, 0, 1, 1}}:\dot{x}=a_{0, 3}y^{3}+a_{1, 1}xy, \dot{y}=b_{{0, 3}}{y}^{3}+b_{{1, 1}}xy, \\
 &\lambda=-1, w=({2\, d_{{1}}-2, d_{{1}}-1, d_{{1}}, 2\, d_{{1}}-1})\;with\;w_{{m}}=({2, 1, 2, 3});\\
 &{\it X}_{{0, 0, 0, 1, 1, 1}}:\dot{x}=a_{0, 3}y^{3}+a_{1, 1}xy, \dot{y}=b_{{1, 2}}x{y}^{2}+b_{{2, 0}}{x}^{2}, \\
 &\lambda=-2, w=({2\, d_{{1}}-2, d_{{1}}-1, d_{{1}}, 3\, d_{{1}}-2})\;with\;w_{{m}}=({2, 1, 2, 4});\\
 &{\it X}_{{0, 0, 0, 2, 1, 1}}: \dot{x}=a_{0, 3}y^{3}+a_{1, 1}xy, \dot{y}=b_{{2, 1}}{x}^{2}y, \\
 &\lambda=-3, w=({2\, d_{{1}}-2, d_{{1}}-1, d_{{1}}, 4\, d_{{1}}-3})\;with\;w_{{m}}=({2, 1, 2, 5});\\
 &{\it X}_{{0, 0, 0, 3, 1, 1}}: \dot{x}=a_{0, 3}y^{3}+a_{1, 1}xy, \dot{y}=b_{{3, 0}}{x}^{3}, \\
 &\lambda=-4, w=({2\, d_{{1}}-2, d_{{1}}-1, d_{{1}}, 5\, d_{{1}}-4})\;with\;w_{{m}}=({2, 1, 2, 6});\\
 &{\it X}_{{0, 2, 0, 0, 1, 1}}:\dot{x}=a_{0, 3}y^{3}+a_{1, 1}xy, \dot{y}=b_{{0, 1}}y, \\
 &\lambda=1, w=({2\, d_{{1}}-2, d_{{1}}-1, d_{{1}}, 1})\;with\;w_{{m}}=({2, 1, 2, 1});\\
 &X_{{0, 0, 0, 0, 1, 2}}:\dot{x}=a_{0, 3}y^{3}+a_{2, 0}x^2;\dot{y}=b_{{0, 3}}{y}^{3}+b_{{2, 0}}{x}^{2}, \\
&\lambda=-1, w=({d_{{1}}-1, 2(d_1-1)/3, d_{{1}}, (4d_1-1)/3})\;with\;w_{{m}}=({3, 2, 4, 5}), \\
&X_{{0, 0, 0, 1, 1, 2}}:\dot{x}=a_{0, 3}y^{3}+a_{2, 0}x^2;\dot{y}=b_{{1, 2}}x{y}^{2}, \\
&\lambda=-2, w=({d_{{1}}-1, 2(d_1-1)/3, d_{{1}}, (5d_1-2)/3}), \;with\;w_{{m}}=({3, 2, 4, 6}), \\
&X_{{0, 0, 0, 2, 1, 2}}:\dot{x}=a_{0, 3}y^{3}+a_{2, 0}x^2;\dot{y}=b_{{2, 1}}{x}^{2}y, \\
&\lambda=-3, w=({d_{{1}}-1, 2(d_1-1)/3, d_{{1}}, 2\, d_{{1}}-1})\;with\;w_{{m}}=({3, 2, 4, 7});\\
&X_{{0, 0, 0, 3, 1, 2}}:\dot{x}=a_{0, 3}y^{3}+a_{2, 0}x^2;\dot{y}=b_{{3, 0}}{x}^{3}, \\
&\lambda=-4, w=({d_{{1}}-1, 2(d_1-1)/3, d_{{1}}, (7d_1-4)/3})\;with\;w_{{m}}=({3, 2, 4, 8});\\
&X_{{0, 1, 0, 0, 1, 2}}:\dot{x}=a_{0, 3}y^{3}+a_{2, 0}x^2;\dot{y}=b_{{0, 2}}{y}^{2}, \\
&\lambda=1, w=({d_{{1}}-1, 2(d_1-1)/3, d_{{1}}, (2d_1+1)/3})\;with\;w_{{m}}=({3, 2, 4, 3});\\
&X_{{0, 2, 0, 0, 1, 2}}:\dot{x}=a_{0, 3}y^{3}+a_{2, 0}x^2;\dot{y}=b_{{0, 1}}y, \\
&\lambda=3, w=({d_{{1}}-1, 2(d_1-1)/3, d_{{1}}, 1})\;with\;w_{{m}}=({3, 2, 4, 1});\\
&X_{{0, 2, 0, 1, 1, 2}}:\dot{x}=a_{0, 3}y^{3}+a_{2, 0}x^2;\dot{y}=b_{{1, 0}}x, \\
&\lambda=2, w=({d_{{1}}-1, 2(d_1-1)/3, d_{{1}}, (d_1+2)/3})\;with\;w_{{m}}=({3, 2, 4, 2});\\
 &X_{{0, 0, 1, 1, 1, 1}}:\dot{x}=a_{1, 2}xy^{2}+a_{2, 0}x^{2}, \dot{y}=b_{{1, 2}}x{y}^{2}+b_{{2, 0}}{x}^{2}, \\
 &\lambda=-1, w=({d_{{1}}-1, (d_1-1)/2, d_{{1}},(3d_1-1)/2}), \;with\;w_{{m}}=({2, 1, 3, 4});\\
 &X_{{0, 0, 1, 2, 1, 1}}: \dot{x}=a_{1, 2}xy^{2}+a_{2, 0}x^{2}, \dot{y}=b_{{2, 1}}{x}^{2}y, \\
 &\lambda=-2, w=({d_{{1}}-1, (d_1-1)/2, d_{{1}}, 2\, d_{{1}}-1})\;with\;w_{{m}}=({2, 1, 3, 5});\\
 &X_{{0, 0, 1, 3, 1, 1}}: \dot{x}=a_{1, 2}xy^{2}+a_{2, 0}x^{2}, \dot{y}=b_{{3, 0}}{x}^{3}, \\
 &\lambda=-3, w=({d_{{1}}-1, (d_1-1)/2, d_{{1}}, (5d_1-3)/2})\;with\;w_{{m}}=({2, 1, 3, 6});\\
  &X_{{0, 1, 1, 0, 1, 1}}:\dot{x}=a_{1, 2}xy^{2}+a_{2, 0}x^{2}, \dot{y}=b_{{0, 2}}{y}^{2}+b_{{1, 0}}x, \\
 &\lambda=1, w=({d_{{1}}-1, (d_1-1)/2, d_{{1}}, (d_1+1)/2})\;with\;w_{{m}}=({2, 1, 3, 2});\\
 &X_{{0, 2, 1, 0, 1, 1}}:\dot{x}=a_{1, 2}xy^{2}+a_{2, 0}x^{2}, \dot{y}=b_{{0, 1}}y, \\
 &\lambda=2, w=({d_{{1}}-1, (d_1-1)/2, d_{{1}}, 1})\;with\;w_{{m}}=({2, 1, 3, 1});\\
 &X_{{0, 0, 2, 0, 1, 1}}:\dot{x}=a_{{2, 1}}{x}^{2}y, \dot{y}=b_{0, 3}y^{3}+b_{1, 1}xy, \\
  &\lambda=1, w=({d_{{2}}-1, (d_2-1)/2, (3d_2-1)/2, d_{{2}}})\;with\;w_{{m}}=({2, 1, 4, 3});\\
  &X_{{0, 0, 3, 0, 1, 1}}:\dot{x}=a_{{3, 0}}{x}^{3}, \dot{y}=b_{0, 3}y^{3}+b_{1, 1}xy, \\
  &\lambda=2, w=({d_{{2}}-1, (d_2-1)/2, 2\, d_{{2}}-1, d_{{2}}})\;with\;w_{{m}}=({2, 1, 5, 3});\\
   &X_{{0, 0, 2, 0, 1, 2}}:\dot{x}=a_{{2, 1}}{x}^{2}y, \dot{y}=b_{0, 3}y^{3}+b_{2, 0}x^{2}, \\
   &\lambda=1, w=({3(d_2-1)/4, (d_2-1)/2, (5d_2-1)/4, d_{{2}}})\;with\;w_{{m}}=({3, 2, 6, 5});\\
   &X_{{0, 0, 3, 0, 1, 2}}:\dot{x}=a_{{3, 0}}{x}^{3}, \dot{y}=b_{0, 3}y^{3}+b_{2, 0}x^{2}, \\
   &\lambda=2, w=({3(d_2-1)/4, (d_2-1)/2, (3d_2-1)/2, d_{{2}}})\;with\;w_{{m}}=({3, 2, 7, 5});\\
   &X_{{1, 0, 0, 0, 1, 2}}:\dot{x}=a_{{0, 2}}{y}^{2}, \dot{y}=b_{0, 3}y^{3}+b_{2, 0}x^{2}, \\
   &\lambda=-3, w=({3(d_2-1)/4, (d_2-1)/2, (d_2+3)/4, d_{{2}}})\;with\;w_{{m}}=({3, 2, 2, 5});\\
   &X_{{1, 0, 1, 0, 1, 2}}:\dot{x}=a_{{1, 1}}xy, \dot{y}=b_{0, 3}y^{3}+b_{2, 0}x^{2}, \\
   &\lambda=-2, w=({3(d_2-1)/4, (d_2-1)/2, (d_2+1)/2, d_{{2}}})\;with\;w_{{m}}=({3, 2, 3, 5});\\
   &X_{{2, 0, 1, 0, 1, 2}}:\dot{x}=a_{{1, 0}}x, \dot{y}=b_{0, 3}y^{3}+b_{2, 0}x^{2}, \\
   &\lambda=-4, w=({3(d_2-1)/4, (d_2-1)/2, 1, d_{{2}}})\;with\;w_{{m}}=({3, 2, 1, 5});\\
  &X_{{0, 0, 2, 0, 2, 1}}:\dot{x}=a_{{2, 1}}{x}^{2}y, \dot{y}=b_{0, 3}y^{3}+b_{1, 0}x, \\
  &\lambda=1, w=({3(d_2-1)/2, (d_2-1)/2, 2\, d_{{2}}-1, d_{{2}}})\;with\;w_{{m}}=({3, 1, 5, 3});\\
  &X_{{0, 0, 3, 0, 2, 1}}:\dot{x}=a_{{3, 0}}{x}^{3}, \dot{y}=b_{0, 3}y^{3}+b_{1, 0}x, \\
  &\lambda=2, w=({3(d_2-1)/2, (d_2-1)/2, 3\, d_{{2}}-2, d_{{2}}})\;with\;w_{{m}}=({3, 1, 7, 3});\\
  &X_{{1, 0, 1, 0, 2, 1}}:\dot{x}=a_{{1, 1}}xy, \dot{y}=b_{0, 3}y^{3}+b_{1, 0}x, \\
  &\lambda=-1/2, w=({3(d_2-1)/2, (d_2-1)/2, (d_2+1)/2, d_{{2}}})\;with\;w_{{m}}=({3, 1, 2, 3});\\
  &X_{{1, 0, 2, 0, 2, 1}}:\dot{x}=a_{{2, 0}}{x}^{2}, \dot{y}=b_{0, 3}y^{3}+b_{1, 0}x, \\
  &\lambda=1/2, w=({3(d_2-1)/2, (d_2-1)/2, (3d_2-1)/2, d_{{2}}})\;with\;w_{{m}}=({3, 1, 4, 3});\\
  &X_{{0, 0, 3, 1, 1, 1}}:\dot{x}=a_{{3, 0}}{x}^{3}, \dot{y}=b_{1, 2}xy^{2}+b_{2, 0}x^{2}, \\
  &\lambda=1, w=({2(d_2-1)/3, (d_2-1)/3, (4d_2-1)/3, d_{{2}}})\;with\;w_{{m}}=({2, 1, 5, 4});\\
   &X_{{0, 1, 2, 0, 1, 1}}:\dot{x}=a_{{2, 1}}{x}^{2}y, \dot{y}=b_{0, 2}y^{2}+b_{1, 0}x, \\
  &\lambda=2, w=({2\, d_{{2}}-2, d_{{2}}-1, 3\, d_{{2}}-2, d_{{2}}})\;with\;w_{{m}}=({2, 1, 4, 2});\\
  &X_{{0, 1, 3, 0, 1, 1}}:\dot{x}=a_{{3, 0}}{x}^{3}, \dot{y}=b_{0, 2}y^{2}+b_{1, 0}x, \\
  &\lambda=3, w=({2\, d_{{2}}-2, d_{{2}}-1, 4\, d_{{2}}-3, d_{{2}}})\;with\;w_{{m}}=({2, 1, 5, 2}).
\end{align*}}
\end{proposition}

\section{Center of  PSQHPDS  of degree $2$ and $3$}\label{se-6}

 In this section we study the existence and the canonical form of center for    planar semi-quasi homogeneous but non-semihomogeneous polynomial vector fields of degree $2$ and $3$. The main result of this section is the following theorem.

\begin{theorem}\label{thm-21}
  Suppose that system \eqref{1} is a   semi-quasi homogeneous but non-semihomogeneous   polynomial differential system of degree $n$.

  $(i)$ If $n=2$, then system \eqref{1} has not center.

  $(ii)$ If $n=3$, then  system \eqref{1} has a center if and only if   after a linear transformation and a rescaling of time, \eqref{1}  can be written as
 \begin{equation}\label{eq-cubcen}\dot{x}=-y^3+x^2, \dot{y}=x. \end{equation} Moreover, system \eqref{eq-cubcen}   has the first integral
 $$H(x,y)=\left(y^3+\frac{3}{2}y^2+\frac{3}{2}y-x^2+\frac{3}{4}\right)e^{-2y},$$
 with integrating factor $M(y)=-2e^{-2y}$, and the  period annulus corresponding to  $H(x,y)=h\in(0,3/4)$. The global phase portrait  of system \eqref{eq-cubcen} is shown in Figure  \ref{fig.1}.

 \begin{figure}[H]
  \begin{center}
   \includegraphics[width=3.8cm,height=3.6cm]{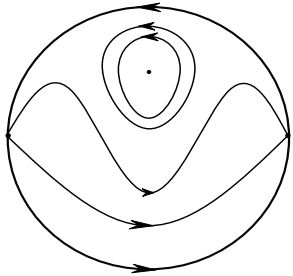}
  \end{center}
  \caption{The global phase portraits of system \eqref{eq-cubcen}}\label{fig.1}
\end{figure}

\end{theorem}

The proof of Theorem \ref{thm-21} rely on the canonical forms of  {\itshape PSQHPDS} \eqref{1}. Thus we need the following two propositions.

\begin{proposition}\label{pr-18}
If system \eqref{1} is a semi-quasi homogeneous but non-semihomogeneous quadratic coprime polynomial differential system  with the minimal weight vector $w_{m}$, then after a linear transformation and a rescaling of independent variables, it can be written as one of the following systems
\begin{itemize}
  \item[(\;$A_{1}$\;)] $\dot{x}=x^{2}$, $\dot{y}=y^{2}+x$, with $w_{m} = (2, 1, 3, 2)$.
  \item[(\;$A_{2}$\;)] $\dot{x}=y^{2}+x$, $\dot{y}=ay^{2}+x$, with $a(a-1)\neq0$, or $\dot{x}=a_{1}y^{2}+a_{2}x$, $\dot{y}=a_{3}y^{2}+a_{4}x$ with $a_{i}\in\{0,1\}$ and $a_{1}+a_{2}+a_{3}+a_{4}=3$, and both with $w_{m}=(2,1,1,2)$.
\end{itemize}
\end{proposition}
\begin{proof}
By Proposition \ref{pr-15}, one can note that in system $X_{0,0,0}$ we have that $a_{0,2}b_{1,0}\neq a_{1,0}b_{0,2}$, otherwise its polynomials $P$ and $Q$ are not
coprime. If $a_{0,2}b_{1,0}a_{1,0}b_{0,2}\neq0$,  then the rescaling of   variables $(X, Y, T ) = \big((a_{0,2}b_{1,0}^{2}/a_{1,0}^{3})x,(a_{0,2}b_{1,0}/a_{1,0}^{2})y,a_{1,0}t\big)$ writes system
$X_{0,0,0}$ into the system $\frac{dX}{dT}=Y^{2}+X$, $\frac{dY}{dT}=aY^{2}+X$, with $a(a-1)\neq0$. If $a_{0,2}b_{1,0}a_{1,0}b_{0,2}=0$, then $a_{0,2}b_{1,0}\neq0$, $a_{1,0}b_{0,2}=0$ $(a_{1,0}\neq b_{0,2})$ or
$a_{0,2}b_{1,0}=0$ $(a_{0,2}\neq b_{1,0})$, $a_{1,0}b_{0,2}\neq0$. If $a_{0,2}b_{1,0}\neq0$, $a_{1,0}b_{0,2}=0$,  then our proof can be split  into three cases:\\
$Case\;i.$ $a_{1,0}\neq0$, $b_{0,2}=0$.

The rescaling of   variables $(X, Y, T ) =\big((a_{0,2}b_{1,0}^{2}/a_{1,0}^{3})x,(a_{0,2}b_{1,0}/a_{1,0}^{2})y,a_{1,0}t\big)$ writes system $X_{0,0,0}$ into   $\frac{dX}{dT}=Y^{2}+X$,
$\frac{dY}{dT}=X$.\\
$Case\;ii.$ $a_{1,0}=0$, $b_{0,2}\neq0$.

Then by the transformation $(X, Y, T ) =\big((b_{0,2}^{3}/a_{0,2}^{2}b_{1,0})x,(b_{0,2}^{2}/a_{0,2}b_{1,0})y,(a_{0,2}b_{1,0}/b_{0,2})t\big)$, we can writes system $X_{0,0,0}$ into   $\frac{dX}{dT}=Y^{2}$,
$\frac{dY}{dT}=Y^{2}+X$.

Similarly, if $a_{0,2}b_{1,0}=0$, $a_{1,0}b_{0,2}\neq0$, our discussion should  be divided into three cases:\\
$Case\;i.$ $a_{0,2}\neq0$, $b_{1,0}=0$.

The  rescaling of   variables $(X, Y, T ) =\big((b_{0,2}^{2}/a_{1,0}a_{0,2})x,(b_{0,2}/a_{1,0})y,a_{1,0}t\big)$ writes system $X_{0,0,0}$ into $\frac{dX}{dT}=Y^{2}+X$,
$\frac{dY}{dT}=Y^{2}$.\\
$Case\;ii.$ $a_{0,2}=0$, $b_{1,0}\neq0$.

The  rescaling of   variables $(X, Y, T ) =\big((b_{0,2}b_{1,0}/a_{1,0}^{2})x,(b_{0,2}/a_{1,0})y,a_{1,0}t\big)$ writes system $X_{0,0,0}$ into   $\frac{dX}{dT}=X$,
$\frac{dY}{dT}=Y^{2}+X$.

For system $X_{0,0,2,0,1,1}$ we have that $a_{2,0}b_{0,2}b_{1,0}\neq0$. Then by the transformation $(X, Y, T ) =\big((a_{0,2}^{2}/b_{0,2}b_{1,0})x,(a_{2,0}/b_{1,0})y,(b_{0,2}b_{1,0}/a_{2,0})t\big)$, system $X_{0,0,2,0,1,1}$ is changed
to the system $\frac{dX}{dT}=X^{2}$, $\frac{dY}{dT}=Y^{2}+X$.

The proof is finished.
\end{proof}

\begin{proposition}\label{pr-19}
If system \eqref{1} is a semi-quasi homogeneous but non-semihomogeneous cubic coprime polynomial differential system  with the minimal weight vector $w_{m}$, then after a linear transformation and a rescaling of time,  it can be written as one of the following systems
\begin{itemize}
\item[(\;$A_{1}$\;)] $\dot{x}=y^{2}+x$, $\dot{y}=x^{3}$, with $w_{m}=(2,1,1,6)$.
\item[($B_{1,k}$)]$\dot{x}=y^{3}+x^{2}$, $\dot{y}=Q(x,y)$, with $w_{m}=(3,2,4,2^{k}-1)$, $Q(x,y)\in\{x^{2}y,y^{2},y\}$, $deg(Q(x,y))=k$, $k=1,2,3$.
\item[$(C_{1,k}$)]$\dot{x}=x^{1+k}y^{2-k}$, $\dot{y}=y^{3}+x^{2}$, with $w_{m}=(3,2,5+k,5)$, $k=1,2$.
\item[$(\;D_{1}\;)$]$\dot{x}=y^{2}$, $\dot{y}=y^{3}+x^{2}$, with $w_{m}=(3,2,2,5)$.
\item[$(E_{1,k})$]$\dot{x}=x^{k}y^{2-k}$, $\dot{y}=y^{3}+x$, with $w_{m}=(3,1,2^{k},3)$, $k=1,2$.
\item[$(\;F_{1}\;)$]$\dot{x}=x^{3}$, $\dot{y}=y^{2}+x$, with $w_{m}=(2,1,5,2)$.
\item[$(G_{1,k}^{\pm})$]$\dot{x}=\pm y^{k+1}+x$, $\dot{y}=(\pm1)^{k} x^{3-k}y^{k}$, with $w_{m}=(1+k,k,3k-2,4+k)$, $k=1,2$.
\item[$(H_{1,k}^{\pm})$]$\dot{x}=\pm x^{k-1}y^{4-k}+x^{2}$, $\dot{y}=x^{2-k}y^{k-1}$, with $w_{m}=(4-k,3-k,5-k,3-k)$, $k=1,2$.
\item[$(I_{1,k}^{\pm})$]$\dot{x}=x^{4-k}y^{k-1}$, $\dot{y}=y^{2-k}((\pm 1)^{k}y^{2}+(\pm 1)^{k+1}x)$, with $w_{m}=(2,1,6-k,4-k)$, $k=1,2$.
\item[$(J_{1,k}^{\pm})$]$\dot{x}=xy^{2-k}$, $\dot{y}=(\pm 1)^{k-1}y^{3}\pm x^{2}$, with $w_{m}=(3,2,5-2k,5)$, $k=1,2$.
\item[$(L_{1,k}^{\pm})$]$\dot{x}=(\pm x)^{k+1}y^{2-k}$, $\dot{y}=\pm y^{3}+x$, with $w_{m}=(3,1,2k+3,3)$, $k=1,2$.
\item[$(M_{1,k}^{\pm})$]$\dot{x}=\pm y^{3}+x^{k}y^{2-k}$, $\dot{y}=\pm x^{3}$, with $w_{m}=(k+1,k,2k,2k+4)$, $k=1,2$.
\item[(\;$A_{2}$\;)] $\dot{x}=ay^{3}+x$, $\dot{y}=\pm y^{3}+x$, with $a(a\mp1)\neq0$, or $\dot{x}=\pm y^{3}+a_{1}x$, $\dot{y}=a_{2}y^{3}+a_{3}x$, or $\dot{x}=a_{1}y^{3}+a_{2}x$, $\dot{y}=\pm y^{3}+a_{3}x$, and both with $w_{m}=(3,1,1,3)$, where $a_{i}\in\{0,1\}$ and $a_{1}+a_{2}+a_{3}=2$.
\item[(\;$B_{2}$\;)] $\dot{x}=ay^{2}+x$, $\dot{y}=\pm y^{3}+xy$, with $a(a\mp1)\neq0$, or $\dot{x}=a_{1}y^{2}+a_{2}x$, $\dot{y}=\pm y^{3}+a_{3}$, with $a_{i}\in\{0,1\}$ and $a_{1}+a_{2}+a_{3}=2$, and both with $w_{m}=(2,1,1,3)$.
\item[(\;$C_{2}$\;)] $\dot{x}=y^{2}+x$, $\dot{y}=xy^{2}+ax^{2}$, with $a\neq1$, or $\dot{x}=y^{2}$, $\dot{y}=xy^{2}\pm x^{2}$, and both with $w_{m}=(2,1,1,4)$.
\item[(\;$D_{2}$\;)]$\dot{x}=ay^{3}+xy$, $\dot{y}=xy^{2}\pm x^{2}$, with $a(a\mp1)\neq0$, or $\dot{x}=\pm y^{3}+a_{1}xy$, $\dot{y}=a_{2}xy^{2}\pm x^{2}$, with $a_{i}\in\{0,1\}$ and $a_{1}+a_{2}=1$, and both with $w_{m}=(2,1,2,4)$.
\item[(\;$E_{2}$\;)] $\dot{x}=ay^{3}+x^{2}$, $\dot{y}=y^{3}+x^{2}$, with $a(a-1)\neq0$, or $\dot{x}=a_{1}y^{3}+a_{2}x^{2}$, $\dot{y}=a_{3}y^{3}+a_{4}x^{2}$, with $a_{i}\in\{0,1\}$ and $a_{1}+a_{2}+a_{3}+a_{4}=3$, and both with $w_{m}=(3,2,4,5)$.
\item[(\;$F_{2}$\;)] $\dot{x}=axy^{2}+x^{2}$, $\dot{y}=y^{2}+x$, with $a(a-1)\neq0$, or $\dot{x}=xy^{2}+a_{1}x^{2}$, $\dot{y}=y^{2}+a_{2}x$, with $a_{i}\in\{0,1\}$ and $a_{1}+a_{2}=1$, and both with $w_{m}=(2,1,3,2)$.
\end{itemize}
\end{proposition}

The proof is similar to the proof of Proposition \ref{pr-18}. For the sake of brevity, we will put the detailed  proof in the Appendix of this paper.

\begin{proof}[Proof of Theorem \ref{thm-21}]
By Lemma 11 of \cite{27}, we know that for a semi-quasi homogeneous coprime polynomial vector field $X=(P,Q)$ the origin is the unique (real and finite) singular point of $X$.

$(i)$  $n=2$.  According to \cite{52} (see Lemma 8.14 of Chapter 8),  if a quadratic vector field $X$ has a center, then up to a translation, a linear transformation and a time rescaling, it can be written as
  \begin{align*}
    \dot{x}=-y-bx^2-Cxy-dy^{2},   \  \dot{y}=x+ax^2+Axy-ay^{2}.
  \end{align*}
  Therefore, from   Proposition \ref{pr-18}, it is not hard to see that any {\itshape PSQHPDS} with degree $n=2$ has not center.

$(ii)$  $n=3$.  In this case we will use   Proposition \ref{pr-19} to prove the result.  Consider the cubic vector fields appearing in Proposition \ref{pr-19}.   It is obviously that    vector fields $(B_{1,k})$, $(B_{2})$, $(C_{1,k})$, $(D_{0,k})$, $(E_{1,k})$, $(F_1)$, $(F_2)$ $(G_{1,k}^\pm)$, $(H_{1,2}^\pm)$, $(I_{1,k}^\pm)$, $(L_{1,k}^\pm)$   have the invariant line on the coordinate axis, and the first component of vector field $(D_{1})$
is always  nonnegative which implies that $(D_{1})$ can not has periodic orbit,
thus it suffice to study the
remaining  vector fields:  $(A_1)$,  $(H_{1,1}^\pm)$, $(M_{1,k}^\pm)$, $(A_2)$, $(C_2)$, $(D_{2})$ and $(E_2)$.

To be clear, we will divide the discussions into several cases.
\begin{enumerate}[(1)]
  \item Systems   $(H_{1,1}^{+})$, $(M_{1,k}^{\pm})$, $(A_2)$, $(C_2)$, $(D_{2})$ and $(E_2)$.

It is already known that,   if  we blow up the semi-quasi homogeneous coprime polynomial differential system \eqref{1} at the origin, which is a Lojasiewicz singularity,
 and
we find that some  singularities  of the corresponding blow-up system is elementary, then the origin of system \eqref{1} cannot be a center
 (see Chapter 3 of  \cite{52}).

Taking $(x,y)=(\bar{x}^2,\bar{x}\bar{y})$ , we blow-up   system $(H_{1,1}^{+})$: $\dot{x}=y^3+x^2$, $\dot{y}=x$ in the positive $x$-direction and get
\begin{equation*}\label{137}
  \dot{\bar{x}}=\frac{1}{2}\bar{x}(\bar{y}^3+\bar{x}),\
\dot{\bar{y}}=1-\frac{1}{2}\bar{y}(\bar{y}^3+\bar{x}).
\end{equation*}

On $\{\bar{x}=0\}$ we find two hyperbolic singularities, situated, respectively, at $\bar{y}=\sqrt[4]{2}$ and $\bar{y}=-\sqrt[4]{2}$.

If we consider the blow-up transformation $(x,y)=(\bar{x}^2,\bar{x}^3\bar{y})$ in the positive $x$-direction, then the subsystems of $(M_{1,1}^{\pm})$: $\dot{x}=\pm y^3+xy$, $\dot{y}=x^3$ becomes
\begin{equation}\label{138}
\dot{\bar{x}}=\frac{1}{2}\bar{x}\bar{y}(1\pm \bar{x}^4\bar{y}^2),\
  \dot{\bar{y}}=1-\frac{3}{2}(1\pm \bar{x}^4\bar{y}^2).
\end{equation}
We obtain two hyperbolic singularities, located at the points $(\bar{x},\bar{y})=(0,\pm\sqrt{2/3})$.

Blowing up the subsystems of $(M_{1,1}^{\pm})$: $\dot{x}=\pm y^3+xy$, $\dot{y}=-x^3$ in the negative $x$-direction $(x=-\bar{x}^2,y=\bar{x}^3\bar{y})$ gives
\begin{equation}\label{139}
\dot{\bar{x}}=-\frac{1}{2}\bar{x}\bar{y}(\pm \bar{x}^4\bar{y}^2-1),\
\dot{\bar{y}}=1+\frac{3}{2}\bar{y}^2(\pm \bar{x}^4\bar{y}^2-1).
\end{equation}
By computing the  singularities of \eqref{139} on $\{\bar{x}=0\}$ we obtain $(\bar{x},\bar{y})=(0,\pm \sqrt{2/3})$ which are hyperbolic.

We blow up   systems $(M_{1,2}^{\pm})$: $\dot{x}=x^2\pm y^3$, $\dot{y}=\pm x^3$ in the positive $x$-direction $(x,y)=(\bar{x}^3,\bar{x}^2\bar{y})$,  and get the blow-up system
\begin{equation}\label{140}
\dot{\bar{x}}=\frac{1}{3}\bar{x}(1\pm\bar{y}^3),\
\dot{\bar{y}}=\pm\bar{x}^4-\frac{2}{3}(1\pm\bar{y}^3).
\end{equation}
On $\{\bar{x}=0\}$ system \eqref{140} has three singularities $(0,0)$ and $(0,\mp1)$. Moreover, the singularity $(0,0)$ is hyperbolic and singularities $(0,\mp1)$ are semi-hyperbolic.

Blowing up the subsystems of $(A_2)$:  $\dot{x}=y^3$, $\dot{y}=y^3+x$ in the positive $x$-direction $(x=\bar{x}^2,y=\bar{x}\bar{y})$ gives
\begin{equation*}\label{141}
\dot{\bar{x}}=\frac{1}{2}\bar{x}\bar{y}^3,\
 \dot{\bar{y}}=1+\bar{x}\bar{y}^3-\frac{1}{2}\bar{y}^4.
\end{equation*}

On $\{\bar{x}=0\}$ we find two hyperbolic singularities, situated, respectively, at $\bar{y}=\sqrt[4]{2}$ and $\bar{y}=-\sqrt[4]{2}$.

Let's consider the subsystems of  $(C_2)$: $\dot{x}=y^2$, $\dot{y}=xy^2\pm x^2$.
Obviously, this system can not have periodic orbit (and hence hasn't center) because $\dot{x}=y^2\geq 0$.

Consider the subsystems of $(D_{2})$: $\dot{x}=ay^3+xy$, $\dot{y}=xy^2+x^2$ with $a(a-1)\neq0$. By substituting blow-up transformation $(x,y)=(\bar{x},\bar{x}\bar{y})$ in the positive $x$-direction, we get
\begin{equation*}
 \dot{\bar{x}}=\bar{x}\bar{y}(\bar{x}\bar{y}^2+1),\ \dot{\bar{y}}=\bar{x}\bar{y}^2-\bar{y}^2(a\bar{x}\bar{y}^2+1)+1.
\end{equation*}
It is easy to check that  the  singularities $(\bar{x}, \bar{y})=(0, \pm 1)$ are hyperbolic.

 Let us deal with the subsystems of $(D_{2})$: $\dot{x}=ay^3+xy$, $\dot{y}=xy^2-x^2$ with $a(a+1)\neq0$. Consider the blow-up $(x,y)=(\bar{x}\bar{y},-\bar{y})$ in the negative $y$-direction.  The  corresponding blow-up system is
 \begin{equation}\label{142}
 \dot{\bar{x}}=\bar{x}^2\bar{y}-\bar{x}^3-a\bar{y}-\bar{x},\
   \dot{\bar{y}}=\bar{x}\bar{y}(\bar{x}-\bar{y}).
 \end{equation}
The unique  singularity on $\bar{y}=0$ occurs for $\bar{x}=0$. It is clear  that the coefficient matrix of the  linear part of   vector field \eqref{142} at $(\bar{x},\bar{y})=(0,0)$ is
\begin{align*}
  \begin{pmatrix}
    -1&-a\\
    0&0
  \end{pmatrix}.
\end{align*}
Then the singularity of system \eqref{142} is semi-hyperbolic.

Let us study the subsystems of $(E_2)$: $\dot{x}=ay^3+x^2$, $\dot{y}=y^3+x^2$ with $a(a-1)\neq0$. By substituting  blow-up $(x,y)=(\bar{x}^3,\bar{x}^2\bar{y})$ in the positive $x$-direction into $(E_2)$, we have
\begin{equation}\label{143}
  \dot{\bar{x}}=\frac{1}{3}\bar{x}(a\bar{y}^3+1),\
 \dot{\bar{y}}=\bar{x}(\bar{y}^3+1)-\frac{2}{3}\bar{y}(a\bar{y}^3+1).
\end{equation}
It is easy to calculate that $(0,0)$, $(0,-\sqrt[3]{1/a})$ are singularities of \eqref{143} on the line $\{\bar{x}=0\}$. By direct computation their linear part of the vector field at these  singularities, we find that  $(0,0)$ and $(0,-\sqrt[3]{1/a})$ are respectively hyperbolic and semi-hyperbolic.

Next consider the subsystems of $(E_2)$
\begin{equation*}\label{144}
  \dot{x}=y^3+x^2,\
    \dot{y}=x^2.
\end{equation*}
This system can not have periodic orbit (and hence hasn't center) because $\dot{y} \geq 0$.

Finally let us study  the subsystems of $(E_2)$:  $\dot{x}=y^3$, $\dot{y}=y^3+x^2$. Taking  blow up $(x,y)=(\bar{x}^4,\bar{x}^3\bar{y})$ in the positive $x$-direction, we obtain
\begin{equation*}
  \dot{\bar{x}}=\frac{1}{4}\bar{x}\bar{y}^3,\
 \dot{\bar{y}}=\bar{x}\bar{y}^3-\frac{3}{4}\bar{y}^4+1.
\end{equation*}
On $\{\bar{x}=0\}$ we find two hyperbolic singularities, situated, respectively, at $\bar{y}=\sqrt[4]{4/3}$ and $\bar{y}=-\sqrt[4]{4/3}$.

  \item Systems $(A_1)$, $(C_2)$ and $(A_2)$.

  In system $(A_1)$: $\dot{x}=y^2+x$, $\dot{y}=x^3$, using  Bendixson's Criteria (the seventh chapter of \cite{52}), we obtain $\partial(y^2+x)/\partial x+\partial x^3/\partial y=1$. Then system $(A_1)$ has no periodic orbit, that is, the origin is not a center.

For  the subsystems of $(C_2)$: $\dot{x}=y^2+x$, $\dot{y}=xy^2+ax^2$ $(a\neq1)$, applying Bendixsons Criteria, we have $\partial(y^2+x)/\partial x+\partial(xy^2+ax^2)/\partial y=1+2xy$. Then there exists a neighborhood $U$ of the origin  consisting not  periodic orbits.

Again, by Bendixsons Criteria,  neither the subsystems of $(A_2)$: $\dot{x}=ay^3+x$, $\dot{y}=\pm y^3+x$ ($a(a\pm 1)\neq0$)   nor $\dot{x}=\pm y^3+x$, $\dot{y}=x$ contain periodic orbits in    a small neighborhood   of the origin. Thus, the  origin can not be  a center.

  \item Systems $(A_2)$.

The subsystems of $(A_2)$: $\dot{x}=-y^3$, $\dot{y}=y^3+x$ has $V(x,y)=2x^2+y^4$ as a Lyapunov function on $\mathbb{R}^2 \setminus \{(0,0)\}$ and the set $\{(x,y)\mid\dot{V}=4y^6=0\}=\{y=0\}$ does not contain a full trajectory. Hence, the origin is asymptotically stable.

  \item Systems  $(H_{1,1}^{-})$.

 The system $(H_{1,1}^{-})$: $\dot{x}=-y^3+x^2$, $\dot{y}=x$  has the first integral
\begin{align*}
    H(x,y)=(y^3+\frac{3}{2}y^2+\frac{3}{2}y-x^2+\frac{3}{4})e^{-2y},
\end{align*}
with integrating factor $M(y)=-2e^{-2y}$. Since $(0,0)$ is an extreme  point for $H(x,y)$,  the origin of system $(H_{1,1}^{-})$ is a center. In the period annulus, the Hamiltonian function takes values between $h_s=0$ and $h_c=3/4$, where the Hamiltonian values $h_s$ and $h_c$ correspond to the period annulus terminates at the separatrix polycycle and the center at $(0,0)$ respectively. Hence, in the period annulus, we obtain $H(x,y)=h\in(0,3/4)$.

The proof is completed.
\end{enumerate}
\end{proof}

\section*{Appendix: Proof of Proposition \ref{pr-19}}\label{app-1}
In this Appendix, we will give the   proof of Proposition \ref{pr-19}.

\begin{proof}[Proof of Proposition \ref{pr-19}]
For the sake of clarity, we will split the discussions into several cases.
\begin{enumerate}[(1)]
\item Systems $X_{1,0,0,3}$, $X_{0,1,0,2}$, $X_{0,2,0,1}$, $X_{0,1,3,0}$, $X_{0,2,3,0}$, $X_{1,0,2,0}$ and $X_{2,0,1,0}$.

Let's first consider $X_{1,0,0,3}$. Then by the transformation
$$(X, Y, T ) =\big(x,a_{0,2}^{1/3}b_{3,0}^{-1/3}y,a_{0,2}^{1/3}b_{3,0}^{2/3}t\big),$$
we can writes system $X_{1,0,0,3}$ into the system $\frac{dX}{dT}=Y^{2}$,
$\frac{dY}{dT}=X^{3}$.  Thus  we get  the canonical form $(A_{0})$ of system $X_{1,0,0,3}$ in Proposition \ref{pr-19}.

Similarly, we can get the canonical form $(B_{0,1})$, $(B_{0,2})$, $(C_{0,1})$, $(C_{0,2})$, $(D_{0,1})$ and $(D_{0,2})$ from
$X_{0,1,0,2}$, $X_{0,2,0,1}$, $X_{0,1,3,0}$, $X_{0,2,3,0}$, $X_{1,0,2,0}$ and $X_{2,0,1,0}$ respectively. For the sake of brevity we omit the detail discussion.
\item Systems $X_{1,0,3}$, $X_{0,0,0,2,1,2}$, $X_{0,1,0,0,1,2}$, $X_{0,2,0,0,1,2}$, $X_{0,0,2,0,1,2}$, $X_{0,0,3,0,1,2}$, $X_{1,0,0,0,1,2}$, $X_{1,0,1,0,2,1}$,
$X_{1,0,2,0,2,1}$ and $X_{0,1,3,0,1,1}$.

 For system $X_{0,0,0,2,1,2}$, we have $a_{0,3}a_{2,0}b_{2,1}\neq0$. Taking $(X,Y,T)=\big(b_{2,1}a_{2,0}^{-1}x,a_{2,0}^{-1}\\b_{2,0}^{2/3}a_{0,3}^{1/3}y, a_{2,0}^{2}b_{2,1}^{-1}t\big),$
system $X_{0,0,0,2,1,2}$ is transformed into $\frac{dX}{dT}=Y^{3}+X^{2}$, $\frac{dY}{dT}=X^{2}Y$. Therefore, we get the canonical form $(B_{1,3})$ of system $X_{0,0,0,2,1,2}$.

 Similarly, we can obtain the canonical forms of $X_{1,0,3}$, $X_{0,1,0,0,1,2}$, $X_{0,2,0,0,1,2}$, $X_{0,0,2,0,1,2}$, $X_{0,0,3,0,1,2}$, $X_{1,0,0,0,1,2}$, $X_{1,0,1,0,2,1}$,
$X_{1,0,2,0,2,1}$ and $X_{0,1,3,0,1,1}$, which are respectively $(A_{1})$, $(B_{1,2})$, $(B_{1,1})$, $(C_{1,1})$, $(C_{1,2}$), $(D_{1})$, $(E_{1,1})$, $(E_{1,2})$ and $(F_{1})$.
\item Systems $X_{1,0,2}$, $X_{0,0,0,1,1,2}$, $X_{0,2,0,1,1,2}$, $X_{0,2,1,0,1,1}$, $X_{0,0,3,0,1,1}$, $X_{1,0,1,0,1,2}$, $X_{2,0,1,0,1,2}$, $X_{0,0,2,0,2,1}$, $X_{0,0,3,0,2,1}$ and $X_{0,1,2,0,1,1}$.

Since in system $X_{1,0,2}$ we have that $a_{0,2}a_{1,0}b_{2,1}\neq0$. Then by the transformation $(X, Y, T )=\big(|b_{2,1}a_{0,1}^{-1}|^{1/2}x,|a_{0,2}|^{1/2}|b_{2,1}a_{0,1}^{-3}|^{1/4}y,a_{0,1}t\big)$, we can
writes system $X_{1,0,2}$ into the system $\frac{dX}{dT}=\pm Y^{2}+X$, $\frac{dY}{dT}=\pm X^{2}Y$. This yields the canonical form
$(G_{1,1}^{\pm})$ of Proposition \ref{pr-19}.

Next consider vector field  $X_{0,0,0,1,1,2}$. Taking transformation $(X,Y,T)=\big(|a_{0,3}^{-1}b_{1,2}^{3}|^{1/2}\\a_{2,0}^{-1}x,b_{1,2}a_{2,0}^{-1}y,|a_{0,3}b_{1,2}^{-3}|^{1/2}t\big)$, $X_{0,0,0,1,1,2}$ becomes $\frac{dX}{dT}=\pm Y^{3}+X$, $\frac{dY}{dT}=XY^{2}$, with $w_{m}=(3,2,4,6)$. This yields the canonical form
$(G_{1,2}^{\pm})$ of Proposition \ref{pr-19}.
\par In the analogous way, we can obtain the canonical forms of systems $X_{0,2,0,1,1,2}$, $X_{0,2,1,0,1,1}$, $X_{0,0,3,0,1,1}$, $X_{0,1,2,0,1,1}$, $X_{1,0,1,0,1,2}$, $X_{2,0,1,0,1,2}$, $X_{0,0,2,0,2,1}$, $X_{0,0,3,0,2,1}$ ,$X_{0,0,0,3,1,1}$ and $X_{0,0,0,3,1,2}$,
which are respectively $(H_{1,1}^{\pm})$, $(H_{1,2}^{\pm})$, $(I_{1,1}^{\pm})$, $(I_{1,2}^{\pm})$, $(J_{1,1}^{\pm})$, $(J_{1,2}^{\pm})$, $(L_{1,1}^{\pm})$,
$(L_{1,2}^{\pm})$, $(M_{1,1}^{\pm})$ and $(M_{1,2}^{\pm})$.

\item Systems $X_{0,0,0}$, $X_{1,0,0}$, $X_{1,0,1}$, $X_{0,0,0,1,1,1}$, $X_{0,0,0,0,1,2}$ and $X_{0,1,1,0,1,1}$.
\par By Proposition \ref{pr-16}, one can observe  that in system $X_{0,0,0}$ we have that $a_{0,3}b_{1,0}\neq a_{1,0}b_{0,3}$, because  otherwise its polynomials $P$ and $Q$ are not coprime.
If $a_{0,3}b_{1,0}a_{1,0}b_{0,3}\neq 0$, then the rescaling of   variables $(X,Y,T)=(b_{1,0}|b_{0,3}a_{1,0}^{-1}|^{1/2}a_{1,0}^{-1}x,|b_{0,3}a_{1,0}^{-1}|^{1/2}y,a_{1,0}t)$
writes system $X_{0,0,0}$ into   $\frac{dX}{dT}=aY^{3}+X$, $\frac{dY}{dT}=\pm Y^{3}+X$, with $a(a\mp1)\neq0$. If $a_{0,3}b_{1,0}a_{1,0}b_{0,3}= 0$, then
$a_{1,0}b_{0,3}=0$ $(a_{1,0}\neq b_{0,3})$, $a_{0,3}b_{1,0}\neq 0$ or $a_{1,0}b_{0,3}\neq0$, $a_{0,3}b_{1,0}=0$ $(a_{0,3}\neq b_{1,0})$. If $a_{1,0}b_{0,3}=0$ $(a_{1,0}\neq b_{0,3})$, $a_{0,3}b_{1,0}\neq 0$,
then our proof is divided into three cases:\\
$Case\;i.$ $a_{1,0}=0$, $b_{0,3}\neq 0$.

The  rescaling of   variables $(X, Y, T ) =\big(b_{0,3}^{2}b_{0,1}|a_{0,3}b_{1,0}|^{-3/2}x,|a_{0,3}b_{1,0}|^{-1}b_{0,3}y, |a_{0,3} \\ \cdot b_{1,0}|b_{0,3}^{-1}t\big)$ writes system $X_{0,0,0}$ into the system $\frac{dX}{dT}=\pm Y^{3}$,
  $\frac{dY}{dT}=Y^{3}+X$.\\
$Case\;ii.$ $a_{1,0}\neq0$, $b_{0,3}=0$.

The  rescaling of   variables $(X, Y, T ) =\big(a_{1,0}^{-1}b_{1,0}|a_{1,0}|^{-1}x|a_{0,3}b_{1,0}|^{1/2},|a_{0,3}b_{1,0}|^{1/2}\\ \cdot |a_{1,0}|^{-1}y,a_{1,0}t\big)$ writes system $X_{0,0,0}$ into   $\frac{dX}{dT}=\pm Y^{3}+X$,
  $\frac{dY}{dT}=X$.
\par Similarly, if $a_{1,0}b_{0,3}\neq 0$, $a_{0,3}b_{1,0}=0$, our discussion can be  divided into three cases:\\
$Case\;i.$ $a_{0,3}\neq0$, $b_{1,0}=0$.

Then  the transformation $(X, Y, T )=\big(a_{1,0}a_{0,3}^{-1}|b_{0,3}a_{1,0}^{-1}|^{3/2}x,|b_{0,3}a_{1,0}^{-1}|^{1/2}y,a_{1,0}t\big)$ writes system $X_{0,0,0}$ into  $\frac{dX}{dT}=Y^{3}+X$,
$\frac{dY}{dT}=\pm Y^{3}$.\\
$Case\;ii.$ $a_{0,3}=0$, $b_{1,0}\neq0$.

Then the rescaling of   variables $(X, Y, T ) =\big(b_{1,0}a_{1,0}^{-1}|b_{0,3}a_{1,0}^{-1}|^{1/2}x,|b_{0,3}a_{1,0}^{-1}|^{1/2}y, a_{1,0}t\big)$ writes system $X_{0,0,0}$ into  $\frac{dX}{dT}=X$,
  $\frac{dY}{dT}=\pm Y^{3}+X$. This prove the $(A_{2})$ of Proposition \ref{pr-19}.

Analogously, we get system $(B_{2})$ from $X_{1,0,0}$, system $(C_{2})$ from $X_{1,0,1}$, system $(D_{2})$ from $X_{0,0,0,1,1,1}$, system $(E_{2})$ from $X_{0,0,0,0,1,2}$, system $(F_{2})$ from $X_{0,1,1,0,1,1}$.

Thus  the proof of Proposition \ref{pr-19} is completed.
\end{enumerate}
\end{proof}

\bibliographystyle{plain}
\bibliography{ref}
\medskip
Received xxxx 2018; revised xxxx 20xx.
\medskip
\end{document}